\documentclass[final,12pt]{article}
\usepackage[utf8]{inputenc}
\usepackage{anysize} 
\marginsize{2cm}{2cm}{2cm}{2cm} 
\usepackage{fancyhdr} 
\usepackage{lipsum} 
\usepackage{color} 
\usepackage{bm} 
\usepackage{amsfonts,amsthm,amsmath,amssymb} 
\usepackage{mathtools}
\usepackage{graphicx}
\usepackage{float}
\usepackage{latexsym,graphicx,epstopdf} 
\usepackage{etoolbox} 
\usepackage{hyperref} 
\usepackage[left,mathlines]{lineno} 
\definecolor{darkgreen}{rgb}{0.0, 0.42, 0.24}
\definecolor{orange}{rgb}{0.93, 0.57, 0.13}
\definecolor{cafe}{rgb}{0.43, 0.21, 0.1}
\definecolor{pink}{rgb}{1.0, 0.0, 0.5}

\newcommand\OO{\Omega}

\newcommand\RR{\mathbb R}

\newcommand\LLh{{\mathcal L}_h}

\newcommand\PP{\mathbb{P}}
\newcommand\B{B_{\mathrm{stab}}}
\newcommand\F{F_{\mathrm{stab}}}

\newcommand\ff{{\boldsymbol f}}

\newcommand\uu{{\boldsymbol u}}
\newcommand\uumh{{\boldsymbol u}_{m}}
\newcommand\zz{{\boldsymbol z}}
\newcommand\vv{{\boldsymbol v}}

\newcommand\ww{{\boldsymbol w}}
\newcommand\II{\boldsymbol{I}}

\newcommand\JJJ{\mathcal{J}}

\newcommand{\defi}{{\,:=\,}}

\newcommand{\zero}{\boldsymbol{0}}

\newcommand{\balpha}{\boldsymbol{\alpha}}

\newcommand{\HH}{\boldsymbol{H}}

\newcommand{\TT}{\mathcal{T}}

\newcommand{\etw}{\eta^{\ww}}

\newcommand{\etp}{\eta^p}

\newcommand{\ew}{e^{\ww}_h}
\newcommand{\ep}{e^p_h}

\newcommand{\T}{\mathcal{T}}

\newcommand{\nh}[1]{{\left\vert\kern-0.25ex\left\vert\kern-0.25ex\left\vert #1 \right\vert\kern-0.25ex\right\vert\kern-0.25ex\right\vert}}

\newcommand{\Sconv}{S_{\mathrm{conv}}}
\newcommand{\Spress}{S_{\mathrm{press}}}

\newcommand{\nhc}[1]{{\left\vert\kern-0.25ex\left\vert\kern-0.25ex\left\vert #1 
		\right\vert\kern-0.25ex\right\vert\kern-0.25ex\right\vert^2}}

\newcommand{\aaa}{\boldsymbol{a}}
\newcommand{\aaah}{\boldsymbol{a}_h}

\newcommand{\trih}{\mathcal{T}_h}

\newcommand{\dmean}[1]{{\left\{\kern-0.6ex\left\{ #1   \right\}\kern-0.6ex\right\}}}

\newcommand{\Cinv}{C_{\mathrm{inv}}}

\newcommand{\vvt}{\hat{\vv}}

\newtheorem{theorem}{Theorem}

\newtheorem{lemma}{Lemma}

\newtheorem{remark}{Remark}
\newtheorem{corollary}{Corollary}

\title{{A stabilized finite element method for a flow  problem arising from 4D flow magnetic resonance imaging}}

\author{{Gabriel R. Barrenechea\thanks{Department of Mathematics and Statistics, University of Strathclyde, 26 Richmond Street,	Glasgow, G1 1XH, UK}\quad Cristian C\'arcamo\thanks{Weierstraß-Institut,  Anton-Wilhelm-Amo-Straße 39, 10117 Berlin, Germany}\quad Abner H. Poza\thanks{Departamento de Matemática y Física Aplicadas, Universidad Católica de la Santísima Concepción, Casilla 297, Concepción, Chile}
}}

\begin{document}

\maketitle

\begin{abstract}
In this work we propose, {analyze}, and validate a stabilized finite element method for 
a flow problem arising from the assessment of 
{4D Flow Magnetic Resonance Imaging quality}.  Starting from the Navier-Stokes equation and splitting its velocity as the MRI-observed one (considered a datum) plus an ``observation  error'',
a modified Navier-Stokes problem is derived. This procedure allows us to estimate the quality of the measured velocity fields, while also providing an alternative approach to pressure reconstruction,
thereby avoiding invasive procedures.  Since equal-order approximations have become  a popular choice for problems linked to pressure recovery from MRI images, we design a stabilized finite element method allowing equal-order interpolations for velocity and pressure.   In the linearized version of the resulting model, we prove stability and (optimal order) error estimates and test the method with a
variety of numerical experiments testing both the linearized case  and the more realistic nonlinear one.
\end{abstract}

\vspace{.25cm}


\section{Introduction}\label{section:intro}
A precise characterization of cardiovascular hemodynamics is fundamental for the early diagnosis and risk stratification in pathologies such as aortic aneurysms, valvular diseases and congenital heart defects \cite{Markl12,Hsu2021}.  In this scenario, {4D flow magnetic resonance imaging (4D flow MRI)} has emerged as the non-invasive gold-standard technique to quantify in vivo the {cardiovascular} blood flow \cite{Horowitz2021}.  Unlike ultrasound or conventional 2D resonance techniques {\cite{Gabbour2015_4DFlow,Ramos2020}}, 4D  flow MRI enables the acquisition of three-dimensional velocity fields resolved in time, opening the door to the calculation of hemodynamics biomarkers such as 
{helicity, oscillatory shear index, pressure, turbulent kinetic energy, vorticity, wall shear stress, among other quantities of interest that are necessary to characterize the physiological or pathological state of the cardiovascular system \cite{Bissell23}}. Nevertheless, the clinic use of this technology presents a longstanding obstacle: the data acquisition time.  Encoding velocity in three spatial directions throughout the cardiac cycle requires the acquisition of a large amount of frequency-domain data ($k-$space).  
As a result, the scanning time increases which may lead to the introduction of artifacts, thereby degrading the image quality and the reliability of subsequent numerical quantification \cite{BER04}. \\

To lower the processing time, the scientific community has moved away from the classical Nyquist-Shanon sampling,  which traditionally required fully sampling $k-$space to avoid aliasing \cite{sha49}, to aggressive undersampling techniques that dramatically reduce acquisition time \cite{Markl12}. This implies that just a small fraction of the data is captured reducing the scanning time. However, the application of aggressive undersampling also means being careful, because otherwise we could obtain aliasing. The shift from Nyquist-Shannon to aggressive undersampling  has been made possible by modern reconstruction strategies,  which recover high-fidelity images from highly sparse $k-$space data, thereby enabling clinically feasible {4D flow MRI} acquisition. In the past decade, two techniques have dominated the methodological landscape: Parallel Imaging (SENSE and  GRAPPA) \cite{Pruessmann1999,Griswold2002,Tsao2003,Schnell2014} and Compressed Sensing \cite{Lustig2007,Ma2019,Neuhaus2019,Pathrose2021}. These methods assume that flow signal is sparse in some transformation domain, and reconstruct the image solving an optimization problem \cite{BloUe07, Pruessmann1999}.\\
%
%
%
%
%
{Total variation regularization is a sparsity-promoting prior commonly used in compressed sensing MRI and 4D flow MRI reconstructions \cite{Lustig2007,Lustig2008}}.
Seminal studies have demonstrated that regularization based on the {total variation} and its generalized variants enable robust reconstructions with significant acceleration factors \cite{BloUe07}. However,  most of the algorithms of standard 4D flow data reconstruction process, ignore the physical nature of the subject of study, which is a moving incompressible fluid, {and thus for example, violate fundamental physical principles, such as mass conservation.  This {gap} introduces systematic errors in the computation of spatial derivatives, which suggests that purely phenomenological reconstruction pressure have reached a ceiling in their numerical precision.  In order to obtain reliable approximations of hemodynamic biomarkers, including pressure differences,  some approaches have been developed over the past decades.
	Despite the careful handling and processing of 4D flow MRI data by trained technicians, the acquired measurements inevitably contain noise. This noise may now originate from limitations in instrument calibration as well as from involuntary patient motion during image acquisition.  \\
	{In light of these considerations, we assume throughout this work that the available 4D flow MRI data consist of noisy velocity measurements, while the underlying physical velocity field $\uu$ satisfies the incompressible Navier–Stokes equations.} Decomposing the fluid velocity $\uu$ into the reconstruction $\uu_m$ of the MRI-obtained data and an unknown noise component $\ww$, 
	we establish, under suitable assumptions, that $\ww$
	satisfies a modified version of the Navier-Stokes equations, whose source terms depend, among other factors, on $\uu_m$. The unknowns in this differential problem are the noise velocity 
	$\ww$ and the fluid pressure $p$,  and in this way we obtain an alternative procedure for pressure reconstruction (see, e.g., \cite{carber2023,GMSCUBM2022}).
	Due to the nonlinear nature of the problem and dependence on time intrinsic to the pulsatile flow, the computational costs might be prohibitively expensive.  
	As a consequence of the above discussion, it has become usual to use low-order, equal order finite element pairs for velocity and pressure in the fluid solution (see, e.g., \cite{GMSCUBM2022}). This choice violates the standard
	inf-sup condition that is advocated to prove stability (see, e.g., \cite{BGHRR24}), but it can be extremely efficient if it is supplemented with appropriate stabilization terms.  
	
	Stabilized finite element methods appeared as a need to both enhance the stability for convection-dominated problems, and also to render velocity-pressure spaces that are not inf-sup stable
	(in particular, equal-order spaces) usable in practice.  Since the mid 1980's there has been a vast amount of work devoted to this technique, and a full literature review is outside the scope
	of the present work. Herein, we briefly mention that, for fluid problems, the methods can be divided into residual (such as PSPG/SDFEM \cite{HFB86,TV96}, GLS \cite{FS91,XIA2007513,Blas-Aniso-08,PR24}, Unusual FEM \cite{FF95,BV02}, just to name a few, 
	see \cite{BBGS-Taxo-04} for a  review), and non-residual (such as CIP \cite{BFH-CIP-06},  or LPS \cite{Be-Bra-016,BV2010,APV15}, see \cite{RST08,John2020} for  reviews). 
	
	In this work we propose a stabilized finite element method for  a flow problem  arising from {4D flow MRI}. The design of the stabilizing term is inspired from the one analyzed in \cite{BFV04}, (and thus it can be aligned with the family of Unusual Stabilized finite element methods)  as this
	type of stabilization has provided very accurate  stable numerical results, e.g., in \cite{ACPV21}.   The method's presentation and analysis are presented for a linearized version of the problem,  where we prove
	well-posedness and optimal error estimates.  Additionally,  we
	present numerical experiments for the steady-state Navier-Stokes equation,  
	showcasing the ability of the method to
	recover pressure and velocities from piecewise constant data, as it would be natural in the {4D flow MRI} context. 
	A related approach can be found in \cite{GMSCUBM2022}, where a finite element scheme is studied, and one of the main contributions of the present work is the analysis of the optimal convergence of the proposed scheme.}

The rest of the manuscript is organized as follows.  The remaining of this introduction is devoted to presenting some preliminary notation used throughout the manuscript. 
In Section~\ref{Sec:model} we present the derivation of the perturbed problem starting from the Navier-Stokes equation,  which is then analyzed in Section~\ref{model}. 
The stabilized finite element method for the linearized problem in presented in Section~\ref{discrete}, and its well-posedness and optimal order error estimates in velocity and pressure
are proven in Section~\ref{Sec:Error}.  Finally, in Section~\ref{Sec:Numerics} we present a series of numerical experiments both validating the error estimates,  and in more realistic cases.
Some conclusions and future work are drawn in Section~\ref{conclusions}.

\subsection{Preliminary notation.} Throughout the work we utilize standard simplified terminology for Sobolev space, inner product and norms (see, e.g. \cite{EG21-FE1}).
In particular, if $\mathcal{O}$ is an open subset of $\RR^d$, with $d\in\{2,3\}$, $L^2(\mathcal{O})$ is the {Hilbert} space of Lebesgue square integrable functions over $\mathcal{O}$ and $(\cdot,\cdot)_{\mathcal{O}}$ denotes the $L^2(\mathcal{O})$ inner product for scalar, vector or {tensor-valued} functions,  whenever appropriate.
As is conventional, $W^{k,p}(\mathcal{O})$, with $1\leq p\leq\infty$ and $k\geq 0$, denotes the Sobolev spaces with norm $\|\cdot\|_{W^{k,p}(\mathcal{O})}$, and $H^k(\mathcal{O})=W^{k,2}(\mathbb{\mathcal{O}})$ denotes the Hilbert spaces with seminorm $|\cdot|_{k,\mathcal{O}}$ and norm $\|\cdot\|_{W^{k,2}(\mathcal{O})}=\|\cdot\|_{k,\mathcal{O}}$ for scalar, vector or tensor-valued functions,  when appropriate.  
In our problem of interest $\Omega\subset\RR^{d}$, $d\in\{2,\,3\}$ will be an open, bounded, domain with polyhedral Lipschitz boundary {$\partial \OO$}.  For the particular case $\mathcal{O}=\Omega$, we will simply denote the inner product $(\cdot,\cdot)_\Omega$ by $(\cdot,\cdot)$.  Finally, in the weak formulation below we will  use of the following Hilbert spaces
\[
\HH\defi H_0^1(\OO)^d=\{\vv\in H^1(\OO)^d:\vv=\boldsymbol{0} \textrm{ on }\partial\OO\}, \qquad 
Q\defi  L^2_0(\OO)=\left\{q\in L^2(\OO):\int_\OO q=0\right\}.
\]

\section{From the 4D flow MRI's image to the model problem}\label{Sec:model}
We consider  the dynamics of blood flow in human circulatory system modeled by 
the incompressible Navier-Stokes equations in the vessel lumen $\OO$: {Find the perfect velocity and pressure measurements $(\uu,p)$, such that}
\begin{equation*}
	\label{NS}
	(\textrm{NS}) \qquad \left\{
	\begin{array}{rll}
		\rho \, \uu_t  - \mu \, \Delta \uu + \rho\, (\nabla \uu)\uu+ \nabla p   \,&= \, \zero&\text{in }\OO \times(0,T],\\
		\nabla\cdot\uu \,&=\, 0 &\text{in }\OO \times (0,T],\\
		\uu  \,&=\,  \boldsymbol{0} &\text{on }\partial \OO\times(0,T],\\
		\uu(0)  \,&=\,  \uu_0 &\text{in } \OO,
	\end{array}
	\right.
\end{equation*}
where $\uu_0\in H^1_0(\OO)^d$ is an initial datum,
$\mu>0$ is the kinematic viscosity, $\rho>0$ is the density, 
and $T$ denotes the final time.\\
As it was mentioned in the introduction, the main goal in this paper is to assess how
good the measurements made using MRI are.   For this purpose, we denote by $\uu_m$ the reconstruction
of the 4D flow measurements (e.g., the interpolation of nodal values of the velocity), and we assume
that there is an additive noise $\ww$ in the velocity measurements, that satisfies the following properties:
\begin{itemize}
	\item[(H1)] $\uu=\uu_m + \ww$ in $\OO$.
	\item[(H2)] $\nabla \cdot \ww= -\nabla \cdot \uu_m=g$ in $\OO$.
	\item[(H3)] $\ww=\zero$ on $\partial \OO$.
\end{itemize}
Using (H1)-(H3), we can see that a variational formulation of $(\textrm{NS})$ is given by:
\emph{Find} $(\ww(t),p(t))\in \HH \times Q$ such that
\begin{align}
	\nonumber
\small	\rho \, (\ww_t,\vv) + \mu \, (\nabla \ww,\nabla \vv) + &\, \rho\, (\nabla(\uu_m + \ww)\ww,\vv) +\rho\, ((\nabla \ww)\uu_m,\vv)
	- (p,\nabla \cdot \vv) + (q,\nabla \cdot \ww)\\
	=&\, -\rho\, ((\uu_m)_t,\vv) -
	\mu\, (\nabla \uu_m,\nabla \vv) - \rho \, ((\nabla \uu_m)\uu_m,\vv) +(g,q),
\end{align}
for all $(\vv,q) \in \HH\times Q$.

Now, if we use a 
semi-implicit Euler scheme with fixed time step $\tau$,
%
we have the following time discretization: Given $\ww^0=\zero$, 
for $k=1,2,\dots,$  find $(\ww^k,p^k)\in \HH \times Q$ such that
\begin{align}
	\nonumber
	&\, \sigma\, (\ww^k,\vv)
	+
	\mu \,(\nabla \ww^k,\nabla \vv) +
	\rho\, ((\nabla \uu_m^k)\ww^k,\vv)+
	\rho\, ((\nabla \ww^k)\ww^{k-1},\vv)
	+\rho \, ((\nabla \ww^k)\uu_m^k,\vv)  - (p^k,\nabla \cdot \vv)\\
	&\, + (q,\nabla \cdot \ww^k)  =\sigma\, (\ww^{k-1},\vv)- \sigma\, (\uu_m^k-\uu_m^{k-1},\vv) -\mu(\nabla \uu_m^k,\nabla \vv) 
	-\rho \, ((\nabla \uu_m^k)\uu_m^k,\vv)+(g,q),
	\label{timescheme}
\end{align}
where $\sigma \defi \dfrac{\rho}{\tau}$, for all $(\vv,q) \in \HH\times Q$.

{The time scheme \eqref{timescheme} is, in fact, the weak formulation of a 
	linear modified Oseen equation. So, in the next section we will introduce a model problem that
	is an idealization of \eqref{timescheme} and for which we will propose a stabilized finite element scheme.
	
	
	\section{Model problem and preliminary results} \label{model}
	The singularly perturbed Oseen problem that will study in this work consists on finding the velocity  $\ww$ and the pressure $p$ such that satisfying
	equation system
	\begin{equation*}
		\label{P}
		(\textrm{P}) \qquad 
		\left\{
		\begin{array}{rll}
			\sigma\, \ww  - \mu \, \Delta \ww +
			\rho\, (\nabla \uu_m)\ww + \rho\, (\nabla \ww)(\aaa+\uu_m)+ \nabla p  \,&= \, \ff &\text{in }\OO,\\
			\nabla\cdot\ww \,&=\, g &\text{in }\OO,\\
			\ww  \,&=\,  \boldsymbol{0} &\text{on }\partial \OO,\\
		\end{array}
		\right.
	\end{equation*}
	where  
	$\ff\in L^2(\OO)^d$ is the body force vector,  $g\in L_0^2(\OO)$,  
	$\uu_m\in W^{1,\infty}(\OO)^d$ and $\aaa\in W^{1,\infty}(\OO)^d$  are  given convective velocities field.  In addition, $\aaa$ is a given solenoidal datum,
	and 
	$\uu_m$ is the velocity field reconstructed from MRI-acquired data in $\Omega$. The concrete definition for each
	case will be specified in Section~\ref{discrete} (cf. Remark  \ref{remark2}).

	The standard variational formulation of  $(\textrm{P})$ reads:
	{\emph{Find}} $(\ww,p)\in \HH \times Q$ such that
	\begin{align}
		\label{vf1}
		\tilde{a}(\ww,\vv) + b(\vv,p) =&\, (\ff,\vv) \quad \forall \vv \in \HH,\\
		\label{vf2}
		b(\ww,q) =&\, {-(g,q)} \quad \forall q \in Q,
	\end{align}
	where $\tilde{a}: \HH\times \HH \longrightarrow \RR$  and  $b: \HH\times Q \longrightarrow \RR$ are the bilinear forms defined by
	\begin{equation}
		\tilde{a}(\ww,\vv) = \sigma\, (\ww,\vv)  + \mu \,(\nabla \ww,\nabla \vv) +
		\rho\, ((\nabla \uu_m)\ww,\vv) 
		+\rho\, ((\nabla \ww)(\aaa+\uu_m),\vv)
		\qquad \forall \ww,\vv \in \HH,
		\label{def-form-a}
	\end{equation}
	and
	\begin{equation}
		b(\vv,q) =-(q,\nabla \cdot \vv) \qquad \forall (\vv,q) \in \HH \times Q.
		\label{def-form-b}
	\end{equation}
	%
	%
	%
	The following results will be needed throughout the paper.
	\begin{lemma}
		\label{trilinearprop} 
		For all $\ww,\vv \in \HH$ and $\balpha \in H^1(\OO)^d$, we get 
		\begin{equation}
			\label{eq_may:2.5}
			((\nabla \ww)\balpha,\vv) = - ((\nabla \vv)\balpha,\ww) - ((\nabla \cdot \balpha)\ww,\vv).
		\end{equation}
	\end{lemma}
	%
	%
	\begin{proof}
		See \cite{BGHRR24}.
	\end{proof}
	%

	%
	%
	\begin{theorem}\label{th1}
		Assume that $\uu_m \in W^{1,\infty}(\OO)^d$ and that
		\begin{equation}
			\label{assum1d}
			\sigma > 4 \rho \, \| \nabla \uu_m \|_{\infty,\OO}.
		\end{equation}
		Then \eqref{vf1}-\eqref{vf2} has a unique solution $(\ww,p)\in \HH\times Q$.
	\end{theorem}
	\begin{proof}
		Using the assumption \eqref{assum1d}, Lemma \ref{trilinearprop}  and that $\aaa$ and is
		divergence--free, for all $\vv \in \HH$ we get 
		\begin{align}
			\nonumber
			\tilde{a}(\vv,\vv) =&\, \sigma\, \|\vv\|_{0,\OO}^2 + \mu \, |\vv|_{1,\OO}^2 +
			\rho\, ((\nabla \uu_m)\vv,\vv)+  
			\rho\, ((\nabla \vv)\aaa,\vv) + \rho \, ((\nabla \vv)\uu_m,\vv) \\
			\nonumber
			\ge&\, \sigma\, \|\vv\|_{0,\OO}^2 +\mu \, |\vv|_{1,\OO}^2 +  
			\rho\, ((\nabla \uu_m)\vv,\vv) - \dfrac{\rho}{2}\, ((\nabla \cdot \uu_m)\vv,\vv) \\
			\nonumber
			\ge&\, \left(\sigma-\rho\, \|\nabla \uu_m\|_{\infty,\OO} - \dfrac{\rho}{2}\, \| \nabla \cdot \uu_m\|_{\infty,\OO}\right)\, \|\vv\|_{0,\OO}^2 +
			\mu \, |\vv|_{1,\OO}^2 
			\\
			\nonumber
			\ge&\,\left(\sigma-2\rho\, \|\nabla \uu_m\|_{\infty,\OO}\right)\, \|\vv\|_{0,\OO}^2  +  \mu \, |\vv|_{1,\OO}^2 \\
			\nonumber
			\ge&\, {\dfrac{1}{2}\, \Big\{\sigma\, \|\vv\|_{0,\OO}^2 + \mu \, |\vv|_{1,\OO}^2  \Big\},}
		\end{align}
		and thus $ \tilde{a}(\cdot,\cdot)$ is elliptic in $\boldsymbol{H}$.
		On the other hand, it is well-known in the literature (for details, see \cite{BGHRR24}), that the bilinear form $b$ satisfies the inf-sup condition:
		\[ \sup_{\vv\in \HH} \dfrac{b(\vv,q)}{\|\vv\|_{1,\OO}} \ge \beta \, \|q\|_{0,\OO} \qquad \forall q \in Q,\]
		where $\beta$ is a positive constant. Then the classical theory of Babu\v{s}ka-Brezzi can be applied to \eqref{vf1}-\eqref{vf2} and  the existence and uniqueness of weak solution $(\ww,p)$ is guaranteed.
	\end{proof}

	\begin{remark}
		The  assumption \eqref{assum1d} may appear as restrictive,  but it is consistent with the CFL condition 
		that would arise in a time-dependent simulation, if we remark that in such a case we would have
		$\dfrac{\sigma}{4\rho}=\dfrac{1}{4\tau}$, where $\tau$ represents a fixed time step.
	\end{remark}

\section{The stabilized method} \label{discrete}

To propose the discrete scheme, in this work we will consider the  following assumptions:

\begin{itemize}
	\item There is a shape-regular initial mesh $\T_H$ of $\overline{\OO}$
	that allows the reconstruction of a {velocity} data $\uu_m\in \HH_H$ (which will be specified later). %
	\item When the initial mesh $\T_H$ is refined, this defines 
	a shape-regular family  $\{\T_h\}_{h>0}$ of partitions of $\overline{\OO}$
	with elements $T$, corresponding to triangles in $d=2$ and tetrahedra in $d=3$, respectively, of diameter $h_T$, and $h:=\max_{T\in\T_h}h_T$.
	\item There is an approximation of $\aaa$, denoted by $\aaa_h$, that belongs to the finite element space $\HH_h$ defined below. 
\end{itemize}

\begin{remark}
	\label{remark2}
	The function $\aaa_h$ may correspond to an approximation of the velocity {noise} in the previous step of a nonlinear scheme such as Newton or Picard iteration to solve the Navier-Stokes equation. For this reason in Section \ref{convergence} we will consider that the convective term $\aaa$ has the same regularity as the velocity {noise  $\ww$} and that his approximation $\aaa_h$ is in the same space as { $\ww_h$}.
\end{remark}

Since the objective is to propose a discrete scheme of finite elements with spaces of the same order of approximation, we introduce the following finite element subspaces of $\HH$ and $Q$, respectively:
\begin{align*}
	\HH_{h}  := & \left\{ \vv \in C(\overline{\OO})^d  \;:\; \vv|_T \in  \PP_k(T)^d,\quad \forall T\in\T_h  \right\}  \cap \HH,\\
	Q_{h}    := & \left\{  q \in C(\overline{\OO})  \;:\; q|_T \in  \PP_k(T), \quad  \forall T\in\T_h \right\}\cap Q,
\end{align*} 
with $k\geq 1$, where $\PP_k$ stands for the  space of polynomials of total degree smaller than or equal to $k$. \\

Using the above-defined finite element spaces 
the discrete stabilized scheme to approximate \eqref{vf1}-\eqref{vf2} {is given by}: {\emph{Find}} $(\ww_h,p_h)\in \HH_{h}\times Q_{h}$ such that
\begin{equation}\label{stab}
	\B(\ww_h,p_h;\vv_h,q_h)=\F(\vv_h,q_h)
	\qquad \forall (\vv_h,q_h)\in \HH_{h}\times Q_{h},
\end{equation}
where $\B:(\HH_{h}\times Q_{h})\times (\HH_{h}\times Q_{h}) \longrightarrow \RR$ is the bilinear form defined by
\begin{equation}
	\B(\ww_h,p_h;\vv_h,q_h) \defi\, A(\ww_h,p_h;\vv_h,q_h) + \Sconv(\ww_h;\vv_h)+\Spress(\ww_h,p_h;\vv_h,q_h).
	\label{def-form-Bstab}
\end{equation}
Here, the bilinear form {$A(\cdot,\cdot)$} is given by
\begin{align}
	\nonumber
	A(\ww_h,p_h;\vv_h,q_h) \defi &\, 
	\sigma\, (\ww_h,\vv_h) +  \mu \,(\nabla \ww_h,\nabla \vv_h) +
	\rho\, ((\nabla \uu_{m})\ww_h,\vv_h) +
	\\
	&\,  \lambda\, (\nabla \cdot \ww_h, \nabla\cdot \vv_h)  - (p_h,\nabla \cdot \vv_h)+ (q_h,\nabla \cdot \ww_h),
	\label{def-form-A}
\end{align}
where $\lambda$ is a positive parameter. Moreover, the modified convective term is 
\begin{equation}
	\Sconv(\ww_h;\vv_h) \defi \,   \rho\, ((\nabla \ww_h)(\aaa_h+\uu_{m}),\vv_h) + \dfrac{\rho}{2} \, \left( (\nabla \cdot \aaa_h)\, \ww_h , \vv_h\right),
	\label{def-form-Sconv}
\end{equation}
and the pressure stabilization term is defined by
\begin{equation}
\small	\Spress(\ww_h,p_h;\vv_h,q_h) 
	\defi \sum_{T \in \T_h}  \tau_T \, \Big(
	\sigma\, \ww_h  - \mu \, \Delta \ww_h + \LLh(\ww_h,p_h),
	\, - \sigma\, \vv_h +   \mu \, \Delta \vv_h + \LLh(\vv_h,q_h)
	\Big)_{T}, 
	\label{def-form-Spress}
\end{equation}
where $\LLh$ is the linear operator defined by $\LLh (\ww_h,p_h) \defi \rho\, (\nabla \uu_{m})\ww_h + \rho\, (\nabla \ww_h)(\aaa_h+\uu_{m}) + \nabla p_h$ and the stabilization parameter is given by
\begin{equation}\label{Def:Stab_Param}
	\tau_T \defi \dfrac{\delta h_T^2}{\sigma\, h_T^2 + \mu} \qquad \forall T \in \TT_h,
\end{equation}
where $\delta$ is a positive constant that to be specified later  (c.f. Lemma \ref{coercivity_lemma}). 
Finally, $\F :\HH_{h}\times Q_{h} \longrightarrow \RR$ is the linear form defined by
\begin{equation}
	\F(\vv_h,q_h) \defi (\ff,\vv_h)+
	(q_h,g) + \lambda\, (g,\nabla \cdot \vv_h)+ \sum_{T \in \T_h}  \tau_T \, \Big(\ff , 
	-\sigma\, \vv_h +   \mu \, \Delta \vv_h + \LLh(\vv_h,q_h) \Big)_{T}, 
	\label{def-form-Fstab}
\end{equation}
for all $(\vv_h,q_h)\in \HH_{h} \times Q_{h}$.

\begin{remark}
	The stabilized finite element method \eqref{stab} is inspired  in \cite{BFV04},  where a  Galerkin least--squares stabilized formulation for the Oseen equation was proposed.  {Unlike the scheme
		in \cite{BFV04},  due to the approximation $\aaa_h$ of $\aaa$, this scheme is no longer consistent,
		so an {ad-doc} analysis will have to be performed. Finally, it can be easily seen that the parameter
		$\tau_T$ satisfies the following properties:}
	\begin{align}
		\label{prop_tau1}
		\sigma\,  \tau_T  \le&\, \delta,\\
		\label{prop_tau2}
		\mu\,  \tau_T  \le&\, \delta h_T^2,
	\end{align}
	for all $T\in \TT_h$.
\end{remark}
%
%
%
%
%
%
{Throughout this paper $C$ and $C_i$, $i>0$ will denote positive constants independent of the discretization parameter $h$ and the constants 
	$\mu$, $\sigma$, and $\rho$, 
	which may take different values at different places.}
\section{Stability and error estimates}\label{Sec:Error}


\subsection{Well-posedness} We start recalling that, as a consequence of the standard inverse inequality
(see, e.g., \cite[Lemma~12.1]{EG21-FE1}), there exists $C_{\mathrm{inv}}>0$, independent of $h$, 
such that
\begin{equation}\label{Cinv}
h_T\,\|\Delta\vv_h\|_{0,T} \le \, C_{\mathrm{inv}}\,|\vv_h|_{1,T},\\
\end{equation}
for all $T\in \TT_h$ and $\vv_h\in \HH_{h}$.

\noindent In  the product spaces $\HH_{h}\times Q_{h}$, we consider the norm
\begin{equation*}
\nh{(\vv_h,q_h)}:=  \left\{ \sigma\, \|\vv_h\|_{0,\OO}^2+\mu\, | \vv_h|_{1,\OO}^2+
\lambda \, \| \nabla\cdot \vv_h\|_{0,\OO}^2
+  \sum_{T \in \T_h}  \tau_T \,  \| \LLh(\vv_h,q_h)\|^2_{0,T}   
\right\}^{1/2},
\end{equation*}
for all $(\vv_h,q_h)\in \HH_{h}\times Q_{h}$.

The next result guarantees the well-posedness of the stabilized finite element scheme \eqref{stab}.

\begin{lemma}
\label{coercivity_lemma}
{Assume  that \eqref{assum1d} in Theorem~\ref{th1}  holds and that
	\begin{equation}
		\label{hip3-lemma3}
		\delta < \dfrac{1}{4}\,\min\left\{ \dfrac{1}{2}, \dfrac{1}{C_{inv}^2} \right\}.
	\end{equation}
	Then,  the bilinear form $\B(\cdot,\cdot)$ satisfies
	\begin{equation}
		\label{eq:3.3}
		\B(\ww_h,p_h;\ww_h,p_h) \ge  \frac{1}{4}\,  \nh{(\ww_h,p_h)}^2,
	\end{equation}
	for all $(\vv_h,q_h)\in \HH_{h} \times Q_{h}$.}
	\end{lemma}
	
	%
	%
	
	\begin{proof}
Similar to the proof of Theorem \ref{th1}, due to the definition of bilinear forms $A(\cdot;\cdot)$ and 
$\Sconv(\cdot;\cdot)$, assumption \eqref{assum1d} and Lemma \ref{trilinearprop}, we have for all $(\ww_h,p_h)\in \HH_h\times Q_h$ that
\begin{align}
	\nonumber
	&\, A(\ww_h,p_h;\ww_h,p_h)  +\Sconv(\ww_h;\ww_h)  \\ \nonumber
	=&\, 
	\sigma\, \|\ww_h\|_{0,\OO}^2 +  \mu \,| \ww_h|_{1,\OO}^2  +
	\rho\, ((\nabla \uu_{m})\ww_h,\ww_h) +
	\lambda\, \|\nabla \cdot \ww_h\|_{0,\OO}^2 +
	\rho\, ((\nabla \ww_h)(\aaa_h+\uu_{m}),\ww_h) \\
	& + \dfrac{\rho}{2} \, \left( (\nabla \cdot \aaa_h)\, \ww_h , \ww_h\right)	\label{eq5:3.9} \\ 
	\nonumber
	\ge&\, \left(\sigma-\rho\, \|\nabla \uu_m\|_{\infty,\OO} - \dfrac{\rho}{2}\, \| \nabla \cdot \uu_m\|_{\infty,\OO}\right)\, \|\ww_h\|_{0,\OO}^2 +
	\mu \, |\ww_h|_{1,\OO}^2 +\lambda\, \|\nabla \cdot \ww_h\|_{0,\OO}^2 
	\\
	\nonumber
	\ge&\, \left(\sigma-2\rho\, \|\nabla \uu_m\|_{\infty,\OO}\right)\, \|\ww_h\|_{0,\OO}^2 +\mu \, |\ww_h|_{1,\OO}^2 + \lambda\, \|\nabla \cdot \ww_h\|_{0,\OO}^2  \\
	\ge&\, \dfrac{\sigma}{2}\, \|\ww_h\|_{0,\OO}^2 +\mu \, |\ww_h|_{1,\OO}^2 + \lambda\, \|\nabla \cdot \ww_h\|_{0,\OO}^2 .
\nonumber
\end{align}
Now, using the definition of $\Spress(\cdot;\cdot)$, Young's inequality, \eqref{prop_tau1}-\eqref{prop_tau2} and \eqref{Cinv}, we obtain
\begin{align}
	\nonumber
	\small \Spress(\ww_h,p_h;\ww_h,q_h) 
	=&\, \small\sum_{T \in \T_h}  \tau_T \, \Big(
	\sigma\, \ww_h  - \mu \, \Delta \ww_h + \LLh(\ww_h,p_h),
	\, - \sigma\, \ww_h +   \mu \, \Delta \ww_h + \LLh(\ww_h,q_h)
	\Big)_{T}\\
	\nonumber
	=&\, \small \sum_{T \in \T_h}  \tau_T \, \left\{
	-\|\sigma\, \ww_h  - \mu \, \Delta \ww_h \|_{0,T}^2+ \|\LLh(\ww_h,p_h)\|_{0,T}^2
	\right\}\\
	\nonumber
	=&\, \small \sum_{T \in \T_h}  \tau_T \, \left\{
	-\sigma^2\,\| \ww_h \|_{0,T}^2+2\sigma \mu\,(\ww_h,\Delta \ww_h)_T- \mu^2\,\|\Delta\ww_h \|_{0,T}^2+ \|\LLh(\ww_h,p_h)\|_{0,T}^2
	\right\}\\
	\nonumber
	\ge&\,\small \sum_{T \in \T_h}  \tau_T \, \left\{
	-2\sigma^2\,\| \ww_h \|_{0,T}^2-2 \mu^2\,\|\Delta\ww_h \|_{0,T}^2+ \|\LLh(\ww_h,p_h)\|_{0,T}^2
	\right\}\\
	%
	%
	%
	%
	%
	%
	\nonumber
	\ge&\, \, \small
	-2\sum_{T \in \T_h}  \sigma\delta \,  \| \ww_h \|_{0,T}^2
	-2 \sum_{T \in \T_h}  \mu \delta \Cinv^2\,|\ww_h |_{1,T}^2
	+\sum_{T \in \T_h}  \tau_T \,  \|\LLh(\ww_h,p_h)\|_{0,T}^2\\
	\ge&\, \, \small
	-\dfrac{\sigma}{4} \,  \| \ww_h \|_{0,\OO}^2
	-\dfrac{\mu}{2}\,|\ww_h |_{1,\OO}^2
	+\sum_{T \in \T_h}  \tau_T \,  \|\LLh(\ww_h,p_h)\|_{0,T}^2 , \label{eq5:5}
\end{align}
for all $(\ww_h,p_h) \in  \HH_h\times Q_h$. \\
Finally,  inserting \eqref{eq5:3.9}  and  \eqref{eq5:5} in the definition of {$\B(\cdot;\cdot)$} (cf. \eqref{def-form-Bstab}),
we get for all $(\ww_h,p_h) \in  \HH_h\times Q_h$ that
\begin{equation}
	\label{eq:3.3}
	\B(\ww_h,p_h;\ww_h,p_h) \ge \dfrac{ \sigma}{4}\, \|\ww_h\|_{0,\OO}^2+\dfrac{\mu}{2}\, |\ww_h|_{1,\OO}^2 
	+\lambda \, \| \nabla\cdot \ww_h\|_{0,\OO}^2
	+  \displaystyle \sum_{T \in \T_h}  \tau_T \,
	\|  \LLh(\ww_h,p_h) \|^2_{0,T}   ,
\end{equation}
and the result follows.
\end{proof}

\subsection{Error analysis}\label{convergence}

In what follows we will assume that the convective velocity field $\aaa_h$ in $\HH_h$ used to approximate  the  given convective velocity field 
$\aaa\in H^{k+1}(\OO)^d$, has the follow approximate property: There is a  positive constant $C$  independent of $h$, such that
{
\begin{align}
	\label{prop_aaa_1}
	\| \aaa-\aaa_h\|_{m,p,\OO}\le\, C\, h^{\ell+1-m}\|\aaa\|_{\ell+1,p,\OO}\quad,\quad
	{\|\aaa-\aaa_h\|_{m,\infty,\OO}\le\, C\, h^{\ell+1-m-d/2}\|\aaa\|_{\ell+1,\OO}},
\end{align}
for $0\le  m\le 1, 1\le p\le \infty$, and $0\le\ell\le k$.}
{As it was mentioned earlier, the stabilized scheme \eqref{stab} is not consistent.
The following result shows that the consistency error can be bounded in an optimal way.}
\begin{lemma} 
\label{pseudo-orthogo}
Let $(\ww,p) \in \big[\HH\cap H^{2}(\OO)^d\big] \times \big[Q\cap H^1(\OO)\big]$ and $(\ww_h,p_h)\in \HH_h \times Q_h$ 
be the solutions of (P) and \eqref{stab}, respectively. Assume the approximation assumption \eqref{prop_aaa_1} and the hypotheses of 
Lemma \ref{coercivity_lemma}.
Then, for all $(\vv_h,q_h)\in \HH_{h} \times Q_{h}$ we have
\begin{align}
	\nonumber
	\B(\ww-\ww_h,p-p_h;\vv_h,q_h) 
	=&\, 
	\dfrac{\rho}{2} \, ( \nabla\cdot(\aaa_h-\aaa)\ww,\vv_h)+ 
	\rho\, ( \nabla \ww(\aaa_h-\aaa),\vv_h)+
	\\ 
	&\,
	\sum_{T \in \T_h}  \tau_T \,
	\Bigg( 
	\rho\, (\nabla \ww)(\aaa_h-\aaa) ,
	-\sigma\, \vv_h +   \mu \, \Delta \vv_h + \LLh(\vv_h,q_h) 
	\Bigg)_T. \label{pseudo-ortho-1}
\end{align}
Moreover,  there is a positive constant $C$, independent of $h$, such that
\begin{align}
	\B(\ww-\ww_h,p-p_h;\vv_h,q_h) 
	\nonumber \le&\, C \, \sigma^{-1/2}
	\rho\, h^k \, 
	\| \aaa\|_{k+1,\OO}
	\|\ww\|_{2,\OO}\nh{(\vv_h,q_h)},
\end{align}
for all $(\vv_h,q_h)\in \HH_{h} \times Q_{h}$.
\end{lemma}

\begin{proof}
Using the definition of bilinear  form {$\B(\cdot;\cdot)$} (cf. \eqref{def-form-Bstab}), \eqref{vf1}-\eqref{vf2}, \eqref{stab}, 
and the fact that $\aaa$ is solenoidal, we have
\begin{align}
	\nonumber
	&\,\small \B(\ww-\ww_h,p-p_h;\vv_h,q_h) 
	\\ \nonumber
	=&\,\small A(\ww-\ww_h,p-p_h;\vv_h,q_h) + \Sconv(\ww-\ww_h;\vv_h)+\Spress(\ww-\ww_h,p-p_h;\vv_h,q_h) \\ \nonumber
	=&\,\small (\ff,\vv_h) +(g,q_h)+ \lambda\, (g,\nabla \cdot \vv_h)+
	\rho\, ( \nabla \ww(\aaa_h-\aaa),\vv_h)+
	\dfrac{\rho}{2} \, ( (\nabla\cdot \aaa_h)\ww,\vv_h)  +{\Spress(\ww,p;\vv_h,q_h)} \\ \nonumber
	&\,\small -\,A(\ww_h,p_h;\vv_h,q_h) - \Sconv(\ww_h;\vv_h)-\Spress(\ww_h,p_h;\vv_h,q_h) 
	\\ \nonumber
	=&\, \small
	\dfrac{\rho}{2} \, ( (\nabla\cdot \aaa_h)\ww,\vv_h)+ 
	\rho\, ( \nabla \ww(\aaa_h-\aaa),\vv_h)+\\ \nonumber
	&\,  \sum_{T \in \T_h}  \tau_T \, \Big(
	\sigma\, \ww  - \mu \, \Delta \ww + \LLh(\ww,p)-\ff,
	\, - \sigma\, \vv_h +   \mu \, \Delta \vv_h + \LLh(\vv_h,q_h)
	\Big)_{T}
	%
	%
	\\
	%
	%
	%
	%
	\nonumber =&\, \small
	\dfrac{\rho}{2} \, ( \nabla\cdot(\aaa_h-\aaa)\ww,\vv_h)+ 
	\rho\, ( \nabla \ww(\aaa_h-\aaa),\vv_h)+
	\\ 
	&\,\small \sum_{T \in \T_h}  \tau_T \,
	\Bigg( 
	\rho\, (\nabla \ww)(\aaa_h-\aaa) ,
	-\sigma\, \vv_h +   \mu \, \Delta \vv_h + \LLh(\vv_h,q_h) 
	\Bigg)_T,  \nonumber 
\end{align}
for all $(\vv_h,q_h)\in \HH_{h} \times Q_{h}$.\\

On the other hand, using  \eqref{prop_tau1}-\eqref{prop_tau2} and Lemma \ref{Cinv}, we can see that
\begin{align}
	\nonumber
	&\sum_{T \in \T_h}  \tau_T
	\, \left\| \sigma\, \vv_h - \mu\, \Delta \vv_h + \LLh(\vv_h,q_h)  \right\|_{0,T}^2  	\nonumber\\
	%
	%
	%
	& \nonumber \qquad \le \,C\,  \sum_{T \in \T_h}  \tau_T
	\Big[ 
	\sigma^2\, \|\vv_h\|_{0,T}^2 + \mu^2\, \| \Delta \vv_h\|_{0,T}^2+\|\LLh(\vv_h,q_h)\|_{0,T}^2\Big]\\
	\nonumber
	&	\qquad \le\, C\, \sum_{T \in \T_h}  
	\Big\{
	\sigma^2\tau_T\, \|\vv_h\|_{0,T}^2 + \mu^2 \tau_T h_T^{-2}\, | \vv_h|_{1,T}^2+\tau_T\, \|\LLh(\vv_h,q_h)\|_{0,T}^2\Big\}\\
	\nonumber
  &	\qquad	\le \, C\, \sum_{T \in \T_h}  
	\Big\{
	\sigma\, \|\vv_h\|_{0,T}^2 + \mu\, | \vv_h|_{1,T}^2+\tau_T\, \|\LLh(\vv_h,q_h)\|_{0,T}^2\Big\}\\
 &	\qquad	\le \, C\, \nh{(\vv_h,q_h)}^2. \label{pseudo-ortho-2}
\end{align}

Finally, using H\"older inequality, the fact that $\|\vv_h\|_{0,\OO} \le \sigma^{-1/2}\, \nh{(\vv_h,q_h)}$ and the continuous embedding  $H^2(\Omega)\xhookrightarrow{ } L^{\infty}(\Omega)$, 
\eqref{pseudo-ortho-2}, \eqref{prop_aaa_1}, we can see that 
\begin{align}
	\nonumber
	&\, \B(\ww-\ww_h,p-p_h;\vv_h,q_h) \\
	\nonumber \le&\, C
	\Bigg\{
	\rho\, \| \nabla \cdot(\aaa-\aaa_h)\|_{0,\OO}\|\ww\|_{\infty,\OO}\|\vv_h\|_{0,\OO}+
	\rho\, \| \aaa-\aaa_h\|_{1,\OO}\|\ww\|_{2,\OO}\|\vv_h\|_{0,\OO}
	+\\ \nonumber
	&\, \rho \, \sum_{T \in \T_h}  \tau_T \,
	\|\aaa-\aaa_h\|_{1,T}  \|\ww\|_{2,T}
	\left\| \sigma\, \vv_h - \mu\, \Delta \vv_h + \LLh(\vv_h,q_h)  \right\|_{0,T}
	\Bigg\}
	\nonumber\\
	\nonumber \le&\, C
	\rho\, 
	\| \aaa-\aaa_h\|_{1,\OO}
	\|\ww\|_{2,\OO}\|\vv_h\|_{0,\OO}+\\ \nonumber
	&\, C\, \rho \,\|\ww\|_{2,\OO}\, \Bigg\{ \sum_{T \in \T_h}  \tau_T \,
	\|\aaa-\aaa_h\|_{1,T}^2\Bigg\}^{1/2}
	\Bigg\{ \sum_{T \in \T_h}  \tau_T
	\left\| \sigma\, \vv_h - \mu\, \Delta \vv_h + \LLh(\vv_h,q_h)  \right\|_{0,T}^2
	\Bigg\}^{1/2}\\
	\nonumber \le&\, C \sigma^{-1/2}
	\rho\, 
	\| \aaa-\aaa_h\|_{1,\OO}
	\|\ww\|_{2,\OO}\nh{(\vv_h,q_h)} \\ \nonumber
	\nonumber \le&\, C \sigma^{-1/2}
	\rho\, h^k\, 
	\| \aaa\|_{k+1,\OO}
	\|\ww\|_{2,\OO}\nh{(\vv_h,q_h)} ,
\end{align}
and we conclude the proof.
\end{proof}

To estimate the velocity error,  we use the Lagrange interpolant 
$\II_h: H^{k+1}(\OO)^d \longrightarrow \HH_h$, which satisfies: For all $T \in \TT_h$
\begin{equation}
|\ww - \II_h\ww |_{l,T} \le C h_T^{s-l}\,|\ww|_{s,T},
\label{LAGRANGE1}
\end{equation}
for all $\ww \in H^{s}(T)^d$ with $l=0,1$ and  $\frac{d}{2} < s\le k+1$.  To approximate the  pressure  we consider the Cl\'ement interpolant $\mathcal{J}_h: H^{k}(\OO) \longrightarrow Q_h$, which satisfies: For all $T \in \TT_h$
\begin{align}
\label{LAGRANGE2}
|p-\mathcal{J}_hp|_{l,T} \le&\,  C h_T^{s-l}\, |p|_{s,\tilde{\omega}_T}  ,\\
\label{eq:5.38}
\|\mathcal{J}_h p \|_{1,\OO} \leq&\,  C\|p\|_{1,\OO},
\end{align}
for all $p \in H^{s}(T)$ with $l=0,1$ and  $ 1 \le s\le k$. where $\tilde{\omega}_T:=\cup\{ T'\in\TT_h:T\cap T'\not=\emptyset\}$.  For details in both interpolants and their analysis see \cite{BS08,EG21-FE1}.

{Now, we present the main result regarding error analysis for Method~\eqref{stab}.}


\begin{theorem}
\label{convergencia}
Assume  the assumptions of Lemmas \ref{coercivity_lemma} and \ref{pseudo-orthogo}.
Let $(\ww,p) \in \Big[\HH\cap H^{k+1}(\OO)^d \Big]\times \Big[Q\cap H^k(\OO)\Big]$ and $(\ww_h,p_h)\in \HH_h \times Q_h$ 
be the solutions of (P) and \eqref{stab}, respectively.
Then, there is a positive constant $C$, independent of $h$ and every physical 
parameter of the problem, such that
\begin{equation}\label{main-estimate} 
	\nh{(\ww-\ww_h,p-p_h)} \le C\,
	h^k\, \Bigg\{ 
	M_1 \,   \|\ww\|_{k+1,\OO}+
	M_2 \, 
	\|p\|_{k,\OO}
	\Bigg\},
\end{equation}
where 
\begin{equation}
	\label{M1}
	M_1 \defi  C_1 +
	\sigma^{-1/2}
	\rho\, \Big[  
	{h\, \|\nabla \cdot \uu_{m}\|_{\infty,\OO} + h^{2-d/2}\|\aaa\|_{2,\Omega}}+ 
	\|\aaa\|_{k+1,\OO}
	\Big],
\end{equation}
with
\begin{equation}
	\label{C1}
	{C_1 \defi  \sigma^{1/2}\,h+ \mu^{1/2}+ \lambda^{1/2} +
		\sigma^{-1/2}
		\rho\,  \big( \|\aaa \|_{\infty,\OO}+\|\uu_{m} \|_{\infty,\OO}\big) , }
\end{equation}
and
\begin{equation}
	\label{M2}
	M_2 \defi  \lambda^{-1/2} +\mu^{-1/2}.
\end{equation}

\end{theorem}

\begin{proof}
Let us define the following operators:

\begin{equation}
	\label{def-eta-eh}
	(\etw,\etp)\defi (\ww-\II_h\ww,p-\JJJ_hp)\quad \text{and}\quad (\ew,\ep)\defi (\ww_h-\II_h\ww,p_h-\JJJ_h p).
\end{equation}
Using the definition of bilinear form $A(\cdot;\cdot)$ and $\Sconv(\cdot;\cdot)$ given in \eqref{def-form-A}-\eqref{def-form-Sconv},
we have for all $(\vv_h,q_h) \in \HH_{h} \times Q_{h}$ that 
\begin{align}
	\nonumber &A((\etw,\etp),(\vv_h,q_h)) + \Sconv(\etw;\vv_h) \\ \nonumber
	=&\, 	\sigma\, (\etw,\vv_h) +  \mu \,(\nabla \etw,\nabla \vv_h) +
	\rho\, ((\nabla \uu_{m})\etw,\vv_h) +  \lambda \,(\nabla \cdot \etw,\nabla \cdot \vv_h) 
	- (\etp,\nabla \cdot \vv_h)+\\ \nonumber
	&\,  (q_h,\nabla \cdot \etw)+
	\rho\, ((\nabla \etw)(\aaa_h+\uumh),\vv_h)+ 
	\dfrac{\rho}{2} \, \left( (\nabla \cdot \aaa_h)\, \etw , \vv_h\right) \\
	=&\, I_1+ I_2.
	\label{eq:5.11.99}
\end{align}

Now, using Cauchy--Schwarz and H\"older inequalities, Lemma \ref{trilinearprop} and  \eqref{LAGRANGE1}-\eqref{LAGRANGE2}, we get 
\begin{align}
	\nonumber
	I_1
	=&\, 	\sigma\, (\etw,\vv_h) +  \mu \,(\nabla \etw,\nabla \vv_h) +
	\rho\, ((\nabla \uu_{m})\etw,\vv_h) +  \lambda \,(\nabla \cdot \etw,\nabla \cdot \vv_h) 
	- (\etp,\nabla \cdot \vv_h)\\ \nonumber
	\le &\, C\,   \Bigg\{ \sigma \, \|\etw\|_{0,\OO}^2+ \mu \, |\etw|_{1,\OO}^2 + 
	\sigma^{-1}\rho^2\, \|\nabla \uumh\|^2_{\infty,\OO}\|\etw\|^2_{0,\OO}+
	\lambda\, \| \nabla \cdot \etw\|_{0,\OO}^2 
	+ \lambda^{-1}\, \|\etp\|^2_{0,\OO}
	\Bigg\}^{1/2} \times \\ \nonumber
	&\, \Bigg\{ \sigma\, \|\vv_h\|^2_{0,\OO}+ \mu\, |\vv_h|_{1,\OO}^2 
	+\lambda\, \| \nabla \cdot \vv_h\|_{0,\OO}^2\Bigg\}^{1/2}\\ \nonumber
	\le &\,  C\, \Bigg\{ \sigma \, \|\etw\|_{0,\OO}^2+
	\left( \mu + \lambda \right) \, |\etw|_{1,\OO}^2 
	+ \lambda^{-1}\, \|\etp\|^2_{0,\OO}
	\Bigg\}^{1/2} \nh{(\vv_h,q_h)}
	\\
	\le &\,  C\,   h^k\, \Bigg\{   \left(\sigma^{1/2}\,h+ \mu^{1/2}+ \lambda^{1/2} \right)\,  \|\ww\|_{k+1,\OO}
	+ \lambda^{-1/2} \,
	\|p\|_{k,\OO}
	\Bigg\}\, \nh{(\vv_h,q_h)}.
	\label{eq:5.12a}
\end{align}

To bound $I_2$, we integrate by parts and using   Cauchy--Schwarz  and H\"older inequalities with \eqref{assum1d}, we have

\begin{align}
	\nonumber
	I_2 =&\,  (q_h,\nabla \cdot \etw)+
	\rho\, ((\nabla \etw)(\aaa_h+\uumh),\vv_h)+ 
	\dfrac{\rho}{2} \, \left( (\nabla \cdot \aaa_h)\, \etw , \vv_h\right) \\ \nonumber
	=&\, -(\nabla q_h,\etw)
	- \rho\,  ((\nabla \vv_h)(\aaa_h+\uu_m),\etw) 
	-\rho \, \left( (\nabla \cdot \uu_{m})\, \etw , \vv_h\right)
	-\dfrac{\rho}{2} \, \left( (\nabla \cdot \aaa_h)\, \etw , \vv_h\right) \\ \nonumber
	=&\, -\Bigg(\etw, \rho\,  (\nabla \vv_h)(\uu_m+\aaa_h) +\nabla q_h\Bigg)
	-\rho \, \left( (\nabla \cdot \uu_{m})\, \etw , \vv_h\right)
	-\dfrac{\rho}{2} \, \left( (\nabla \cdot \aaa_h)\, \etw , \vv_h\right) 
	\\ \nonumber
	\le&\, \sum_{T\in \TT_h}\|\etw\|_{0,T}\| \rho\,  (\nabla \vv_h)(\uu_m+\aaa_h) +\nabla q_h\|_{0,T}
	+\rho \, \|\nabla \cdot \uu_{m}\|_{\infty,\OO}\| \etw \|_{0,\OO}\| \vv_h\|_{0,\OO}
	\\ \nonumber
	&  +\rho \, \|\nabla \cdot \aaa_h\|_{\infty,\OO}\| \etw \|_{0,\OO}\| \vv_h\|_{0,\OO}
	\\ \nonumber
	\le&\, \sum_{T\in \TT_h}\|\etw\|_{0,T}\| \LLh(\vv_h,q_h)\|_{0,T}
	+ \sum_{T\in \TT_h}\|\etw\|_{0,T}\| \rho\,  (\nabla \uu_m)\vv_h\|_{0,T}
	\\ \nonumber & +\rho \, \Big[\|\nabla \cdot \uu_{m}\|_{\infty,\OO}+\|\nabla \cdot \aaa_h\|_{\infty,\OO} \Big]\| \etw \|_{0,\OO}\| \vv_h\|_{0,\OO}
	\\ \nonumber
	\le&\, \sum_{T\in \TT_h}\|\etw\|_{0,T}\| \LLh(\vv_h,q_h)\|_{0,T}
	+ \sum_{T\in \TT_h} \rho\, \|\nabla \uu_m\|_{\infty,T} \|\etw\|_{0,T}\| \vv_h\|_{0,T}
	\\ \nonumber & +\rho \, \Big[\|\nabla \cdot \uu_{m}\|_{\infty,\OO}+\|\nabla \cdot \aaa_h\|_{\infty,\OO} \Big]\| \etw \|_{0,\OO}\| \vv_h\|_{0,\OO}
	\\ \nonumber
	\le&\, C\, \Bigg\{  \sum_{T\in \TT_h} \Big[ \tau_T^{-1}\,  \|\etw\|_{0,T}^2
	+ \sigma\, \|\etw\|_{0,T}^2 \Big]
	+\sigma^{-1}\rho^2 \, \Big[\|\nabla \cdot \uu_{m}\|_{\infty,\OO}^2+\|\nabla \cdot \aaa_h\|_{\infty,\OO}^2\Big] \| \etw \|_{0,\OO}^2 \Bigg\}^{1/2} \times\\ \nonumber
	&\,  \Bigg\{  \sum_{T\in \TT_h} \Big[ \tau_T\, \| \LLh(\vv_h,q_h)\|_{0,T}^2 
	+\sigma\, \| \vv_h\|_{0,T}^2\Big]
	+\sigma\, \| \vv_h\|_{0,\OO}^2 \Bigg\}^{1/2}
	\\ \nonumber
	\le&\, C\, \Bigg\{  \sum_{T\in \TT_h} ( \sigma + \mu\, h_T^{-2})\,  \|\etw\|_{0,T}^2
	+\sigma^{-1}\rho^2 \, \Big[\|\nabla \cdot \uu_{m}\|_{\infty,\OO}^2+\|\nabla \cdot \aaa_h\|_{\infty,\OO}^2\Big] \| \etw \|_{0,\OO}^2 \Bigg\}^{1/2} \nh{(\vv_h,q_h)}
	\\ \nonumber
	\le&\, {C\, \Bigg(\sigma^{1/2}\, h + \mu^{1/2}  +\sigma^{-1/2}\rho \,h\, \Big[\|\nabla \cdot \uu_{m}\|_{\infty,\OO}+\|\nabla \cdot (\aaa_h-\aaa)\|_{\infty,\OO}\Big] \Bigg)}
	\, h^k \, 
	\| \ww\|_{k+1,\OO}
	\nh{(\vv_h,q_h)}\\
	\le&\, {C\, \Bigg(\sigma^{1/2}\, h + \mu^{1/2}  +\sigma^{-1/2}\rho \,h\, \Big[\|\nabla \cdot \uu_{m}\|_{\infty,\OO}+h^{1-d/2}\|\aaa\|_{2,\OO}\Big] \Bigg)}
	\, h^k \, 
	\| \ww\|_{k+1,\OO}
	\nh{(\vv_h,q_h)}
	\label{eq:5.12cc}
\end{align}

Next, from \eqref{eq:5.11.99} to \eqref{eq:5.12cc}, we have
\begin{align}
	\small
	\nonumber &A((\etw,\etp),(\vv_h,q_h)) + \Sconv(\etw;\vv_h) \\
	\le &\,  C\,  h^k\, \Bigg\{ 
	\left(\sigma^{1/2}\,h+ \mu^{1/2}+ \lambda^{1/2} +\sigma^{-1/2}\rho h\, 
	\Big[\|\nabla \cdot \uu_{m}\|_{\infty,\OO}
	{+h^{1-d/2}\|\aaa\|_{2,\OO}}\Big]
	\right) \,
	\|\ww\|_{k+1,\OO} \nonumber\\
	& + \lambda^{-1/2}\,
	\|p\|_{k,\OO}
	\Bigg\}\, \nh{(\vv_h,q_h)}.
	\label{eq:5.12}
\end{align}

Previous to bound the bilinear form $\Spress(\cdot;\cdot)$ given in \eqref{def-form-Spress}, we need the follow estimate. Using 
triangle inequality and \eqref{prop_tau1}-\eqref{prop_tau2}, we deduce that
\begin{align}
	\nonumber
	& \sum_{T \in \T_h}  \tau_T
	\, \left\| \LLh(\etw,\etp)  \right\|_{0,T}^2 \nonumber \\
	& \qquad =\,  \sum_{T \in \T_h}  \tau_T
	\, \left\| \rho\, (\nabla \uu_{m})\etw+ \rho\, (\nabla \etw)(\aaa_h+\uu_{m}) + \nabla \etp \right\|_{0,T}^2 
	\\ \nonumber
	& \qquad \le \, C\,  \sum_{T \in \T_h}  
	\Big[ \rho^2\tau_T \,  \|\nabla \uumh\|_{\infty,T}^2\|\etw\|_{0,T}^2+
	\tau_T\rho^2 \, \|\aaa_h+\uumh\|_{\infty,T}^2|\etw|_{1,T}^2+\tau_T  \, |\etp|_{1,T}^2
	\Big]\\ 
	& \qquad \le \,C\,  \sum_{T \in \T_h}  
	\Big[ \sigma\, \|\etw\|_{0,T}^2 + \sigma^{-1}  \rho^2\, \|\aaa_h+\uumh\|_{\infty,T}^2|\etw|_{1,T}^2+ \dfrac{\delta h^2_T}{\mu}   \, |\etp|_{1,T}^2
	\Big].
	%
	%
	%
	\label{eq:5.125}
\end{align}

Now, applying the Cauchy-Schwarz inequality, Lemma \ref{Cinv}, 
\eqref{pseudo-ortho-2} and \eqref{eq:5.125} we get 

\begin{align} 
	\nonumber &
	\Spress(\etw,\etp;\vv_h,q_h) \\ \nonumber
	=&\, \sum_{T \in \T_h}  \tau_T\, \Big(
	\sigma\, \etw  - \mu \, \Delta \etw + \LLh(\etw,\etp),
	- \sigma\, \vv_h +   \mu \, \Delta \vv_h + \LLh(\vv_h,q_h)
	\Big)_{T} \\ \nonumber
	\le&\, \sum_{T \in \T_h}  \tau_T \, \Bigg\{
	\sigma\, \|\etw\|_{0,T}  + \mu \, \| \Delta \etw \|_{0,T}+
	\|\LLh(\etw,\etp)  \|_{0,T}
	\Bigg\} \, 
	\|- \sigma\,\vv_h  + \mu \, \Delta \vv_h +\LLh(\vv_h,q_h)\|_{0,T}\\ \nonumber
	\le&\, C\, \Bigg\{\sum_{T \in \T_h}  \Bigg[
	\sigma^2\tau_T \, \|\etw\|_{0,T}^2  +
	\mu^2\tau_T \, \|  \etw \|_{2,T}^2+
	\|\LLh(\etw,\etp)  \|_{0,T}^2
	\Bigg] \Bigg\}^{1/2} \\ \nonumber
	&\times\,
	\Bigg\{\sum_{T \in \T_h} \tau_T\, 
	\| -\sigma\,\vv_h  + \mu \, \Delta \vv_h +\LLh(\vv_h,q_h)\|_{0,T}^2
	\Bigg\}^{1/2} \\ \nonumber
	\le&\small \, C\, \Bigg\{\sum_{T \in \T_h}  \Bigg[
	\sigma \, \|\etw\|_{0,T}^2  +
	\mu h_T^2 \, \|  \etw \|_{2,T}^2+
	\sigma^{-1}  \rho^2 \, \|\aaa_h+\uumh\|_{\infty,T}^2|\etw|_{1,T}^2+  \dfrac{\delta h^2_T}{\mu} \, |\etp|_{1,T}^2
	\Bigg] \Bigg\}^{1/2}\, \nh{(\vv_h,q_h)}\\ 
	%
	%
	%
	%
	\le &\,  {C\, h^k\, \Bigg\{ 
		\left( \sigma^{1/2}\,h+ \mu^{1/2}
		+\sigma^{-1/2}\rho \, \big(\|\aaa\|_{\infty,\OO} +\|\uu_{m} \|_{\infty,\OO} \big)\right)  \,
		\|\ww\|_{k+1,\OO}+\mu^{-1/2}\,
		\|p\|_{k,\OO}
		\Bigg\} \, \nh{(\vv_h,q_h)}, }\label{eq:5.14}
\end{align}
for all $(\vv_h,q_h) \in \HH_{h} \times Q_{h}$.\\
Hence, inserting \eqref{eq:5.12} and \eqref{eq:5.14} into definition of $\B$ given in \eqref{def-form-Bstab} gives
\begin{align}
	\nonumber &\,
	\B(\etw,\etp;\vv_h,q_h) \\  \nonumber
	\le &\,  C\,  h^k\, \Bigg\{
	\Big(
	\sigma^{1/2}\,h+ \mu^{1/2}+ \lambda^{1/2}+
	\sigma^{-1/2}
	\rho\, \Big[  
	h \, \Big(  \|\nabla \cdot \uu_{m}\|_{\infty,\OO}+ {h^{1-d/2}\|\aaa\|_{2,\OO}} \Big)  +\\ 
	&\, 
	\|\aaa_h+\uu_{m} \|_{\infty,\OO} 
	\Big]
	\Big) \,   \|\ww\|_{k+1,\OO}+
	\Big(\lambda^{-1/2} 
	+\mu^{-1/2}
	\Big) \,
	\|p\|_{k,\OO} \nonumber
	\Bigg\} \nh{(\vv_h,q_h)} \\ 
	= &\,  C\,  h^k\, \Bigg\{
	\Big(
	C_1 +
	\sigma^{-1/2}
	\rho\,  
	h \, \Big[  \|\nabla \cdot \uu_{m}\|_{\infty,\OO}+ {h^{1-d/2}\|\aaa\|_{2,\OO}} \Big] 
	\Big) \,   \|\ww\|_{k+1,\OO}+
	M_2 \,
	\|p\|_{k,\OO}
	\Bigg\} \nh{(\vv_h,q_h)}, \label{eq:5.15}
\end{align}
for all $(\vv_h,q_h) \in \HH_{h} \times Q_{h}$.

Now, from Lemma \ref{coercivity_lemma}, Lemma \ref{pseudo-orthogo} and \eqref{eq:5.15}, we conclude that
\begin{align}
	\nonumber
	&\, C_B \, \nh{(\ew,\ep)}^2 \\
	\nonumber
	\le&\, \B((\ew,\ep),(\ew,\ep))\\
	\nonumber
	=&\,  \B((\etw,\etp),(\ew,\ep)) - \B((\ww-\ww_h,p-p_h),(\ew,\ep))  \\
	\nonumber
	\le &\, C\,  h^k\, \Bigg\{ \Bigg(
	\sigma^{1/2}\,h+ \mu^{1/2}+ \lambda^{1/2}+ 
	\sigma^{-1/2}
	\rho\, \Big[  
	h \, \Big(  \|\nabla \cdot \uu_{m}\|_{\infty,\OO}+ {h^{1-d/2}\|\aaa\|_{2,\OO}}  \Big)  +\\ 
	&\, 
	\|\aaa\|_{\infty,\OO}+\|\uu_{m} \|_{\infty,\OO} + \|\aaa\|_{k+1,\OO}
	\Big]
	\Bigg) \,   \|\ww\|_{k+1,\OO}+\Bigg( \lambda^{-1/2} +\mu^{-1/2} \Bigg)\,
	\|p\|_{k,\OO}
	\Bigg\} \, \nh{(\ew,\ep)} \nonumber \\
	\nonumber
	= &\, C\,  h^k\, \Bigg\{ \Bigg(
	C_1+ 
	\sigma^{-1/2}
	\rho\, \Big[  
	h \, \Big(  \|\nabla \cdot \uu_{m}\|_{\infty,\OO}+{h^{1-d/2}\|\aaa\|_{2,\OO}} \Big)  + 
	\|\aaa\|_{k+1,\OO}
	\Big]
	\Bigg) \,   \|\ww\|_{k+1,\OO} \nonumber \\ 
&	+M_2\,	\|p\|_{k,\OO}
	\Bigg\} \, \nh{(\ew,\ep)} \nonumber \\
	= &\, C\,  h^k\, \Bigg\{ M_1  \,   \|\ww\|_{k+1,\OO} + M_2\, \|p\|_{k,\OO}
	\Bigg\} \, \nh{(\ew,\ep)}.
	\label{eq:5.16}
\end{align}
In the same manner, from {\eqref{hip3-lemma3}}, the interpolation properties \eqref{LAGRANGE1}-\eqref{LAGRANGE2} and \eqref{eq:5.125}, we can see that
\begin{align}
	\nonumber
	&   \nh{(\etw,\etp)}^2 \\
	\nonumber
	\le&\,
	\sigma\, \|\etw\|_{0,\OO}^2+\mu\, | \etw|_{1,\OO}^2 + \lambda\, \| \nabla \cdot \etw \|_{0,\OO}^2+  \sum_{T \in \T_h}  \tau_T  \,  \| \LLh(\etw,\etp)\|^2_{0,T}   \\
	%
	%
	%
	\nonumber
	\le &\, C\, \Bigg\{
	\sigma\, \|\etw\|_{0,\OO}^2+\mu\, | \etw|_{1,\OO}^2+ \lambda\, \| \nabla \cdot \etw \|_{0,\OO}^2+
	\sum_{T \in \T_h} \Bigg[ \sigma^{-1}\rho^2  \, \|\aaa_h+\uu_{m}\|^2_{\infty,T}  |  \etw|_{1,T}^2 +
	\dfrac{\delta h^2_T}{\mu}  \,  |  \etp |^2_{1,T}   \Bigg]
	\Bigg\} \\ \nonumber
	\nonumber
	\le &\, C\, \Bigg\{
	\sigma\, \|\etw\|_{0,\OO}^2+\left(\mu+\lambda+ \sigma^{-1}\rho^2  \, \|\aaa_h+\uu_{m} \|^2_{\infty,\OO}\right)
	\, | \etw|_{1,\OO}^2+ 
	\sum_{T \in \T_h} \dfrac{\delta h^2_T}{\mu}  \,  |  \etp |^2_{1,T}  
	\Bigg\} \\ \nonumber
	\le&\,C\,  h^{2k}\, \Bigg\{ 
	\left(\sigma h^2 + \mu
	+ \lambda+ \sigma^{-1}\rho^2  \, \|\aaa_h+\uu_{m} \|^2_{\infty,\OO} 
	\right) \,
	\|\ww\|_{k+1,\OO}^2+
	\mu^{-1}
	\,
	\|p\|^2_{k,\OO}
	\Bigg\}  \\\nonumber
	=&\,C\,  h^{2k}\, \Bigg\{ 
	C_1^2 \,
	\|\ww\|_{k+1,\OO}^2+
	\mu^{-1}
	\,
	\|p\|^2_{k,\OO}
	\Bigg\} ,
\end{align}
and combining the last estimate with \eqref{eq:5.16}, the result follows.
\end{proof}

\begin{remark}
{The error estimate presented in the previous result is not robust with
	respect to the viscosity $\mu$.  One alternative to avoid this dependence is to make a stronger
	assumption on the pressure.  In fact,  if we suppose that
	$p\in H^{k+1}(\OO)\cap L^2_0(\OO)$,  then we get
	\[ 
	\sum_{T\in \TT_h} \tau_T \,  |  \etp |^2_{1,T}   \le C\, 
	\sum_{T\in \TT_h} \sigma^{-1}\,  |  \etp |^2_{1,T}  \le C\, \sigma^{-1}\,h^{2k}\,  \|  p \|^2_{k+1,\OO},
	\]
	and following the same steps as to obtain \eqref{main-estimate} we obtain
	the following robust estimate}
\begin{equation}
	\nh{(\ww-\ww_h,p-p_h)} 
	\le \, 
	C\,
	h^k\, \Bigg\{ 
	M_1\,  \|\ww\|_{k+1,\OO}+
	\Big(\lambda^{-1/2} +\sigma^{-1/2}
	\Big) \, 
	\|p\|_{k+1,\OO}
	\Bigg\}.
\end{equation}
\end{remark}

%
%

\begin{corollary}
Assume  the assumptions of Theorem \ref{convergencia}.
Then, there is a positive constant $C$ independent of $h$, such that
\begin{equation}
	\|p-p_h\|_{0,\OO}  \le \, C\, h^k\, 
	\Bigg\{
	\Big( 
	M_1M_3+M_4\Big) \,   \|\ww\|_{k+1,\OO}+
	M_2M_3 \, 
	\|p\|_{k,\OO}
	\Bigg\},
\end{equation}
where
\begin{equation}
	\label{M3}
	\begin{split}
		M_3 \defi & \; \sigma^{1/2} + C_1 + \rho\, \Big( h^2 \mu^{-1} \sigma^{1/2}+h\mu^{-1/2}\Big)\, (   \|\aaa\|_{\infty,\OO}+\|\uu_m\|_{\infty,\OO} ) \\
		& + \sigma^{-1/2}\rho\,  \Big( \|\nabla \cdot \uu_{m}\|_{\infty,\OO}+ {h^{1-d/2}\|\aaa\|_{2,\OO}} \Big),
	\end{split}
\end{equation}
and
\begin{equation}
	\label{M4}
	{M_4 \defi  \Bigg[ \rho  +\rho  \sigma^{-1/2}   \mu^{1/2}+ \sigma^{-1}\rho^2 \, (   \|\aaa\|_{\infty,\OO}+\|\uu_m\|_{\infty,\OO} ) 
			\Bigg] \, \|\aaa\|_{k+1,\OO} }.
	\end{equation}
\end{corollary}
%
%
\begin{proof}
	%
	%
	As is well known, given $p-p_h \in Q$, there exist 
	$\vvt \in \HH$ and a positive constant $C$   (see \cite{BGHRR24}) such that
	\begin{equation}
		\label{eq:5.32}
		\nabla \cdot \vvt=p-p_h \qquad  \text{ and } \qquad  \|\vvt\|_{1,\OO}\le C\|p-p_h\|_{0,\OO}.
	\end{equation}
	
	Furthermore, let  $\vv_h \defi \mathcal{C}_h \vvt$, where 
	$\mathcal{C}_h \vvt\in \HH_h$ is the Cl\'ement interpolate  of $\vvt$ (cf. to scalar case and the properties \eqref{LAGRANGE2}-\eqref{eq:5.38}). 
	Applying integration by parts and the fact that $\vvt$ and $\vv_h$ vanish on $\partial \OO$, we get
	{\begin{equation}
			\|p-p_h\|_{0,\OO}^2
			=\, (p-p_h, \nabla \cdot \vvt)
			%
			%
			%
			%
			%
			%
			=\, -(\nabla(p-p_h),\vvt-\vv_h) + (p-p_h,\nabla \cdot \vv_h) 
			%
			%
			%
			%
			%
			%
			= I_3 + I_{4}.
			\label{eq:5.39}
	\end{equation}}
	To bound $I_3$, {we first notice that}
	\begin{align}
		\nonumber
		&\, 
		\sum_{T\in \T_h} \tau_T\, |p-p_h|_{1,T}^2 \\
		\nonumber
		\le &\, C\,  
		\sum_{T\in \T_h} \tau_T\,  \Big[\|\LLh(\ww-\ww_h,p-p_h)\|_{0,T}^2+
		\rho^2\, \| (\nabla \uu_m)(\ww-\ww_h) + \nabla(\ww-\ww_h)(\uu_m+ \aaa_h)   \|_{0,T}^2
		\Big]
		\\
		\nonumber
		\le &\small \, C\, 
		\sum_{T\in \T_h} \tau_T\,  \Big[\|\LLh(\ww-\ww_h,p-p_h)\|_{0,T}^2+
		\rho^2\, \| \nabla \uu_m\|_{\infty,\OO}^2\| \ww-\ww_h \|_{0,T}^2+
		\rho^2\,\|\uu_m+ \aaa_h   \|_{\infty,T}^2 |\ww-\ww_h|_{1,T}^2 
		\Big]
		\\
		%
		%
		%
		%
		%
		\nonumber
		\le &\, C\,  
		\nh{(\ww-\ww_h,p-p_h)}^2+
		\rho^2\,\|\uu_m+ \aaa_h   \|_{\infty,\OO}^2
		\sum_{T\in \T_h} \tau_T\,   |\ww-\ww_h|_{1,T}^2  \\
		%
		%
		%
		%
		%
		%
		\le &{\, C\,  \Bigg\{
			1+\rho^2{\mu^{-2}}h^2\,\|\uu_m+ \aaa_h   \|^2_{\infty,\OO}
			\Bigg\}\,
			\nh{(\ww-\ww_h,p-p_h)}^2} ,
		\label{eq:5.40alt2}
	\end{align}
	and that,  
	{using the definition of $\tau_T$ and \eqref{LAGRANGE2} we obtain
		\begin{equation}
			\sum_{T\in \T_h} \tau_T^{-1}\, \|\vvt-\vv_h\|_{0,T}^2
			%
			%
			%
			%
			%
			\le  C\,  
			\Bigg\{ 
			\sigma h^2 +  \mu
			\Bigg\}\,
			\| \vvt\|_{1,\OO}^2.
			\label{eq:5.40alt3}
	\end{equation}}
	From these two facts, we conclude that
	%
	%
	%
	%
	%
	\begin{align}
		I_3
		\nonumber
		=&\,- (\nabla(p-p_h),\vvt-\vv_h)  \\
		\le &\, C\,  
		\Bigg\{ 
		\sigma^{1/2}h +  \mu^{1/2}+ 
		\rho h \mu^{-1/2} \, (\sigma^{1/2}\mu^{-1/2}h+1)\,\|\uu_m+ \aaa_h   \|_{\infty,\OO}
		\Bigg\}\,
		\nh{(\ww-\ww_h,p-p_h)} \,
		\| \vvt\|_{1,\OO}.
		\label{eq:5.40}
	\end{align}
	%
	%
	%
	%
	%
	%
	%
	%
	%
	%
	On the other hand, from \eqref{pseudo-ortho-1}, we can deduce  that
	\begin{align}
		\nonumber
		&\, A(\ww-\ww_h,p-p_h;\vv_h,0) + \Sconv(\ww-\ww_h;\vv_h)+\Spress(\ww-\ww_h,p-p_h;\vv_h,0) \\ \nonumber
		\nonumber =&\, 
		\dfrac{\rho}{2} \, ( \nabla\cdot(\aaa_h-\aaa)\ww,\vv_h)+ 
		\rho\, ( \nabla \ww(\aaa_h-\aaa),\vv_h)+
		\nonumber\\  
		 &\, 		\sum_{T \in \T_h}  \tau_T \,
		\Bigg( 
		\rho\, (\nabla \ww)(\aaa_h-\aaa) ,
		-\sigma\, \vv_h +   \mu \, \Delta \vv_h + \LLh(\vv_h,0) 
		\Bigg)_T,
		\nonumber
	\end{align}
	and as consequence, we get 
	\begin{align}
		\nonumber
		&\sigma\, (\ww-\ww_h,\vv_h) +  \mu \,(\nabla (\ww-\ww_h),\nabla \vv_h) +
		\rho\, ((\nabla \uu_{m})(\ww-\ww_h),\vv_h) +
		\\
		\nonumber
		&\,  \lambda\, (\nabla \cdot (\ww-\ww_h), \nabla\cdot \vv_h)  - (p-p_h,\nabla \cdot \vv_h)
		+ \Sconv(\ww-\ww_h;\vv_h)+\Spress(\ww-\ww_h,p-p_h;\vv_h,0) 
		\\ \nonumber
	  =&\, 
		\dfrac{\rho}{2} \, ( \nabla\cdot(\aaa_h-\aaa)\ww,\vv_h)+ 
		\rho\, ( \nabla \ww(\aaa_h-\aaa),\vv_h) \nonumber \\ 
		& + \sum_{T \in \T_h}  \tau_T \,
		\Bigg( 
		\rho\, (\nabla \ww)(\aaa_h-\aaa) ,
		-\sigma\, \vv_h +   \mu \, \Delta \vv_h + \LLh(\vv_h,0) 
		\Bigg)_T.
		\nonumber
	\end{align}
	This last allows us to prove that
	\begin{align}
		\nonumber
		I_{4}
		=&\, (p-p_h,\nabla \cdot \vv_h)\\
		%
		%
		%
		%
		%
		%
		%
		\nonumber
		=&\,\small  \sigma\, (\ww-\ww_h,\vv_h) +  \mu \,(\nabla (\ww-\ww_h),\nabla \vv_h) +
		\rho\, ((\nabla \uu_{m})(\ww-\ww_h),\vv_h) + \lambda\, (\nabla \cdot (\ww-\ww_h), \nabla\cdot \vv_h)  +
		\\
		\nonumber
		&\, \small
		\rho\, (\nabla (\ww-\ww_h)(\aaa_h+\uu_{m}),\vv_h) + \dfrac{\rho}{2} \, \left( (\nabla \cdot \aaa_h)\, (\ww-\ww_h) , \vv_h\right)-\\
		\nonumber
		&\, \small
		\dfrac{\rho}{2} \, ( \nabla\cdot(\aaa_h-\aaa)\ww,\vv_h)-
		\rho\, ( (\nabla \ww)(\aaa_h-\aaa),\vv_h)+
		\\ 
		\nonumber &\,\small
		\sum_{T \in \T_h}  \tau_T \,
		\Bigg( 
		\sigma\, (\ww-\ww_h)- \mu\, \Delta (\ww-\ww_h) +\LLh(\ww-\ww_h,p-p_h)
		, -\sigma\, \vv_h +   \mu \, \Delta \vv_h + \LLh(\vv_h,0) 
		\Bigg)_T \nonumber -\\ \nonumber
		&\,\small
		\sum_{T \in \T_h}  \tau_T \,
		\Bigg( 
		\rho\, (\nabla \ww)(\aaa_h-\aaa) ,
		-\sigma\, \vv_h +   \mu \, \Delta \vv_h + \LLh(\vv_h,0) 
		\Bigg)_T \\
		%
		%
		%
		=&\, I_5 + I_6+I_7+I_8.
		\label{eq:5.26}
	\end{align}
	In order to bound the first term on \eqref{eq:5.26}, we use integration by parts, Cauchy-Schwarz and H\"older inequalities and \eqref{eq_may:2.5}, we get
	\begin{align}
		\nonumber
		I_5 
		%
		%
		%
		=&\, \small\sigma\, (\ww-\ww_h,\vv_h) +  \mu \,(\nabla (\ww-\ww_h),\nabla \vv_h) +
		\rho\, ((\nabla \uu_{m})(\ww-\ww_h),\vv_h) 
		\\
		\nonumber
		&\, \small
		+ \lambda\, (\nabla \cdot (\ww-\ww_h), \nabla\cdot \vv_h)  +\rho\, ((\nabla (\ww-\ww_h))(\aaa_h+\uu_{m}),\vv_h) + \dfrac{\rho}{2} \, \left( (\nabla \cdot \aaa_h)\, (\ww-\ww_h) , \vv_h\right)\\
		\nonumber 
		=&\, \small\sigma\, (\ww-\ww_h,\vv_h) +  \mu \,(\nabla (\ww-\ww_h),\nabla \vv_h) +
		\rho\, ((\nabla \uu_{m})(\ww-\ww_h),\vv_h) + \lambda\, (\nabla \cdot (\ww-\ww_h), \nabla\cdot \vv_h)  -
		\\
		\nonumber
		&\, \small
		\rho\, ((\nabla \vv_h)(\aaa_h+\uu_{m}),\ww-\ww_h) 
		- \rho \, (\nabla \cdot (\aaa_h+\uu_{m}) \vv_h,\ww-\ww_h) +
		\dfrac{\rho}{2} \, \left( (\nabla \cdot \aaa_h)\, (\ww-\ww_h) , \vv_h\right)\\
		%
		%
		%
		%
		%
		%
		\nonumber
		\le&\, \small C\, \Bigg\{
		\sigma^2 \, \|\ww-\ww_h\|_{0,\OO}^2 +  \mu^2 \,|\ww-\ww_h|_{1,\OO}^2+
		\rho^2\, \|\nabla \uu_{m}\|_{\infty,\OO}^2\|\ww-\ww_h\|_{0,\OO}^2 +\\
		\nonumber
		&\, \small
		\lambda^2\, \| \nabla \cdot (\ww-\ww_h)\|_{0,\OO}^2  +
		\rho^2\,  \|\aaa_h+\uu_{m}\|_{\infty,\OO}^2\|\ww-\ww_h\|_{0,\OO}^2+
		\rho^2 \, \|\nabla \cdot (\aaa_h+\uu_{m})\|_{\infty,\OO}^2\|\ww-\ww_h\|_{0,\OO}^2 +\\
		\nonumber
		&\,   \small \rho^2 \, \|\nabla \cdot \aaa_h\|_{\infty,\OO}^2  \|\ww-\ww_h\|_{0,\OO}^2
		\Bigg\}^{1/2}\, \| \vv_h\|_{1,\OO}\\ 
		%
		%
		%
		%
		%
		\le&\, \small C\, C_2\, \nh{(\ww-\ww_h,p-p_h)}\, \| \vv_h\|_{1,\OO},
		\label{eq:5.27}
	\end{align}
	where \[C_2 \defi  \sigma^{1/2}+  \mu^{1/2} +
	\lambda^{1/2}+
	\sigma^{-1/2}\rho\,  \|\aaa_h+\uu_{m}\|_{\infty,\OO}+
	\sigma^{-1/2}\rho\,  \, \|\nabla \cdot \uu_{m}\|_{\infty,\OO}+
	\sigma^{-1/2}\rho\,  {h^{1-d/2}\|\aaa\|_{2,\OO}} .\]
	Now, to bound $I_6$, using H\"older inequality (with $L^2,L^4,L^4$ or $L^4,L^4,L^2$) and \eqref{prop_aaa_1}, we have
	\begin{align}
		\nonumber
		I_{6} =&\,
		-  \dfrac{\rho}{2} \, ( \nabla\cdot(\aaa_h-\aaa)\ww,\vv_h)-
		\rho\, ( \nabla \ww(\aaa_h-\aaa),\vv_h)\\
		\nonumber
		\le&\, C\, \Bigg\{
		\rho \, \| \nabla\cdot(\aaa_h-\aaa)\|_{0,\OO} \|\ww\|_{1,\OO} \|\vv_h\|_{1,\OO}+
		\rho\, \|\aaa-\aaa_h\|_{1,\OO}\|\ww\|_{2,\OO}\|\vv_h\|_{1,\OO}\Bigg\}\\ 
		%
		%
		%
		%
		%
		%
		\le&\, C \rho \, h^k\, 
		\| \aaa\|_{k+1,\OO} \, \|\ww\|_{2,\OO}\|\vv_h\|_{1,\OO}.
		%
		%
		%
		\label{eq:5.28}
	\end{align}
	
	On the other hand, to bound $I_7$, we need to decompose it in the follow way:
	\begin{align}
		\nonumber
		\nonumber
		I_7=&\,\small     \sum_{T \in \T_h}  \tau_T \,
		\Bigg( 
		\sigma\, (\ww-\ww_h)- \mu\, \Delta (\ww-\ww_h) +\LLh(\ww-\ww_h,p-p_h)
		, -\sigma\, \vv_h +   \mu \, \Delta \vv_h + \LLh(\vv_h,0) 
		\Bigg)_T  \nonumber \\
		\nonumber
		=&\,  \small \sum_{T \in \T_h}  \tau_T \,
		\Bigg( 
		\sigma\, (\ww-\ww_h) +\LLh(\ww-\ww_h,p-p_h)
		, -\sigma\, \vv_h +   \mu \, \Delta \vv_h + \LLh(\vv_h,0) 
		\Bigg)_T  \\ \nonumber
		&\,\small -
		\sum_{T \in \T_h}  \tau_T \,
		\Bigg( 
		\mu\, \Delta (\ww-\ww_h)
		, -\sigma\, \vv_h +   \mu \, \Delta \vv_h + \LLh(\vv_h,0) 
		\Bigg)_T 
		\nonumber \\
		%
		%
		%
		=&\, I_9+ I_{10}.
		\label{eq:5.29}
	\end{align}
	%
	%
		%
		Before that we continue bounding these new terms, we need to note that using \eqref{assum1d},  \eqref{prop_tau1}-\eqref{prop_tau2}
		and  Lemma \ref{Cinv}, we have
		\begin{align}
			\sum_{T \in \T_h}  \tau_T
			\, \left\| -\sigma\, \vv_h + \mu\, \Delta \vv_h + \LLh(\vv_h,0)  \right\|_{0,T}^2 
			%
			%
			%
			%
			%
			%
			\le&\, C\, 
			\Bigg\{ \sigma +  \mu+ \sigma^{-1}\rho^2 \, \|\aaa_h+\uu_m\|_{\infty,\OO}^2
			\Bigg\}\, \| \vv_h\|_{1,\OO}^2.\label{eq:5.12ccc}
		\end{align}
		%
		Next, using Cauchy-Schwarz inequality, \eqref{assum1d} and \eqref{eq:5.12ccc}, we have
		\begin{align}
			\nonumber
			\nonumber
			I_9
			=&\,   \sum_{T \in \T_h}  \tau_T \,
			\Bigg( 
			\sigma\, (\ww-\ww_h) +\LLh(\ww-\ww_h,p-p_h)
			, -\sigma\, \vv_h +   \mu \, \Delta \vv_h + \LLh(\vv_h,0) 
			\Bigg)_T  \\ 
			%
			%
			%
			%
			%
			%
			\le &\, C\, 
			\Bigg\{ \sigma^{1/2} +  \mu^{1/2}+ \sigma^{-1/2}\rho \, \|\aaa_h+\uu_m\|_{\infty,\OO}
			\Bigg\}\,  
			\nh{(\ww-\ww_h,p-p_h)} \, \| \vv_h\|_{1,\OO}.
			\label{eq:5.30}
		\end{align}
		To estimate $I_{10}$, we use the interpolation operator $\II_h \ww$ and the definition of $\etw$ and $\ew$ given in \eqref{def-eta-eh} join to \eqref{eq:5.12ccc} with Lemma \ref{Cinv}, and  we obtain
		\begin{align}
			\nonumber
			&\small I_{10}= \\\, -
			& \nonumber \small \sum_{T \in \T_h}  \tau_T \,
			\Bigg( 
			\mu\, \Delta (\ww-\ww_h)
			, -\sigma\, \vv_h +   \mu \, \Delta \vv_h + \LLh(\vv_h,0) 
			\Bigg)_T\\  \nonumber
			%
			%
			%
			%
			%
			%
			\le&\small \, C\, 
			\Bigg\{ \sigma^{1/2} +  \mu^{1/2}+ \sigma^{-1/2}\rho \, \|\aaa_h+\uu_m\|_{\infty,\OO}
			\Bigg\}\,  
			\Bigg\{ \sum_{T \in \T_h} 
			\mu  h_T^2 \, \Big[ \| \Delta \etw\|^2_{0,T}  + \| \Delta \ew\|^2_{0,T} \Big]
			\Bigg\}^{1/2}\, \|\vv_h\|_{1,\OO}
			\nonumber \\
			%
			%
			%
			%
			%
			%
			\le& \small \, C\, 
			\Bigg\{ \sigma^{1/2} +  \mu^{1/2}+ \sigma^{-1/2}\rho \, \|\aaa_h+\uu_m\|_{\infty,\OO}
			\Bigg\}\,  
			\Bigg\{ \sum_{T \in \T_h} \, \Big[ 
			\mu  h_T^2\, \| \etw\|^2_{2,T}  +\mu\, |\ww-\ww_h |^2_{1,T}  \nonumber \\ 
			& \small +\mu\, |\etw|^2_{1,T} \Big]
			\Bigg\}^{1/2}\, \|\vv_h\|_{1,\OO}
			\nonumber \\
			%
			%
			%
			%
			%
			%
			\le& \small \, C\, 
			\Bigg\{ \sigma^{1/2} +  \mu^{1/2}+ \sigma^{-1/2}\rho \, \|\aaa_h+\uu_m\|_{\infty,\OO}
			\Bigg\}\,  
			\Bigg\{ 
			\mu^{1/2}  h^{k}\, \| \ww\|_{k+1,\OO}  +\nh{(\ww-\ww_h,p-p_h)} 
			\Bigg\} \, \|\vv_h\|_{1,\OO}.
			\nonumber \\
			%
			%
			%
			\label{eq:5.31}
		\end{align}
		We may now to bound the last term on \eqref{eq:5.26}, using  Cauchy-Schwarz and H\"older inequalities (with $L^4,L^4,L^2$), \eqref{eq:5.12ccc}  and we obtain
		\begin{align}
			\nonumber
			\nonumber
			&I_8 =\\
			&\,   \small \sum_{T \in \T_h}  \tau_T \,
			\Bigg( \rho\, (\nabla \ww)(\aaa-\aaa_h)
			, -\sigma\, \vv_h +   \mu \, \Delta \vv_h + \LLh(\vv_h,0) 
			\Bigg)_T  \\ \nonumber
			%
			%
			%
			%
			%
			%
			\le &\,  \small    \rho\, \left\{\sum_{T \in \T_h}  \tau_T \, \|\aaa-\aaa_h\|_{0,4,T}^2\|\nabla \ww\|_{0,4,T}^2 \right\}^{1/2}
			\left\{\sum_{T \in \T_h}  \tau_T \, \| -\sigma\, \vv_h +   \mu \, \Delta \vv_h + \LLh(\vv_h,0) \|_{0,T}^2\right\}^{1/2}
			\\ \nonumber
			\le &\,   \small  \rho \sigma^{-1/2} \, \left\{\sum_{T \in \T_h}  \|\aaa-\aaa_h\|_{0,4,T}^4 \right\}^{1/4}\,
			\, \left\{\sum_{T \in \T_h} \|\nabla \ww\|_{0,4,T}^4 \right\}^{1/4}\,
			\left\{\sum_{T \in \T_h}  \tau_T \, \| -\sigma\, \vv_h +   \mu \, \Delta \vv_h + \LLh(\vv_h,0) \|_{0,T}^2\right\}^{1/2}
			\\ \nonumber
			\le &\,\small C\,      \rho \,\sigma^{-1/2} \,\|\ww\|_{2,\OO}  \|\aaa-\aaa_h\|_{1,\OO}
			\Bigg\{    \sum_{T \in \T_h} \tau_T \| -\sigma\, \vv_h +   \mu \, \Delta \vv_h + \LLh(\vv_h,0) \|_{0,T}^2 \Bigg\}^{1/2}
			\\ 
			%
			%
			%
			%
			%
			%
			\le &\,\small C\,  h^{k}\,  \Bigg\{ \rho  +\rho  \sigma^{-1/2}   \mu^{1/2}+ \sigma^{-1}\rho^2 \, \|\aaa_h+\uu_m\|_{\infty,\OO}
			\Bigg\} \,  \|\ww\|_{2,\OO} \|\aaa\|_{k+1,\OO} \|\vv_h\|_{1,\OO}.
			\label{eq:5.310}
		\end{align}
		%
		%
		%
		Thus, from \eqref{eq:5.26}-\eqref{eq:5.29} and \eqref{eq:5.30}-\eqref{eq:5.310}, we get 
		%
		\begin{equation}
			I_4
			\le \,  C\, \Bigg\{
			C_2\,\nh{(\ww-\ww_h,p-p_h)}+ h^{k}\, M_4  \|\ww\|_{2,\OO}  \Bigg\}\, \|\vv_h\|_{1,\OO}.
			\label{eq:5.31b}
		\end{equation}
		%
		%
		%
		%
		%
		%
		%
		%
		%
		Next, connecting \eqref{eq:5.40} and  \eqref{eq:5.31b} with  \eqref{eq:5.39},  we can deduce that
		%
		%
		%
		%
		%
		\begin{equation}
			\|p-p_h\|_{0,\OO}^2 
			\le \,  C\, \Bigg\{
			M_3 \,\nh{(\ww-\ww_h,p-p_h)}+
			M_4\, 
			h^{k} \,  \|\ww\|_{2,\OO} 
			\Bigg\}\,  \|\vvt\|_{1,\OO}.
			\label{eq:5.41}
		\end{equation}
		%
		%
		%
		%
		%
		%
		%
		Finally, using estimate \eqref{eq:5.32} in \eqref{eq:5.41}  with Theorem \ref{convergencia} and H\"older inequality ($2ab \le a^2+b^2$) the result follow.
		%
		%
		%
		%
		%
		%
	\end{proof}
\section{Numerical experiments}\label{Sec:Numerics}

In this section, we present three series of numerical experiments in order to illustrate the expected results.
The numerical routines have been built using the open-source finite element libraries FEniCS \cite{alnaes_fenics_2015}.
The legends of the different convergence curves follow the following notation:

\[ e_0(p) \defi \Vert p-p_h \Vert_{0,\OO},\qquad  e_1(\ww) \defi \vert \ww-\ww_h \vert_{1,\Omega} \qquad \text{and}\qquad  e_0(\ww) \defi \Vert \ww-\ww_h \Vert_{0,\Omega}.\]

We first test the convergence of the stabilized method  by means of an academic problem with a smooth analytical
solution. Next, we test the robustness of the pressure recovery process with respect to perturbations on the data. Finally, in the third example, we recover the velocity and reconstruct
the pressure from a Navier-Stokes simulation.

\subsection{\bf Experiment 1: {The Kovasznay flow solution}}

In this case we have considered $\OO \defi (-1/2,3/2)\times (0,2)$ and we choose the data 
$\ff$  and $g$ so that the unknowns are defined as follows:
\begin{align*}
	&\ww = \begin{pmatrix}  1-e^{\zeta x}\cos(2\pi y) \\ \frac{\zeta}{2\pi} e^{\zeta x}\sin(2\pi y) \end{pmatrix}, \qquad p = \frac{1}{2}e^{2\zeta x} -\frac{1}{8\zeta} (e^{3\zeta}-e^{-\zeta}) , \qquad \uu =  \begin{pmatrix} x \\ -y  \end{pmatrix},
\end{align*}
where $\zeta \defi \displaystyle \frac{1}{2\mu} \sqrt{\frac{1}{4\mu^2}+4\pi^2}${, and $\uu_m = \uu-\ww$}.  For this experiment, we are assessing our scheme by using different values for $\mu$, which are $1, 10^{-1}, 10^{-2},$ and $10^{-3}$.  {Moreover, we set $\rho=\sigma=1.0$, and the stabilization parameters are given by $\lambda=0.5$ and $\delta=0.001$.} 

Recall that $\aaa$ represents an approximation of the noise $\ww$ at a previous time step (cf. Remark \ref{remark2}). For this reason, we start by using $\aaa=0.9\,\ww$ (and so $\aaah = 0.9 \, \II_h(\ww)$)}, for different { approximation} degrees {$k$}. The error curves are depicted in Figure~\ref{curve:ex1-case1}, where we observe that all  error norms exhibit the expected convergence rates.
Next, we extend the scheme to address the model problem (P) with $\aaa = \displaystyle  \ww$. This nonlinear differential problem is referred to as the perturbed Navier–Stokes problem, and it is solved by extending the stabilized formulation \eqref{stab} as follows: Find $(\ww_h,p_h)\in \boldsymbol{H}_h\times Q_h$ such that
\begin{equation}\label{stab:NS}
\begin{split}
	&A(\ww^{}_h,p^{}_h;\vv_h,q_h) + \rho\, ((\nabla \ww^{}_h)(\ww^{}_h+\uu_{m}),\vv_h) + \dfrac{\rho}{2} \, \left( (\nabla \cdot \ww^{}_h)\, \ww^{}_h , \vv_h\right)
	\\ 
	&	+	\sum_{T \in \T_h}  \tau_T \, \Big(
	\sigma\, \ww^{}_h  - \mu \, \Delta \ww^{}_h + \LLh^{\boldsymbol{n}}(\ww^{}_h,\ww^{}_h, p_h),
	\, - \sigma\, \vv_h +   \mu \, \Delta \vv_h + \LLh^{\boldsymbol{n}}(\vv_h,\ww^{}_h,q_h)
	\Big)_{T} \\
	=& (\ff,\vv_h)+
	(q_h,g) + \lambda\, (g,\nabla \cdot \vv_h)+ \sum_{T \in \T_h}  \tau_T \, \Big(\ff , 
	-\sigma\, \vv_h +   \mu \, \Delta \vv_h + \LLh^{\boldsymbol{n}}(\vv_h, \ww^{}_h, q_h) \Big)_{T},
\end{split}
\end{equation}
where $A(\cdot; \cdot)$ is defined as in  \eqref{def-form-A}, and
\[\LLh^{\boldsymbol{n}} (\ww_h, \zz_h, p_h) \defi \rho\, (\nabla \uu_{m})\ww_h + \rho\, (\nabla \ww_h)(\zz_h+\uu_{m}) + \nabla p_h\,.\]
This nonlinear problem is solved
using Picard's scheme, and the iterations are stopped when the error between two consecutive approximations is smaller than $\textrm{tol} = 10^{-6}$. The results are displayed in Figure \ref{curve:ex1-case2}, where we observe once again that the errors 
{exhibit the expected orders of convergence.} One further interesting remark to be made is that all the errors show a decreasing behavior as the viscosity $\mu$ decreases.  A similar pattern has been observed before (see, e.g.,  \cite{carber2023}) for different pressure recovery strategies.  To our best knowledge, the mathematical explanation of this fact remains an open question.

\begin{figure}[H]
\centering
\begin{tabular}{ccc}
	$k=1$ & $k=2$ & $k=3$\\
	\includegraphics[width=4.5cm,height=4.0cm]{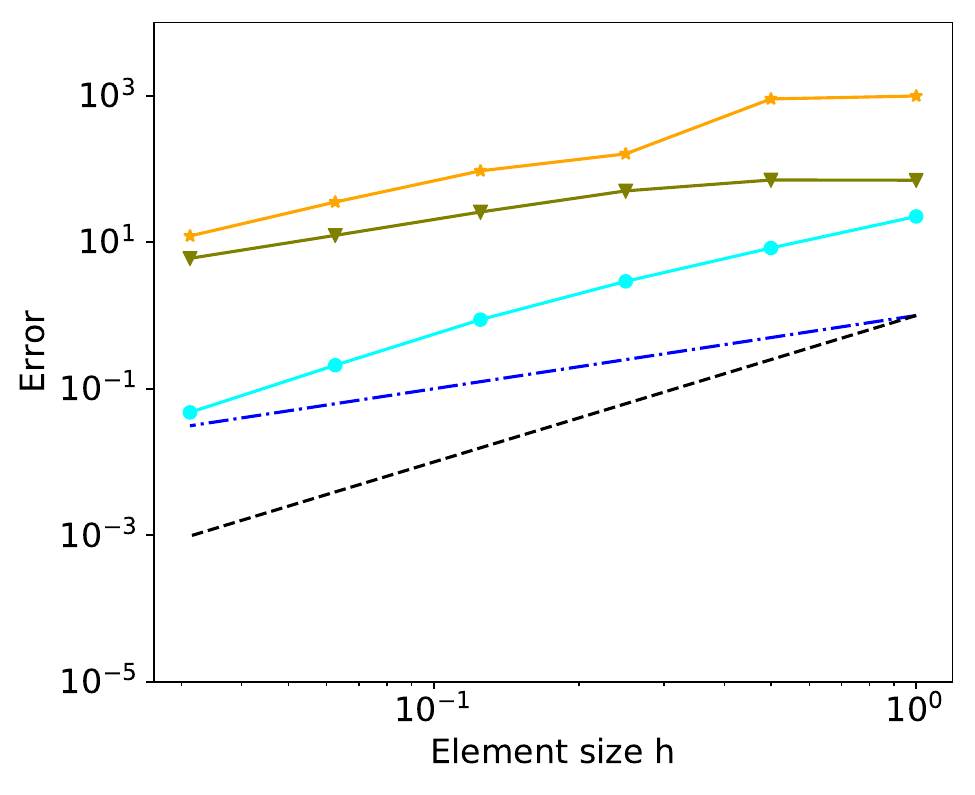}&
	\includegraphics[width=4.5cm,height=4.0cm]{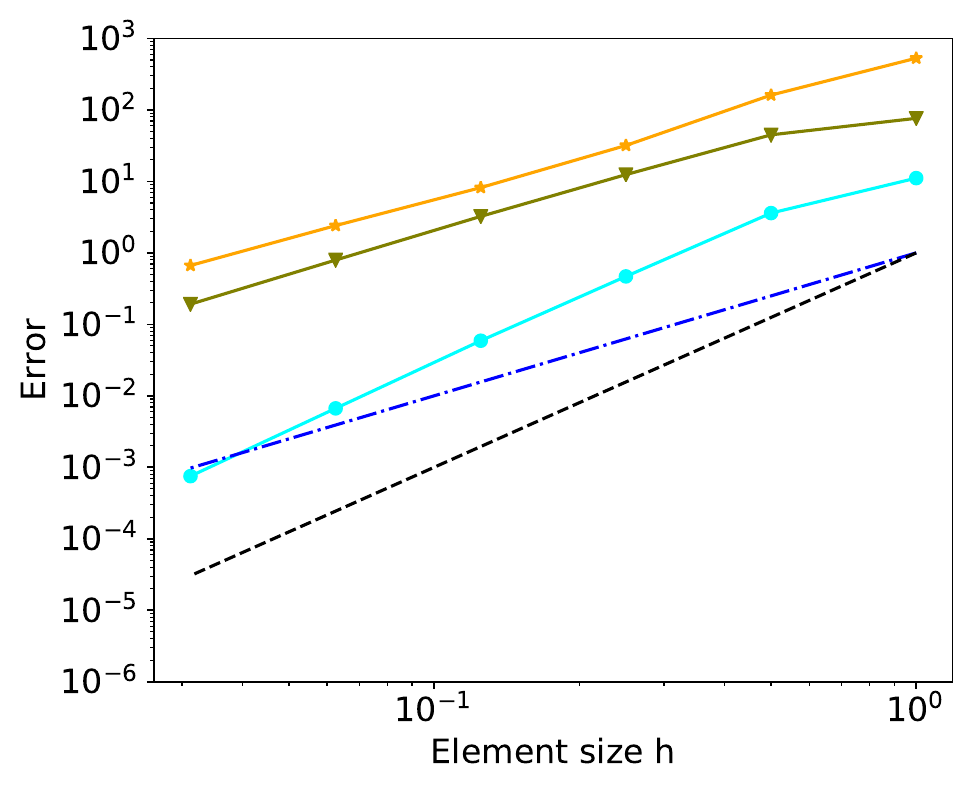}&
	\includegraphics[width=4.5cm,height=4.0cm]{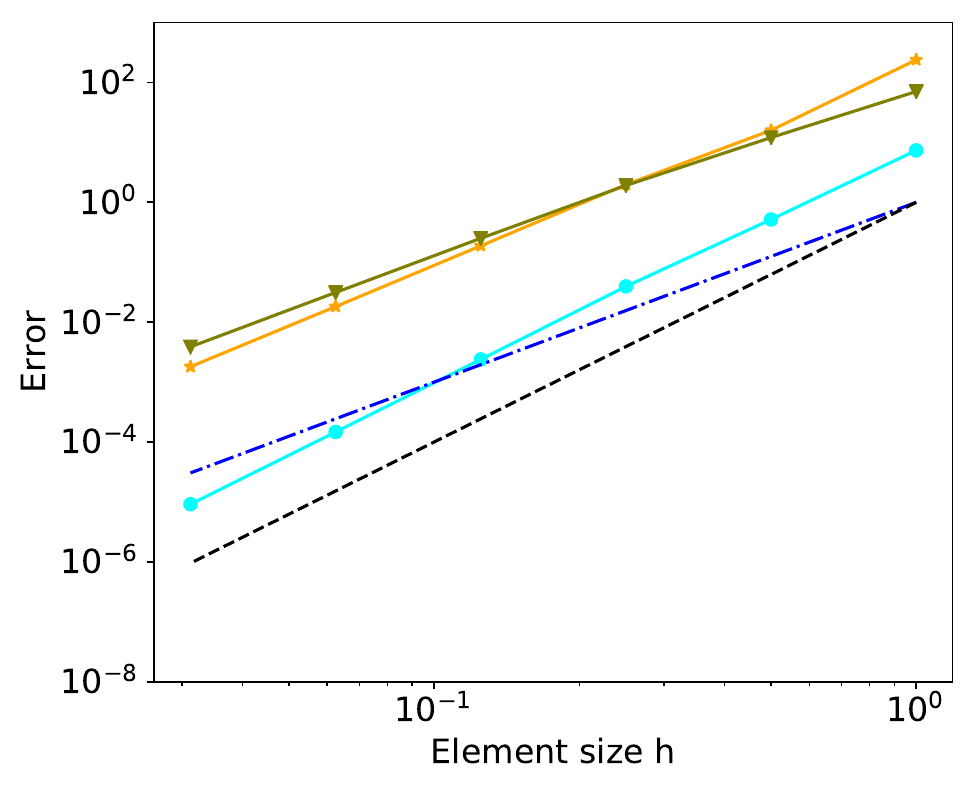}\\
	\includegraphics[width=4.5cm,height=4.0cm]{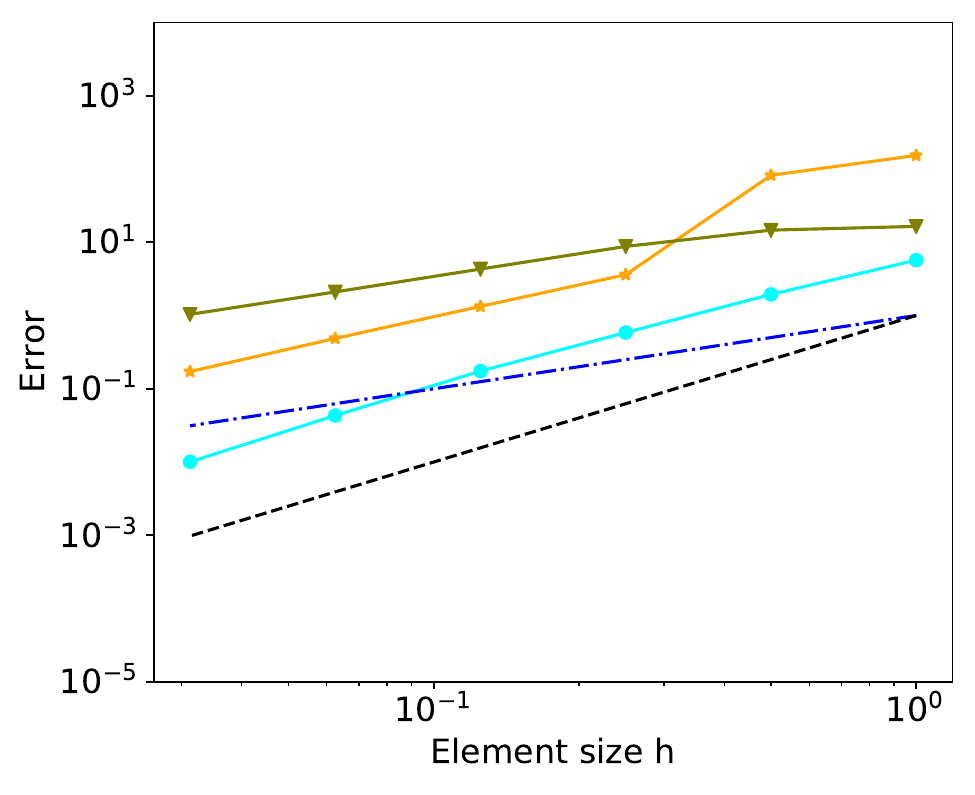}&
	\includegraphics[width=4.5cm,height=4.0cm]{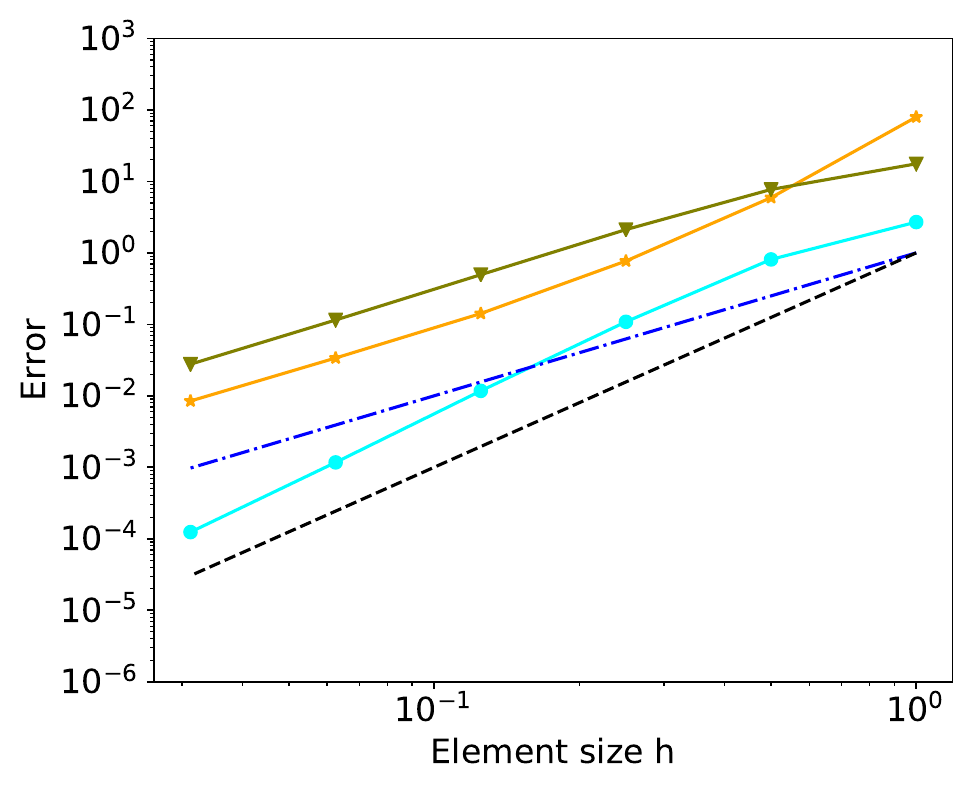}&
	\includegraphics[width=4.5cm,height=4.0cm]{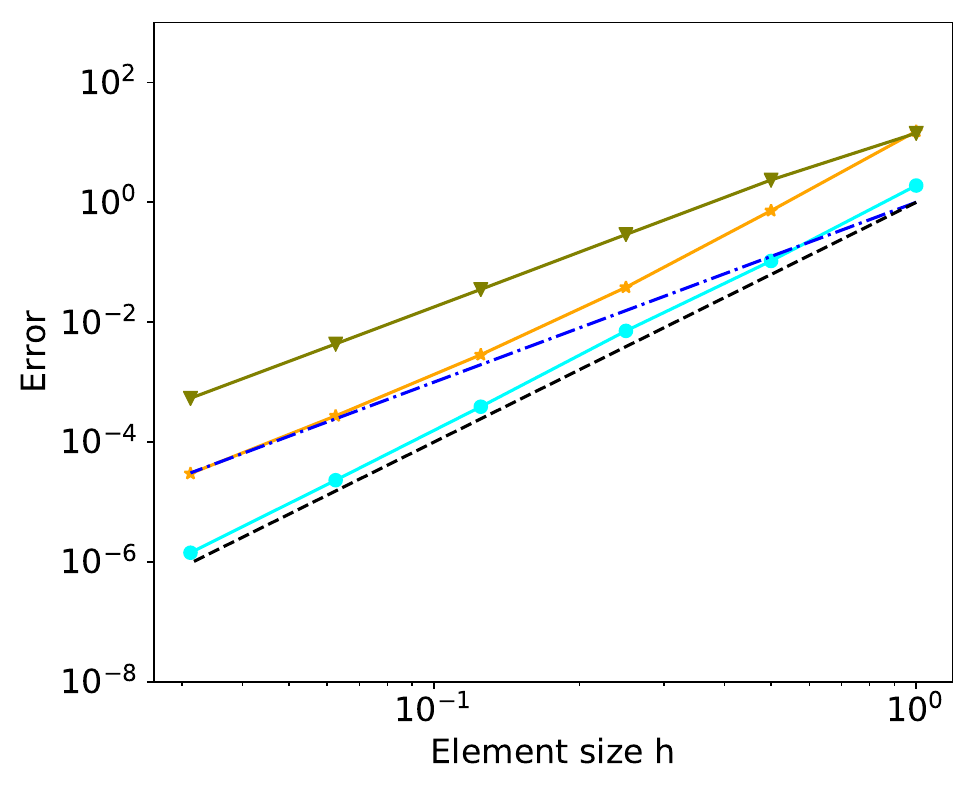}\\
	\includegraphics[width=4.5cm,height=4.0cm]{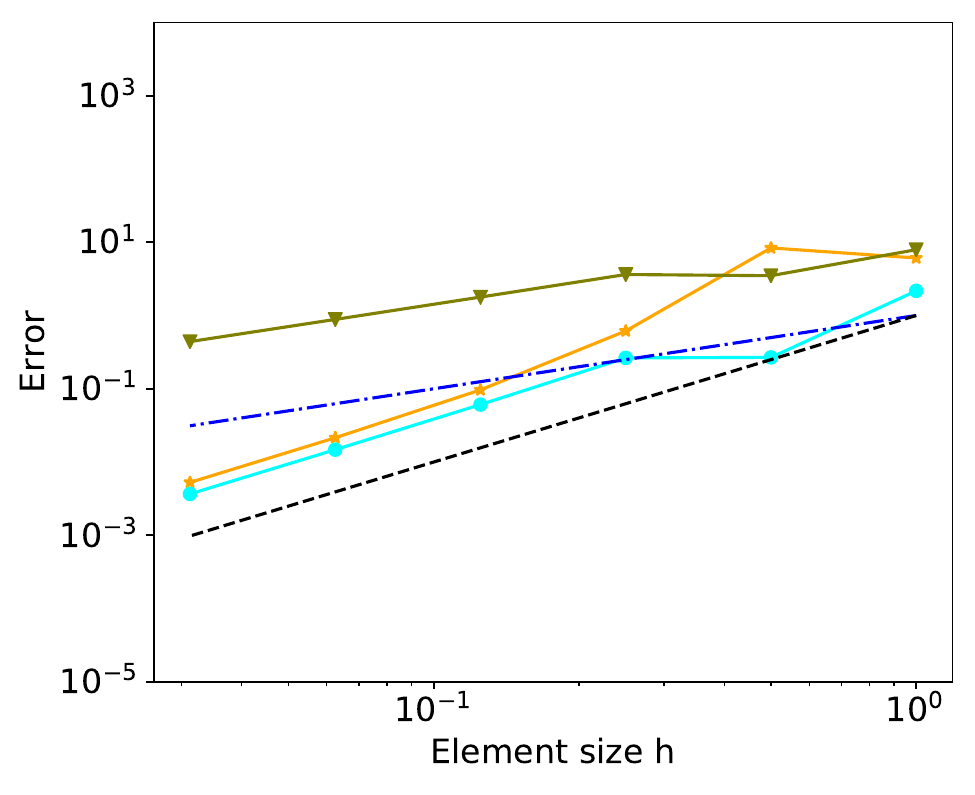}&
	\includegraphics[width=4.5cm,height=4.0cm]{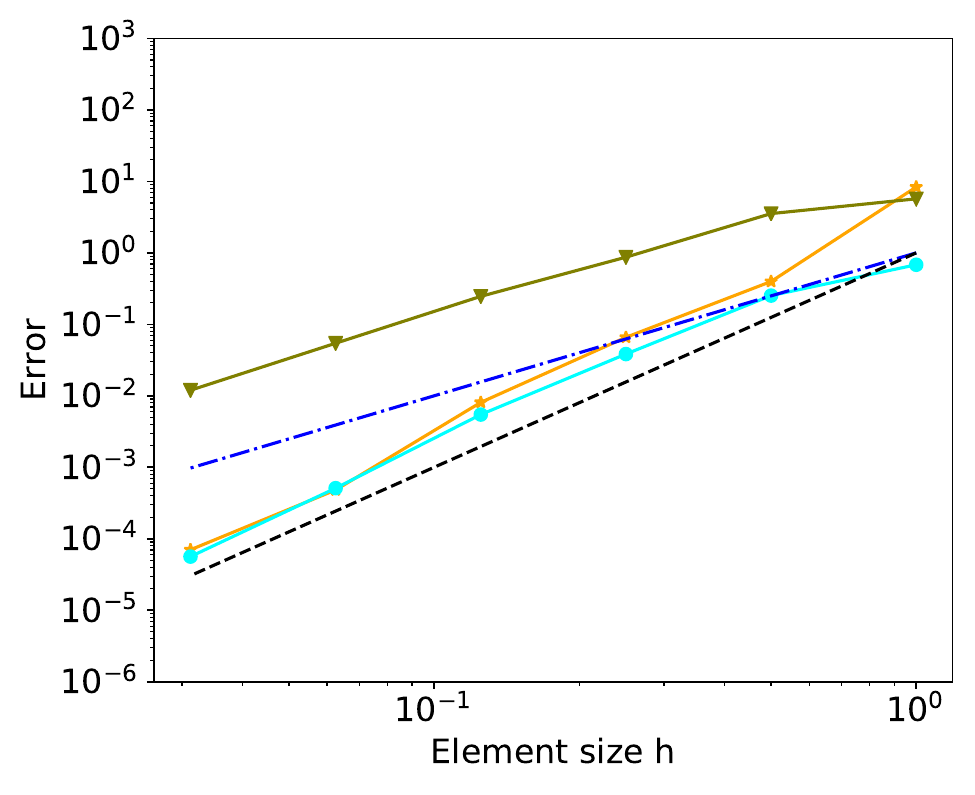}&
	\includegraphics[width=4.5cm,height=4.0cm]{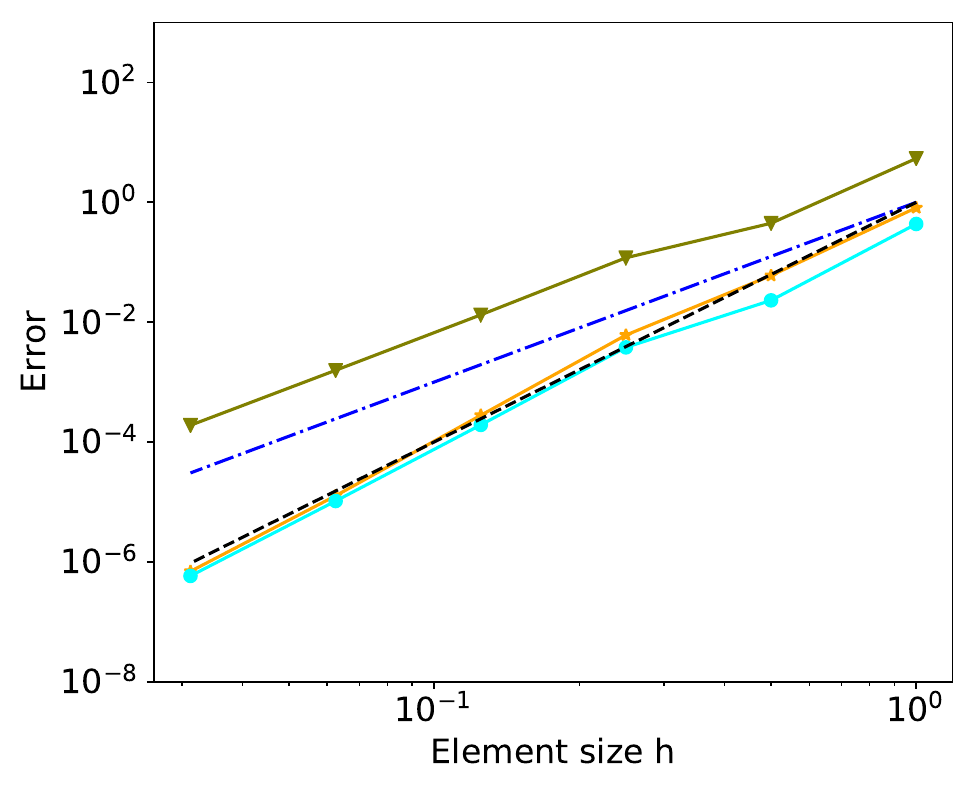}\\
	\includegraphics[width=4.5cm,height=4.0cm]{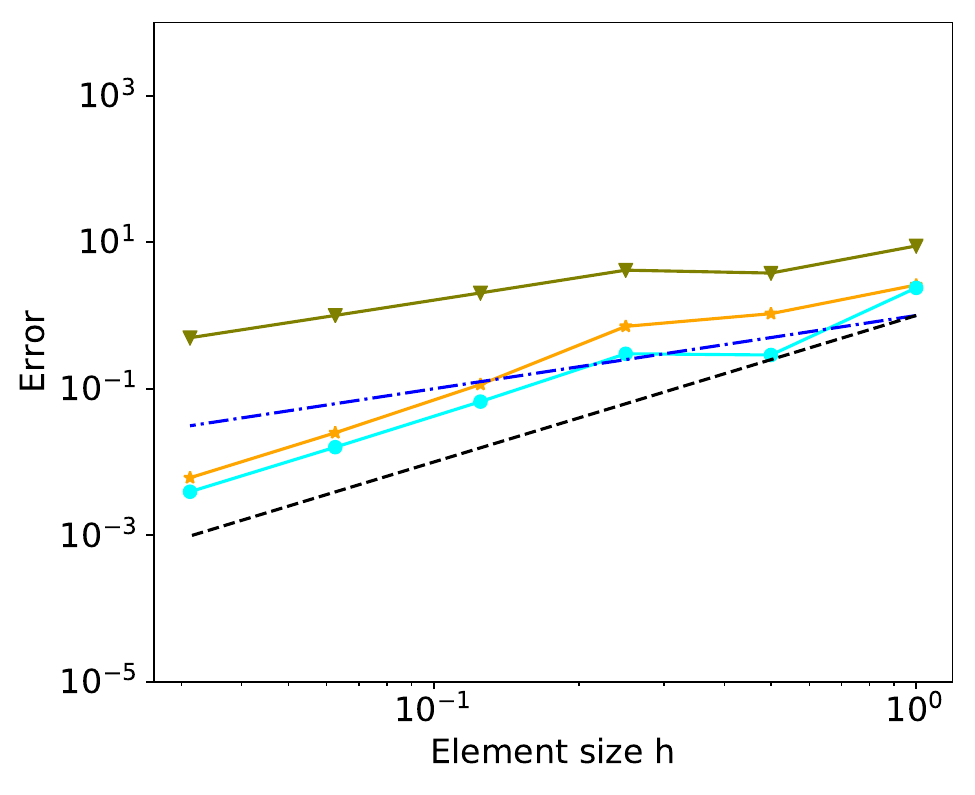}&
	\includegraphics[width=4.5cm,height=4.0cm]{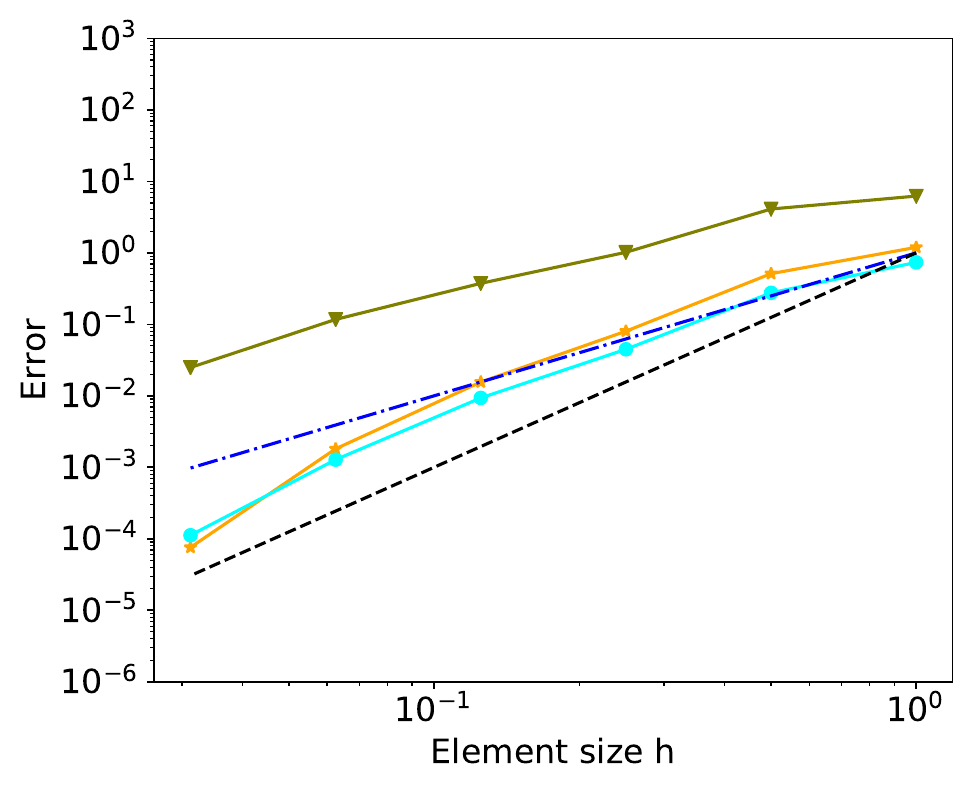}&
	\includegraphics[width=4.5cm,height=4.0cm]{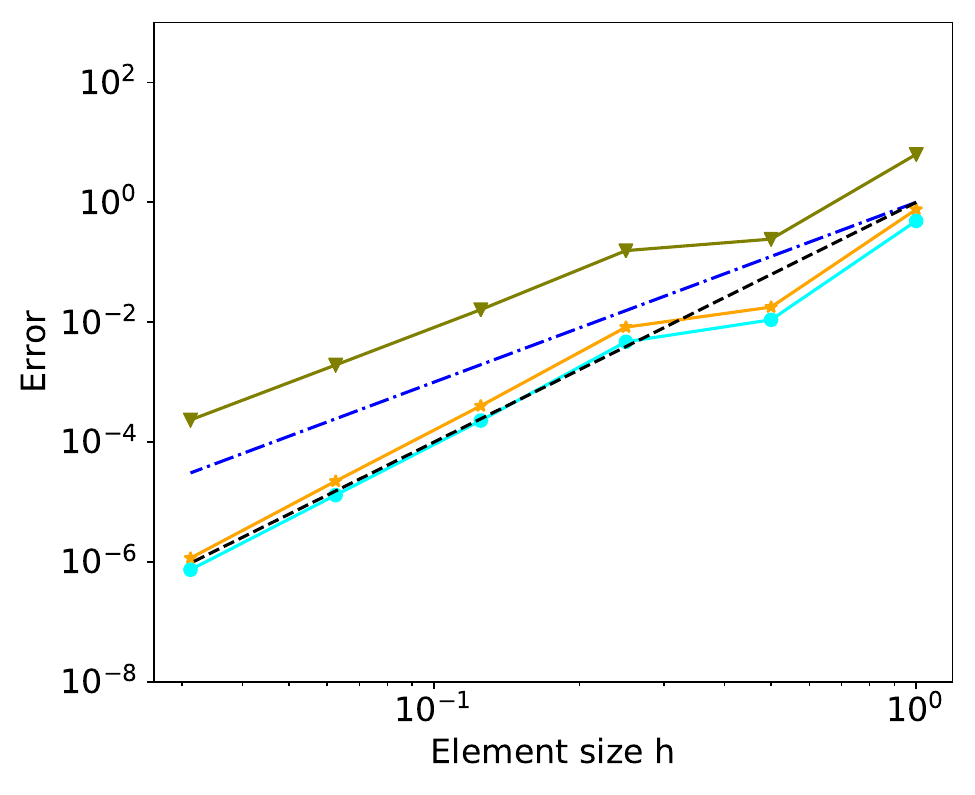}\\
	\multicolumn{3}{c}{\includegraphics[width=12.5cm,height=0.7cm]{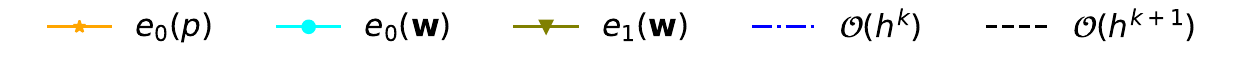}} 
\end{tabular}
\vspace{-0.3cm}
\caption{{\bf Kovasznay {flow for the model problem with} $\aaah = 0.9\II_h(\ww)$.} 
	{Convergence histories for linear elements ($k=1$) (left), quadratic elements ($k=2$) (center) and cubic elements (right).} 
	From top to bottom the values of $\mu$ are $1, 10^{-1}, 10^{-2},$ and $10^{-3}$.} 	
\label{curve:ex1-case1}
\end{figure}

\begin{figure}[H]
\centering
\begin{tabular}{ccc}
	$k=1$ & $k=2$ & $k=3$\\
	\includegraphics[width=4.5cm,height=4.0cm]{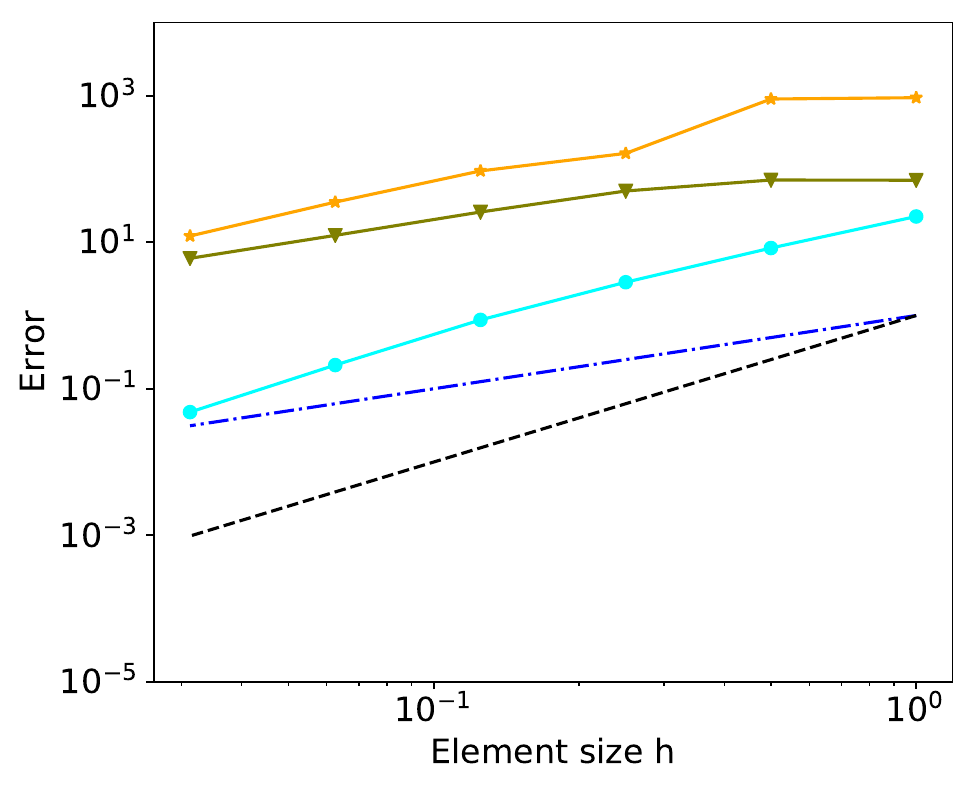}&
	\includegraphics[width=4.5cm,height=4.0cm]{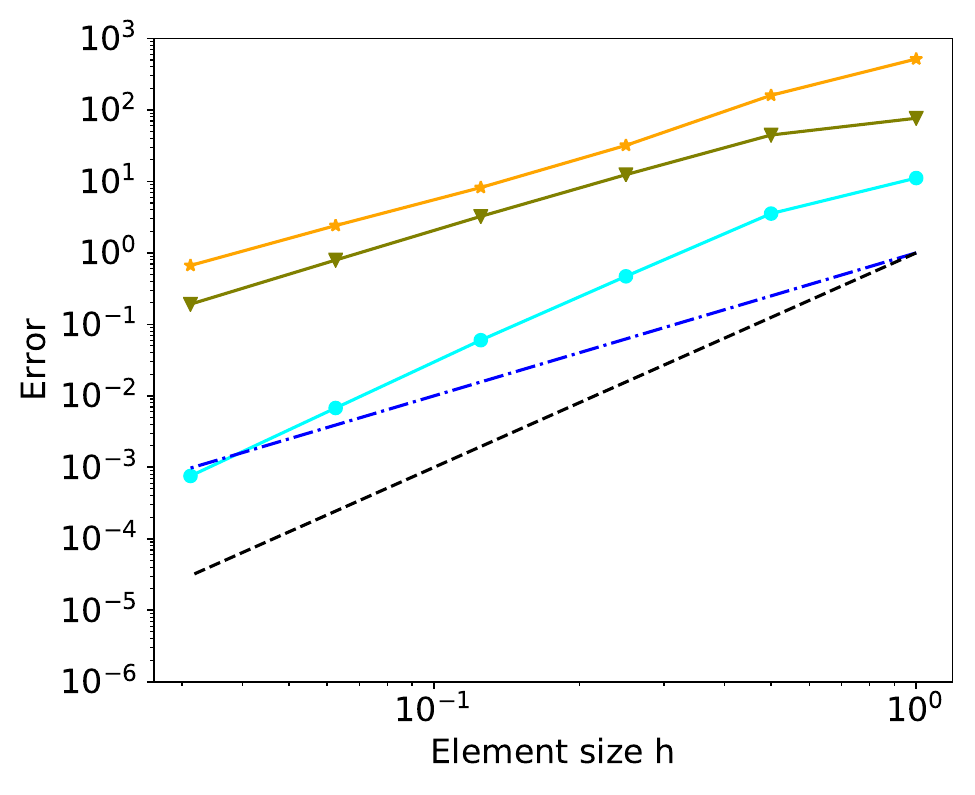}&
	\includegraphics[width=4.5cm,height=4.0cm]{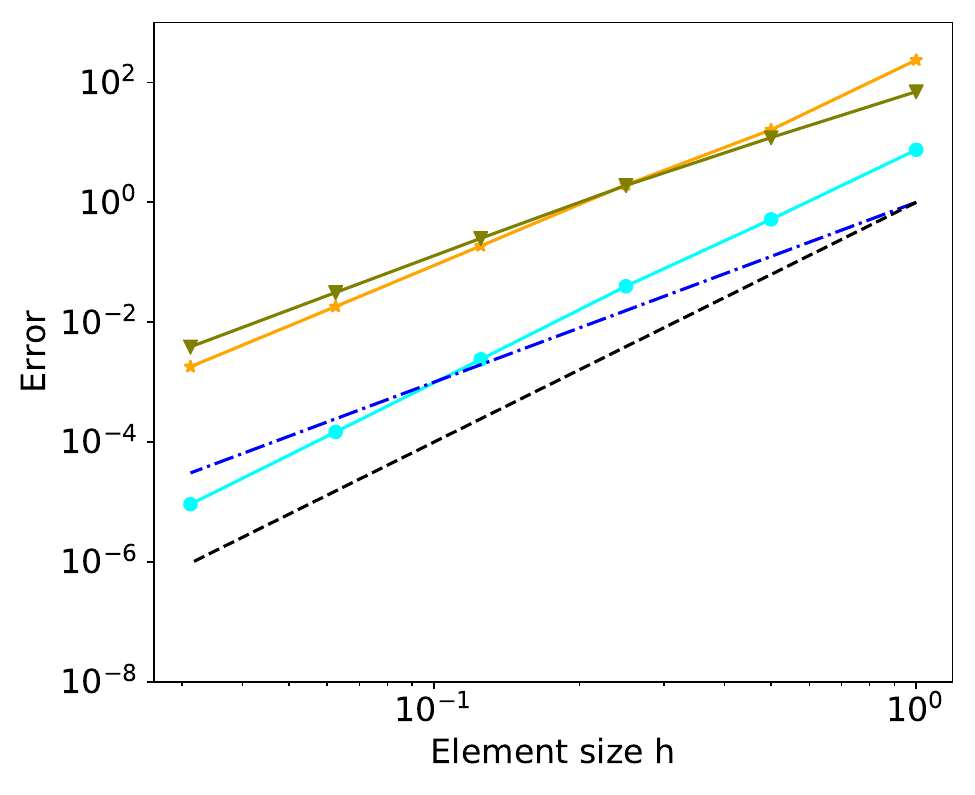}\\
	\includegraphics[width=4.5cm,height=4.0cm]{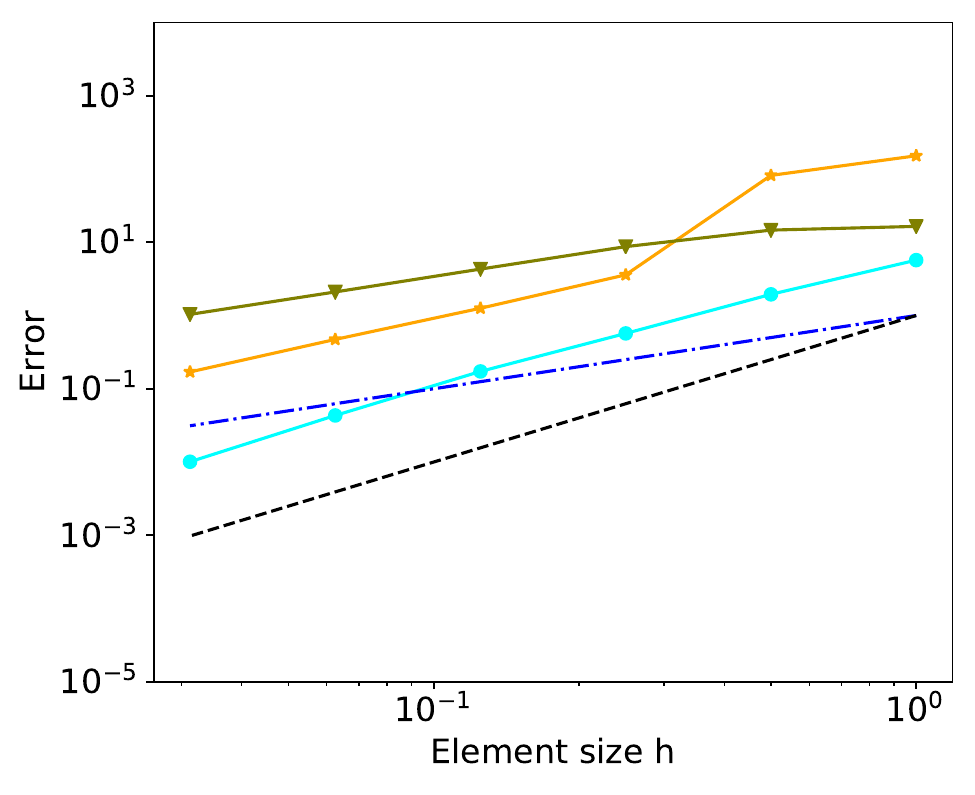}&
	\includegraphics[width=4.5cm,height=4.0cm]{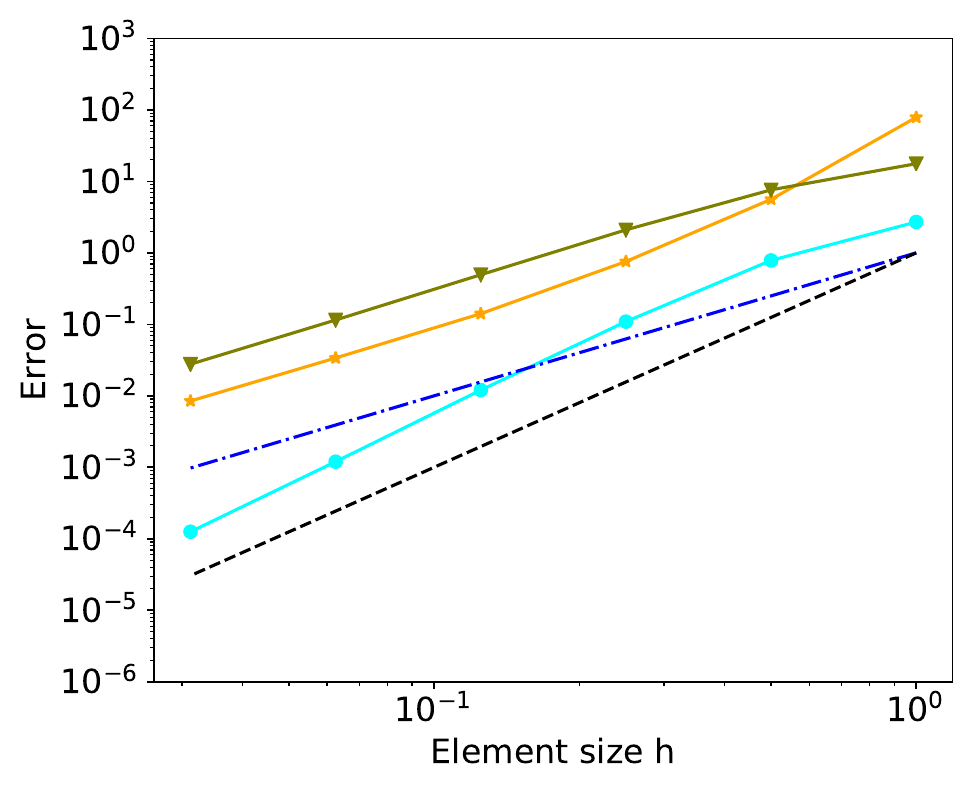}&
	\includegraphics[width=4.5cm,height=4.0cm]{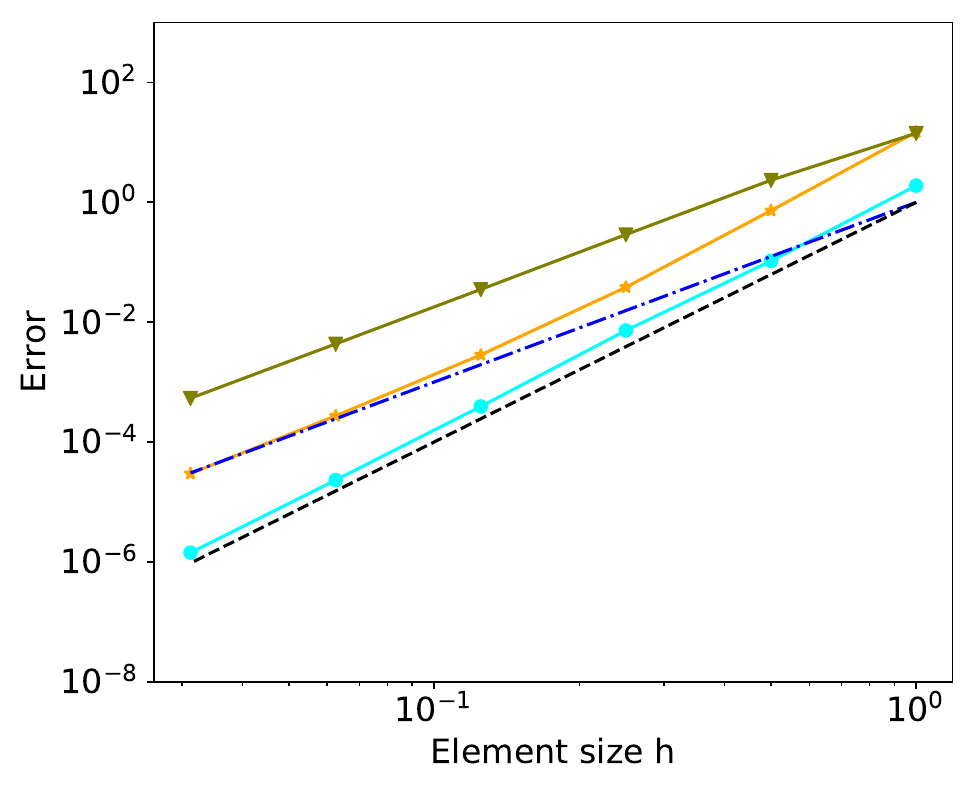}\\
	\includegraphics[width=4.5cm,height=4.0cm]{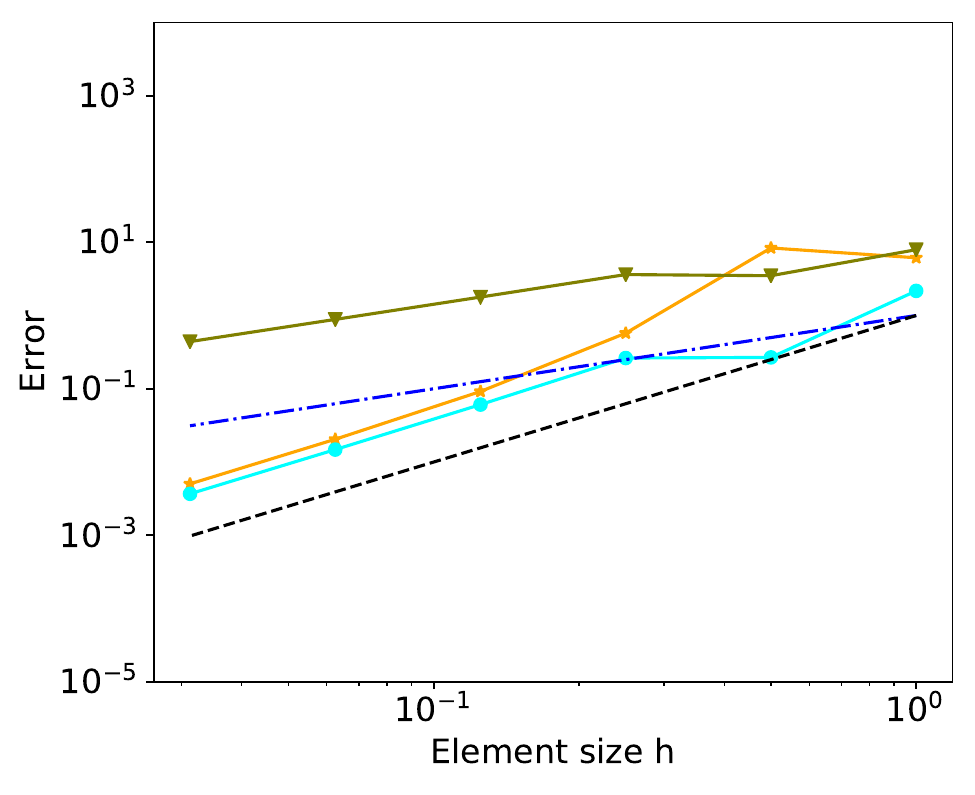}&
	\includegraphics[width=4.5cm,height=4.0cm]{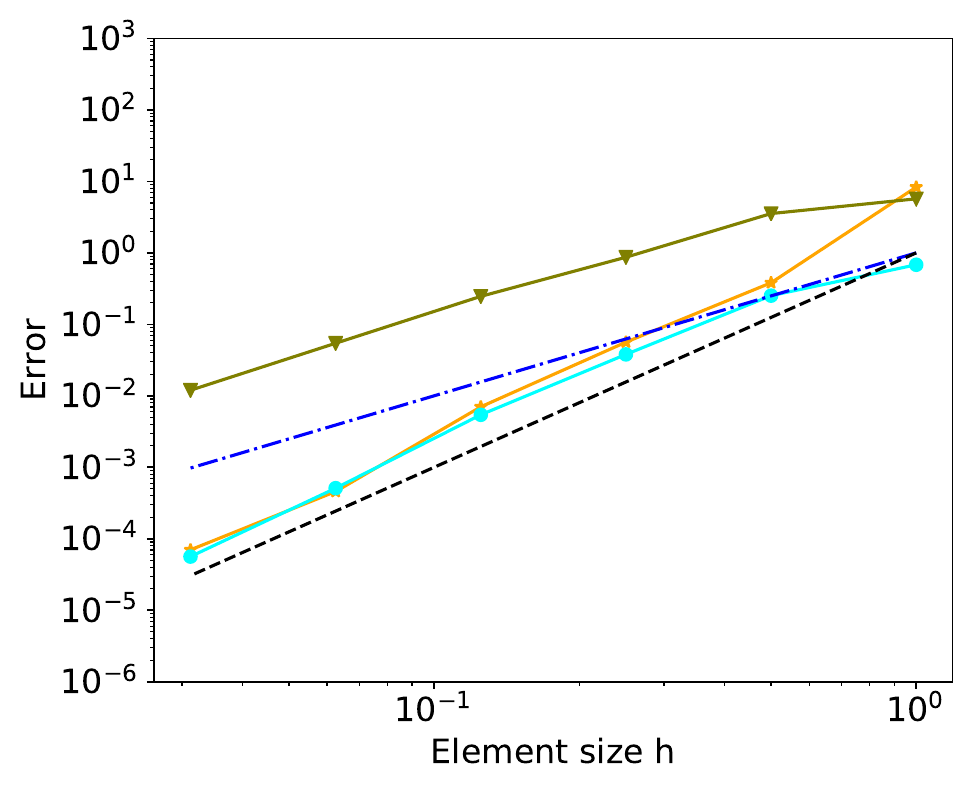}&
	\includegraphics[width=4.5cm,height=4.0cm]{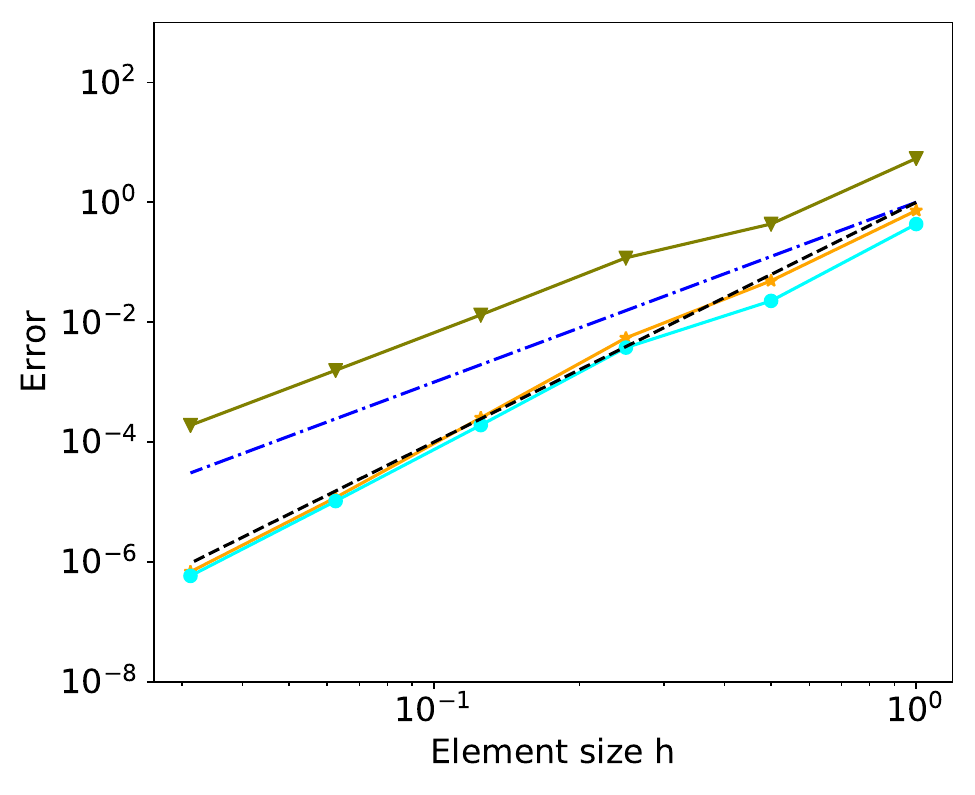}\\
	\includegraphics[width=4.5cm,height=4.0cm]{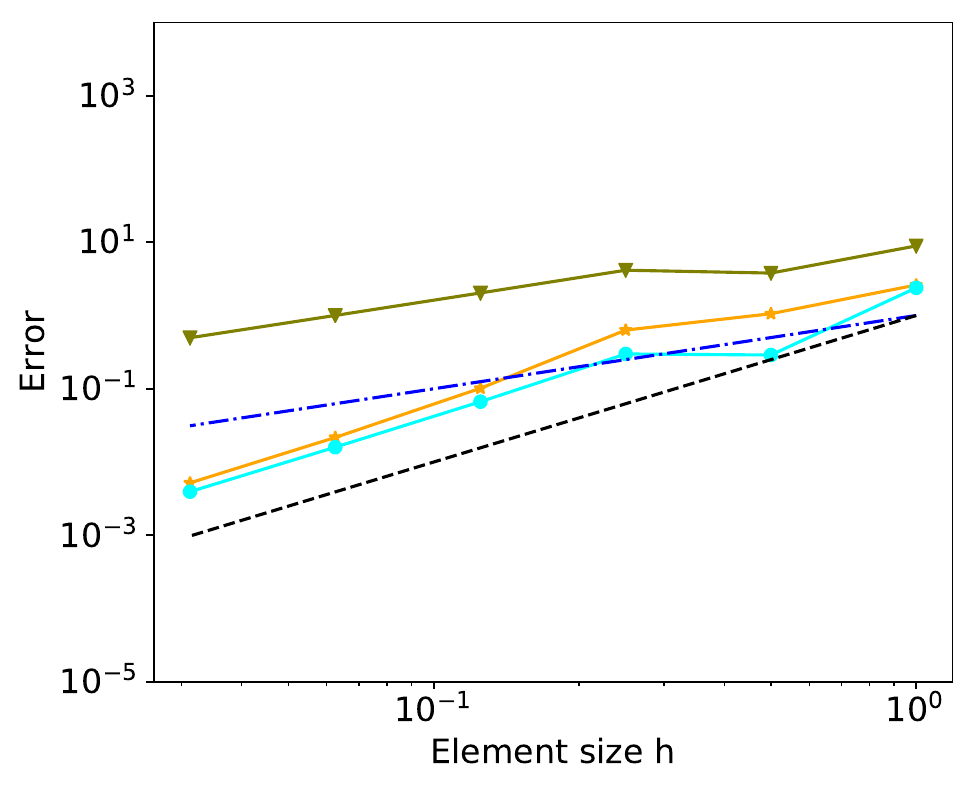}&
	\includegraphics[width=4.5cm,height=4.0cm]{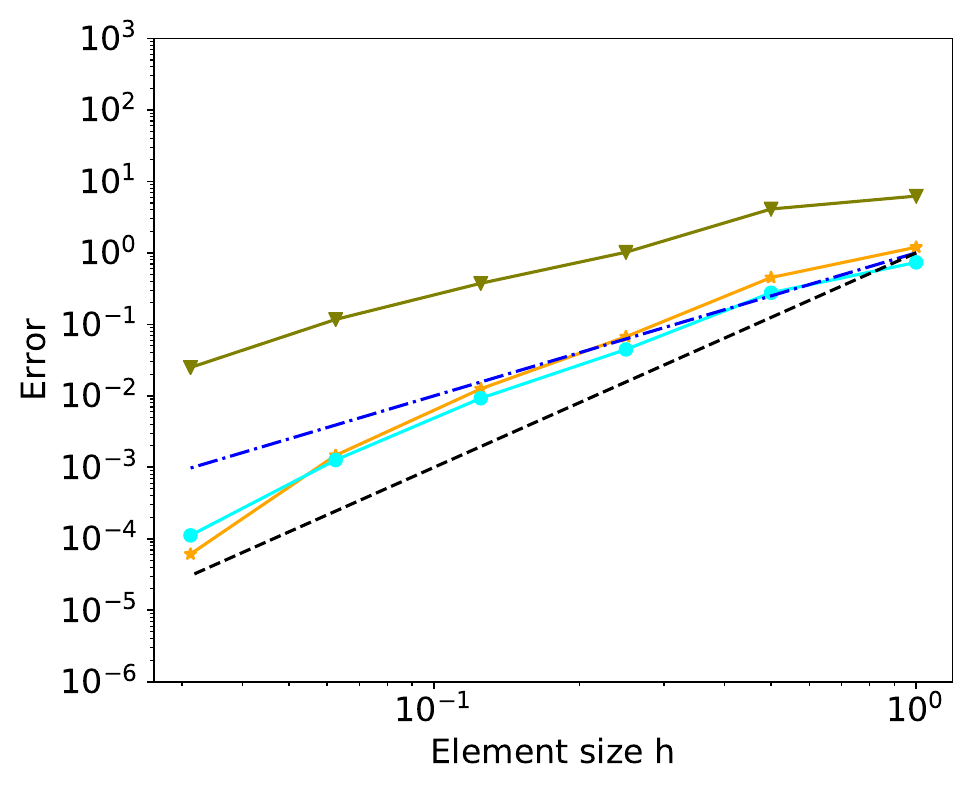}&
	\includegraphics[width=4.5cm,height=4.0cm]{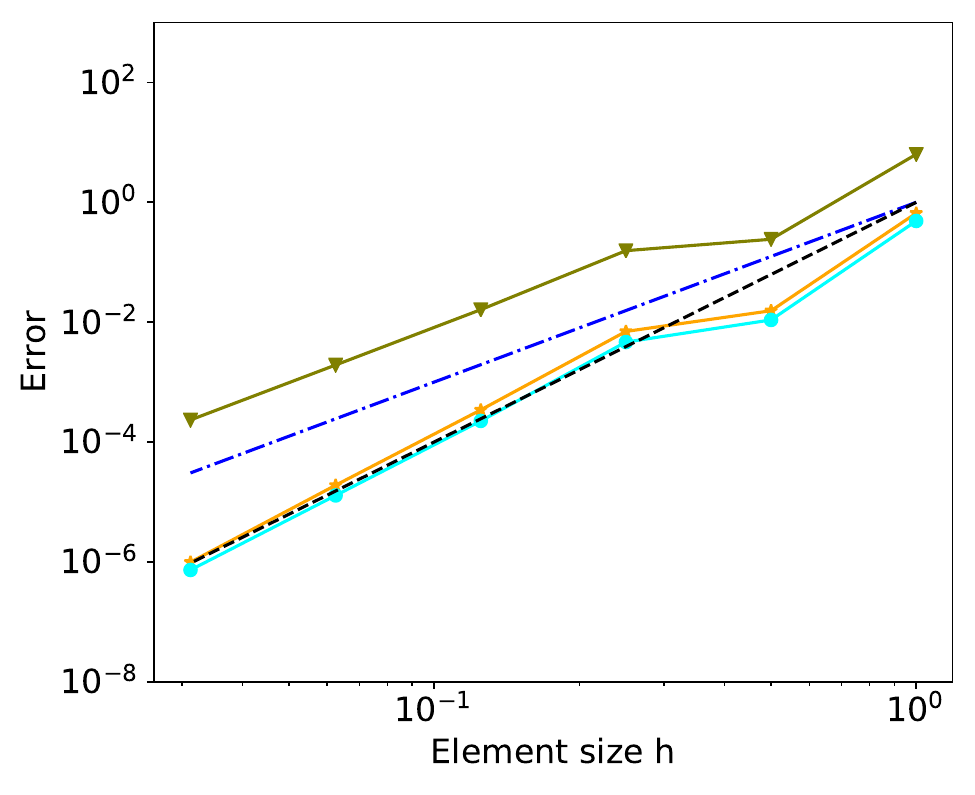}\\
	\multicolumn{3}{c}{\includegraphics[width=12.5cm,height=0.7cm]{Figures/Example1/legend.pdf}} 
\end{tabular}
\vspace{-0.3cm}
\caption{{\bf Kovasznay {flow for the perturbed Navier--Stokes problem with} $\aaah = \ww_h$.} 
	{Convergence histories for linear elements ($k=1$) (left), quadratic elements ($k=2$) (center) and cubic elements (right).} From top to bottom the values of $\mu$ are $1, 10^{-1}, 10^{-2},$ and $10^{-3}$.} 	\label{curve:ex1-case2}
	\end{figure}

	\subsection{\bf Experiment 2: Pressure reconstructed from random velocity measurements $\uu_m$ on a  bent computational domain.}
	
	In the next experiment, we reconstruct the pressure field from a random velocity field $\uu_{m}$ built on a bent computational domain. To mimic the 4D flow MRI setup, where the velocity is piecewise constant in a pixel-based grid, we consider a quadrilateral mesh formed by 693 elements. We assign a random field $\uu$ on each element, with magnitude values between $0$ and $120 cm/s$. 
	Next, we construct a criss-cross mesh of 2.772 triangles on which we interpolate the field $\uu$ in order to get the velocity field $\uu_m\in \HH_h$ 
with $k=1$ {(see Figure~\ref{bend1} for an illustration)}. Once the velocity field $\uu_{m}$ is known, we use it to solve the problem  \eqref{stab}, where 
we select the data $\ff$ and $g$, such that the unknowns are defined as follows:
\begin{align*}
\ww = \begin{pmatrix}
\pi\cos(\pi y)\sin(\pi x) \\ -\pi\cos(\pi x)\sin(\pi y)
\end{pmatrix},
\quad  p = \cos(\pi x)\cos(\pi y), \quad \text{and}\quad \aaa = 0.9\,\ww.
\end{align*}
The physical parameters used to compute this solution are $\rho =1.119 gr/cm^3$ and $\mu =0.0483 P$.  Regarding $\sigma $, given that {$\Vert  \nabla \uu_m\Vert_{\infty,\OO} \thickapprox 1.2 $}, we have taken $\sigma = 5.37$ in order to satisfy the condition \eqref{assum1d}.  
The stabilization parameters are $\lambda =0.5$ and $\delta = 0.5$.
\begin{figure}[H]
\begin{center}
\includegraphics[width=6cm,height=4cm]{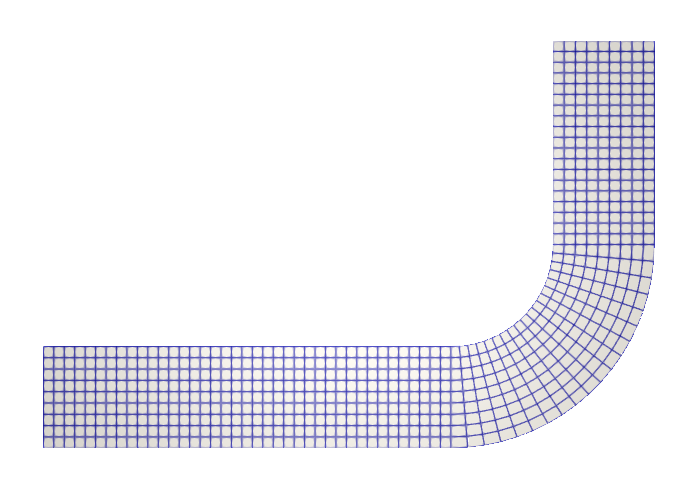}
\includegraphics[width=6cm,height=4cm]{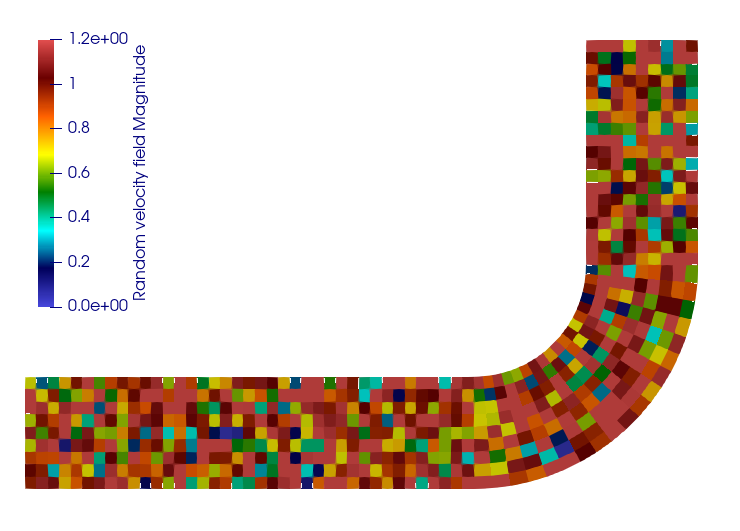}
\includegraphics[width=6cm,height=4cm]{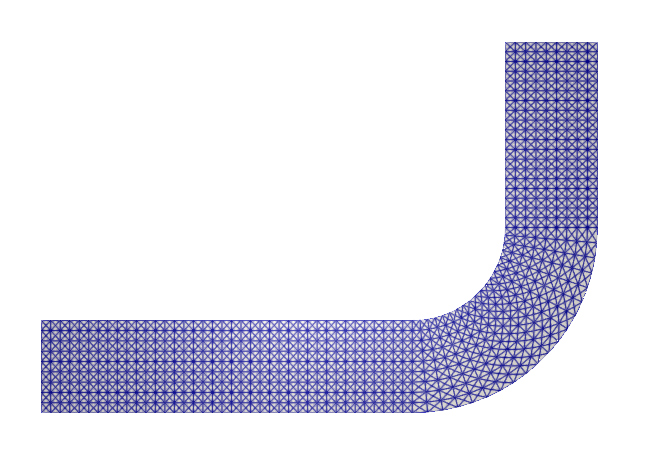}
\includegraphics[width=6cm,height=4cm]{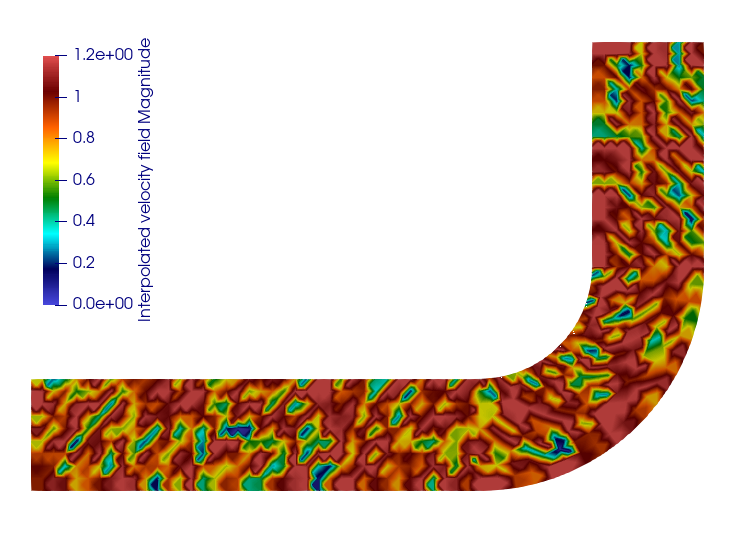}
\caption{{\bf Pressure reconstructed on a bent computational domain.}		
	Top: Quadrilateral mesh {(left)} and map of the distribution of the random velocity field $\uu$ {(right)}. Bottom: Triangular  criss-cross mesh
	{(left)}   and the interpolated velocity {$\uu_m\in \HH_h$}, with $k=1$ {(right)}. } 	\label{bend1}
	\end{center}
\end{figure}
\begin{figure}[H]
\begin{center}
\includegraphics[width=6cm,height=4cm]{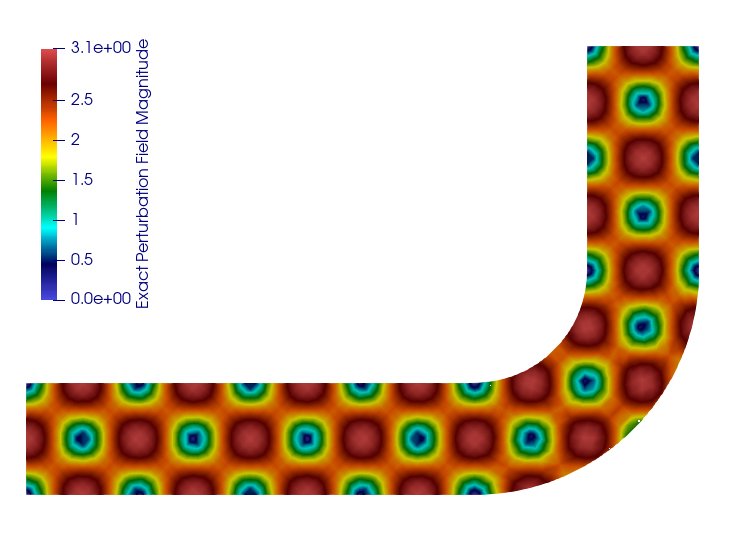}
\includegraphics[width=6cm,height=4cm]{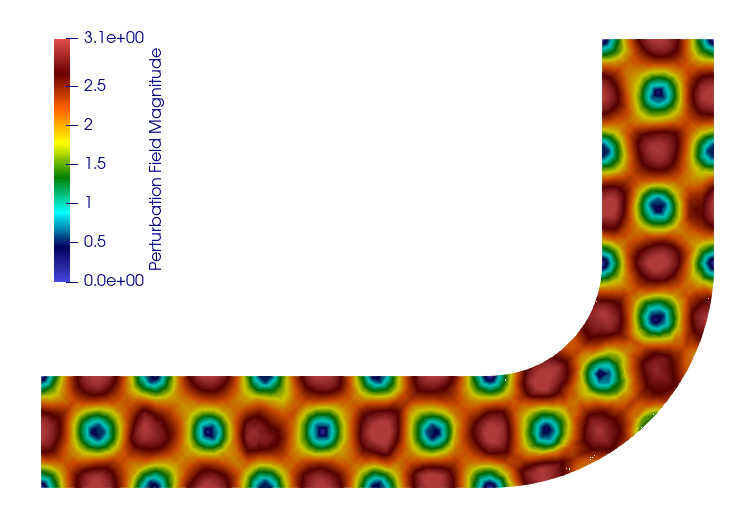}
\includegraphics[width=6cm,height=4cm]{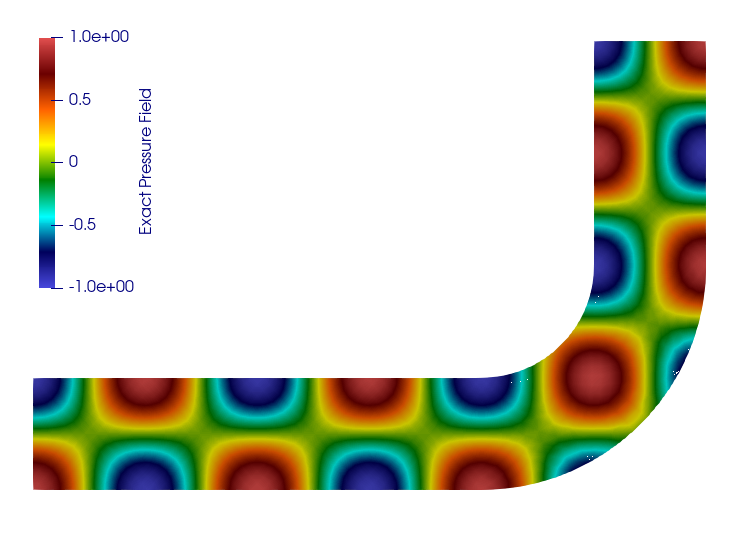}
\includegraphics[width=6cm,height=4cm]{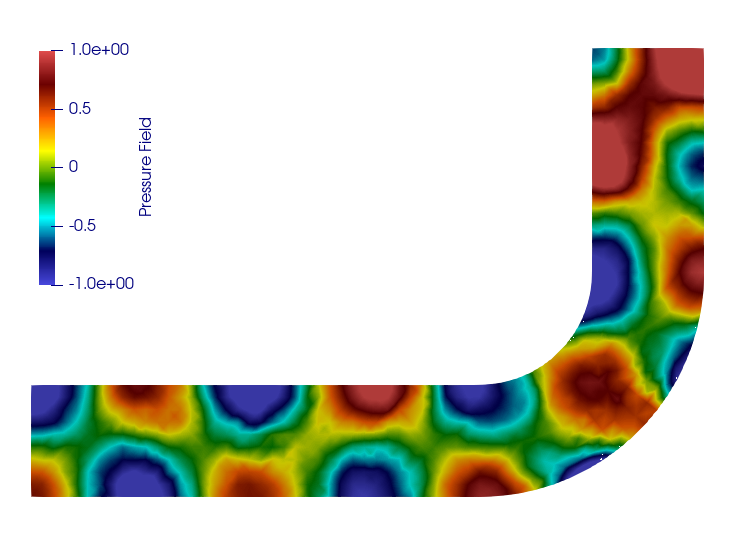}
\caption{{\bf Pressure reconstructed on a bent computational domain.}	
	Top: Exact solution $\ww$ (left) and computed solution $\ww_h$ (right). Bottom:  Exact solution $p$ (left) and computed solution $p_h$ (right). Here, we use 
	{linear elements ($k=1$) on a mesh formed by 2.772 triangles}. } 	\label{bend2}
	\end{center}
\end{figure}

The results depicted in Figure~\ref{bend2} indicate that the perturbation field $\ww_h$ is well captured, while the reconstructed pressure {$p_h$} is  close to the analytical solution,  although there is some blurring that appears, mostly due to the interpolation error between a piecewise constant function in a quadrilateral mesh ($\uu$) and its continuous interpolate $\uu_m$, which shows the sensitivity of the reconstructed pressure to perturbations on the data.

\subsection{\bf Experiment 3: Pressure reconstruction using an interpolated Navier–Stokes velocity field $\uu_m$.}

{In this final numerical experiment
we consider a test case motivated by the main idea of this paper, which is to reconstruct the pressure using a velocity measurement 
$\uu_m$  obtained from an interpolated solution of the incompressible Navier–Stokes equation:  
\begin{equation}
\label{NSG}
\sigma \,\uu - \mu \,\Delta \uu + (\nabla \uu)\uu + \nabla p = \zero, \qquad \nabla \cdot \uu =0.
\end{equation}
Here, the computational domain $\OO\defi(0,4)\times(0,1)$, and the different boundary conditions are shown in Figure~\ref{domain-example3}, with $u_p=  1-\left(\dfrac{y-0.5}{0.5}\right)^2.$
\begin{figure}[H]
\begin{center}
	\includegraphics[scale=2.8]{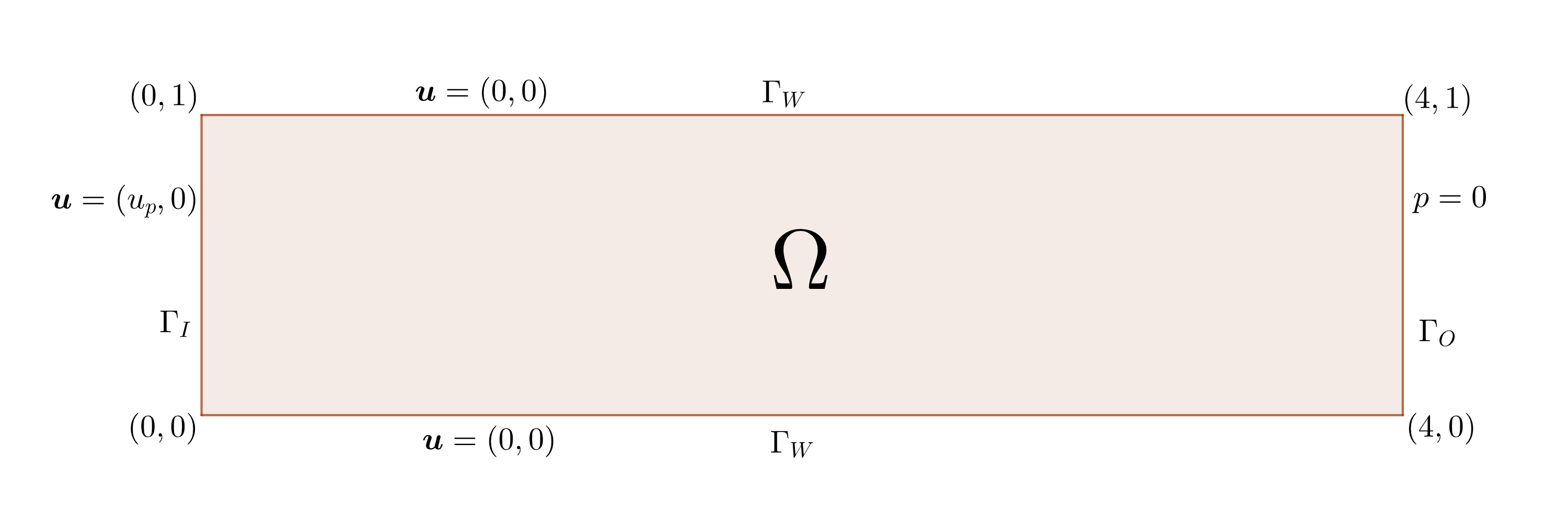}
	\caption{Domain computational and boundary conditions.} 	\label{domain-example3}
\end{center}
\end{figure}
In addition, the physical parameters considered as part of the problem are $\mu = 0.035$} and $\rho = 1$.   In order to reach our goal,  the main steps of the proposed approach are detailed below.

\begin{itemize}

\item[Step 1:]  {\textit{Solve the incompressible Navier-Stokes equations on a quadrilateral partition.}
We first solve the  Navier-Stokes equations \eqref{NSG}  in the domain $\OO$, using a coarse quadrilateral mesh $\mathcal{T}_{H,q}$ of $400$ cells with $\mathbb{Q}_1^2\times\mathbb{Q}_1$ elements by implementing the stabilized problem \cite{John2020}: Find $(\uu_{m,q},p_q)\in \mathbb{Q}_1^2\times\mathbb{Q}_1$ such that
\begin{align*}
	&\sigma (\uu_{m,q},\vv) +\mu( \nabla \uu_{m,q},\nabla \vv) + ((\nabla\uu_{m,q})\uu_{m,q},\vv) - (p_q,\nabla \cdot \vv) + \rho((\nabla\cdot \uu_{m,q}) ,r)  \\
	&\qquad +\sum_{K \in \T_{H,q}} \frac{H^2_K}{2\mu }(\sigma\uu_{m,q}+(\nabla \uu_{m,q})\uu_{m,q}+\nabla p_q-\Delta\uu_{m,q} , \nabla r)_K^{}=0 ,  
\end{align*}
where $H_K$ represents the diameter of each element $K\in \mathcal{T}_{H,q}$. 
The mesh and the computed velocity field $\uu_{m,q}$ are displayed in Figure \ref{coarse_results}. }

\item[Step 2:] \textit{Pressure reconstruction.} Now, we consider a triangulation $\trih$ composed by $1.600$ elements 
built by refining
$\T_{H,q}$, that is, each quadrilateral composing the mesh $\T_{H,q}$ is barycentrically refined into four triangles 
thus building the
resulting criss-cross mesh $\trih$. 
To reconstruct the pressure $p_h$ and compute the perturbation field $\ww_h$, we interpolate $\uu_{m,q}$ onto the subspace  $\boldsymbol{H}_h^{}$ with linear elements ($k=1$).
We denote this interpolated field by $\uu_m$ (see Figure~\ref{plots_interpoled} for an illustration).  We then solve \eqref{stab:NS} using stabilization parameters
$\lambda = 0.5$, $\displaystyle \delta = 0.001$, with $\aaa_h=\ww_h$, and  the datum $\ff$ given by
\begin{equation*}
\boldsymbol{f} \defi -\sigma \uu_{m}-\rho\, (\nabla  \uu_{m}) \uu_m.
\end{equation*}
The nonlinear system is solved using  a Picard iteration with tolerance $\mathrm{tol} = 10^{-10}$, and in order to satisfy the condition \eqref{assum1d}, we  consider $\sigma = 3.92$, because $\Vert \uu_m \Vert_{\infty,\OO} \approx  0.98$.  

\item[Step 3:] \textit{Validation of the computed approximation.} We compare the recovered pressure with a reference solution $p_{\mathrm{ref}}$ obtained
{by solving the Navier-Stokes equations \eqref{NSG} on a very fine mesh} $\trih^{\mathrm{ref}}$ using 80.000 triangular elements and  lowest-order Taylor-Hood element $\mathbb{P}_2^2\times\mathbb{P}_1$.   The  reference solution is displayed in Figure~\ref{fine_results}.
\end{itemize}

In Figure~\ref{plots_u} we  compare different profiles of
the datum {$|\uu_{m,q}|$} and its interpolate $|\uu_m|$,  showing that they are essentially identical (notice the scale on the utmost right-hand side plot).
Figure~\ref{results}  confirms
that {the magnitude of}  the perturbed velocity field $\ww_h$ 
{remains small in regions where the interpolated velocity $|\uu_m|$  closely matches $|\uu_{m,q}|$}, and that the reconstructed pressure $p_h$ is close to the reference one. {{These behaviors are also observed in the last two figures.} In Figure \ref{plots_p} } we depict three cross-sections of both the reference and reconstructed pressure  and we 
observe that they are indeed close. In fact, the largest error is at $x=0$, where it is of the order of  $1.6\%$.   Finally, the curves depicted in Figure \ref{plots_uref} confirm that the method 
reproduces $|\uu_{\mathrm{ref}}|$ very accurately.

\begin{figure}[H]
\begin{center}
\includegraphics[width=13.3cm,height=3.0cm]{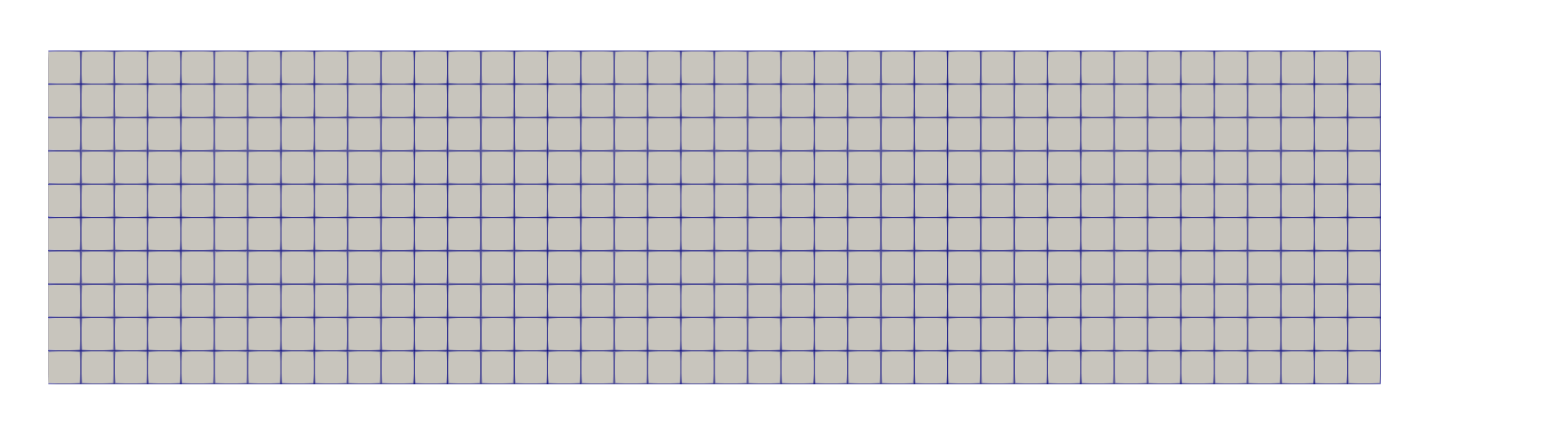}
\includegraphics[width=14cm,height=3.0cm]{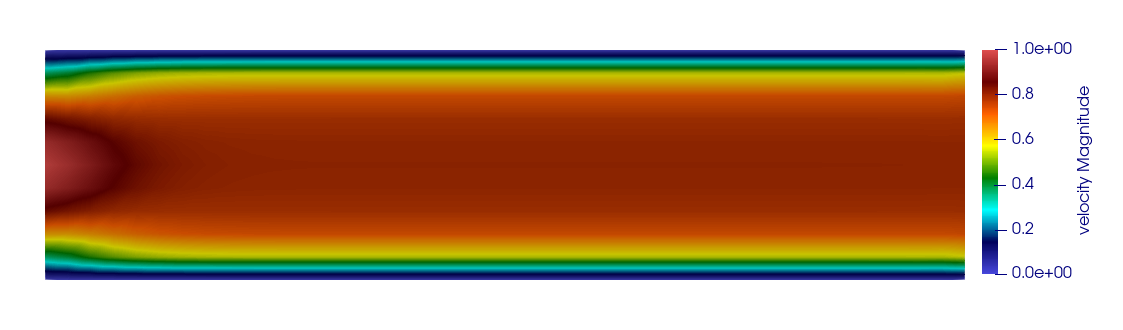}
\caption{{Top: mesh $\T_{H,q}$ composed of 400 elements. Bottom: isovalues of the velocity magnitude $|\uu_{m,q}|$.}} 	\label{coarse_results}
\end{center}
\end{figure}

\begin{figure}[H]
\begin{center}
\includegraphics[width=14cm,height=3.0cm]{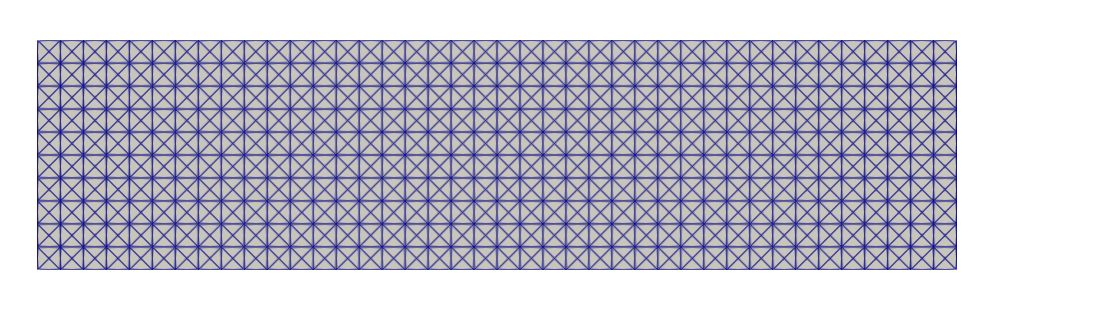}
\includegraphics[width=14cm,height=3.0cm]{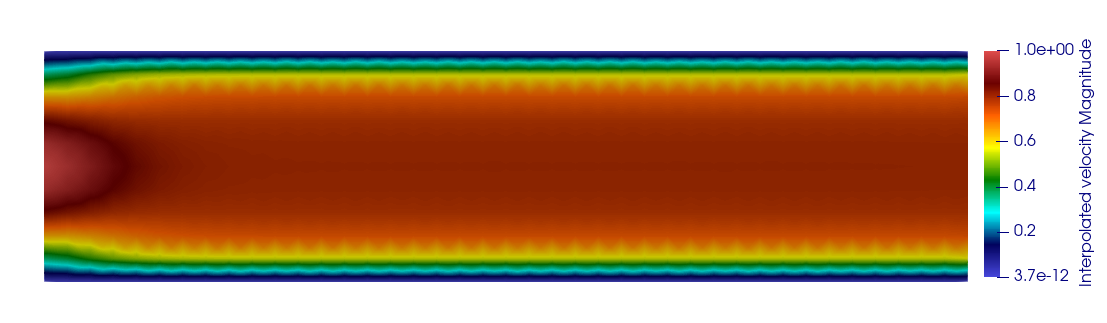}
\caption{ From top to bottom: Mesh $\trih$ composed by 1.600 elements, {magnitude of the interpolated velocity field $|\uu_{m}|$}.}
\label{plots_interpoled}
\end{center}
\end{figure}

\begin{figure}[H]
\begin{center}
\includegraphics[width=13.75cm,height=3.0cm]{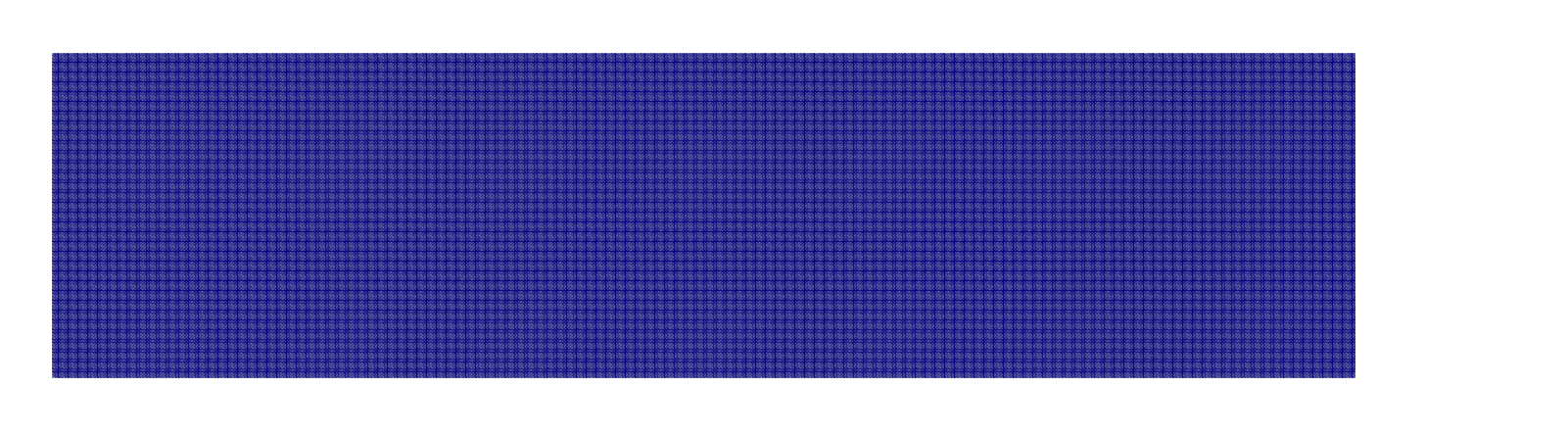}
\includegraphics[width=14cm,height=3.0cm]{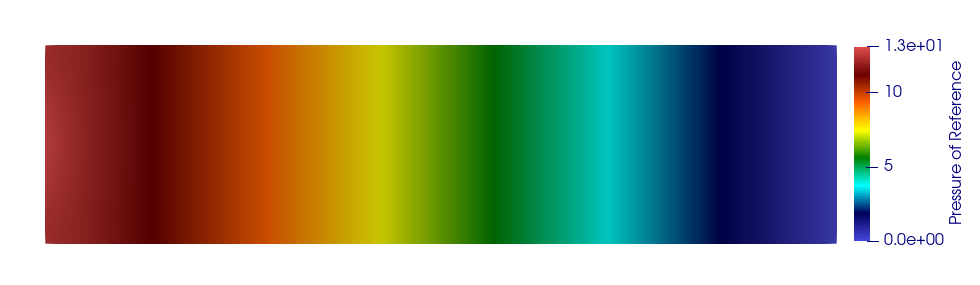}
\caption{{Triangulation  $\trih^{\mathrm{ref}}$ with 80.000 elements (top) and reference pressure field $p_{\mathrm{ref}}$ (bottom).}} 	\label{fine_results}
\end{center}
\end{figure}

\begin{figure}[!ht]
\centering
\begin{tabular}{ccc}
$x=0$ & $y = 0.5$ & $y=1$\\
\includegraphics[width=5cm,height=3.0cm]{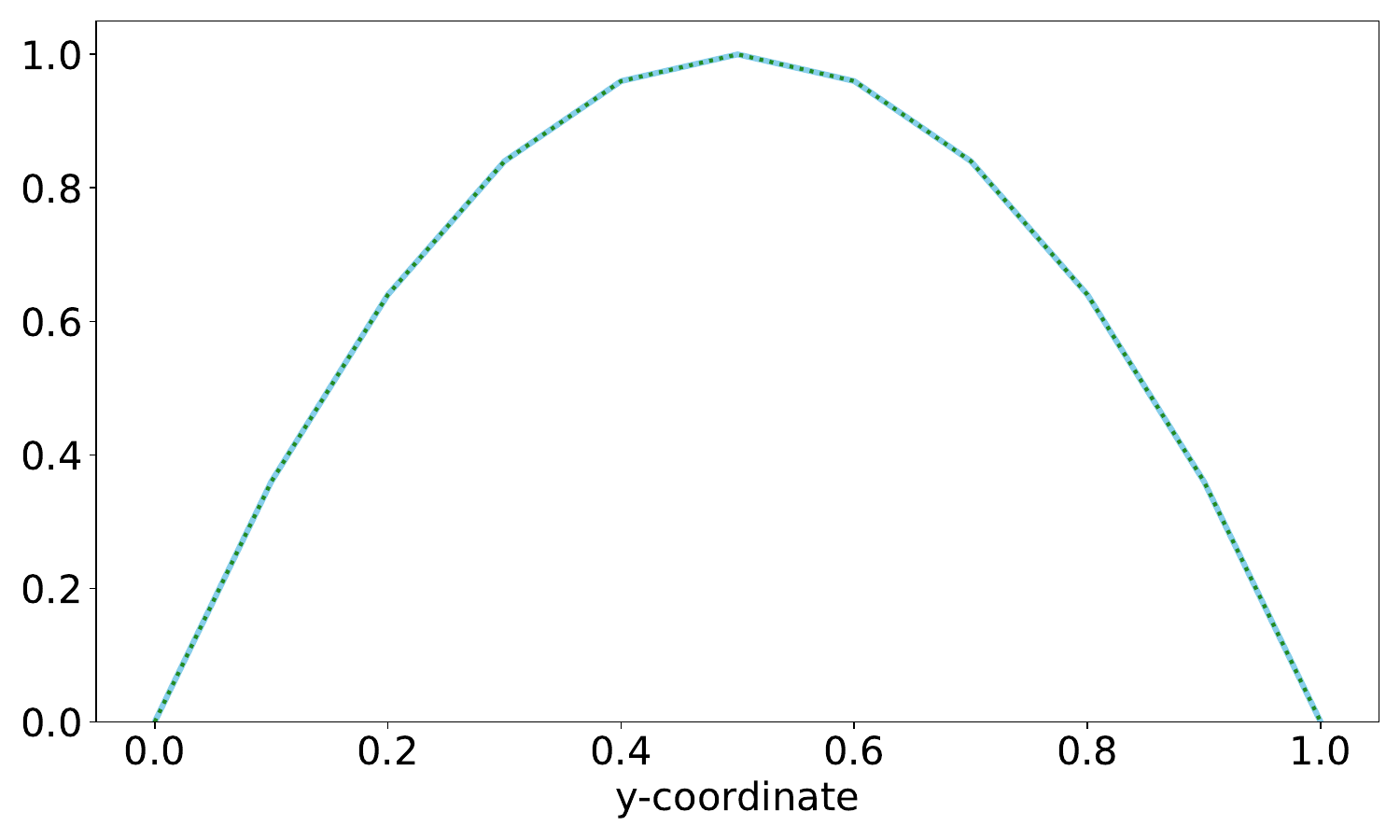}&
\includegraphics[width=5cm,height=3.0cm]{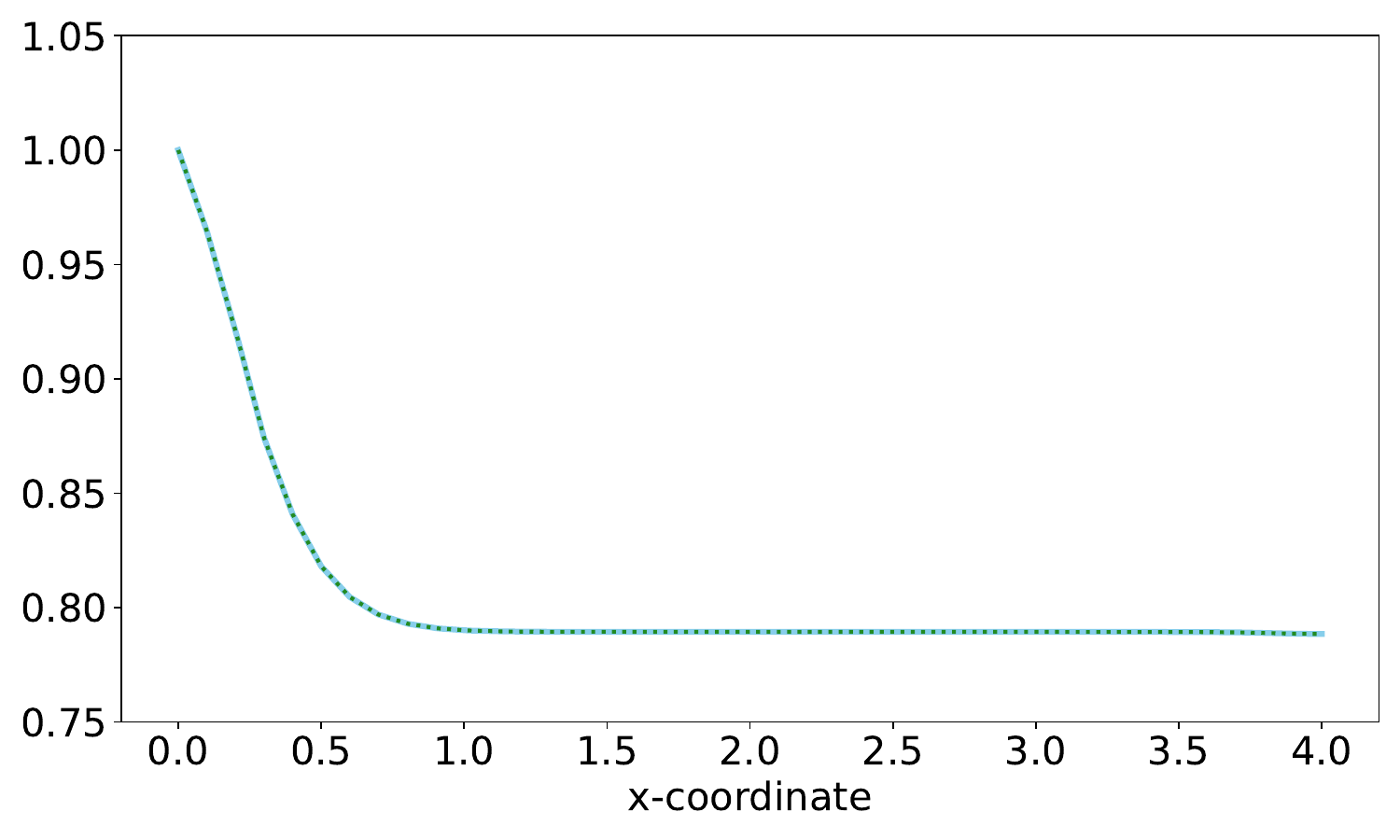}&
\includegraphics[width=5cm,height=3.0cm]{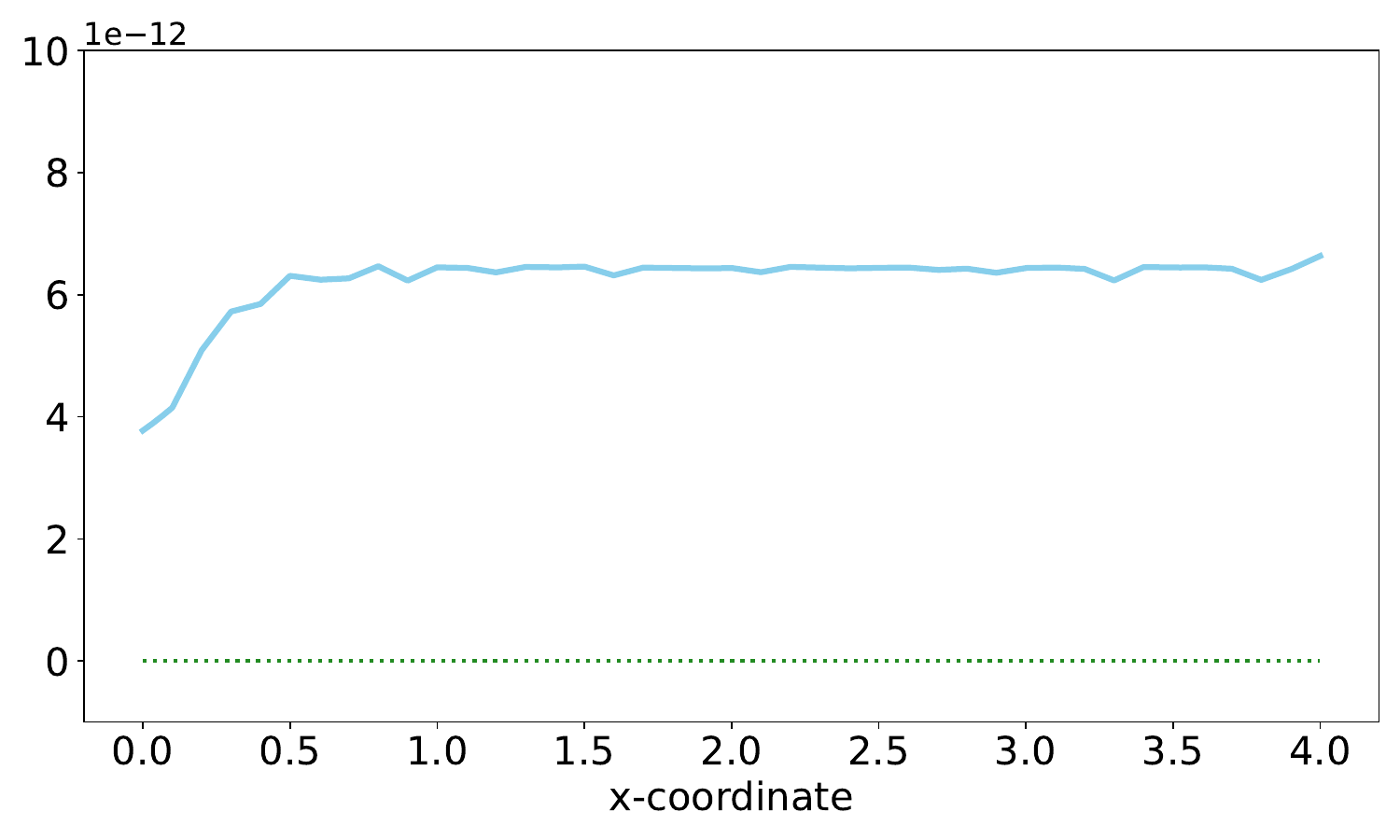}\\
\multicolumn{3}{c}{\includegraphics[width=5.5cm,height=0.5cm]{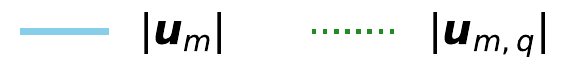}} 
\end{tabular}
\caption{Computed velocity field {$|\uu_{m,q}|$} versus interpolated velocity field {$|\uu_{m}|$} on $x=0$ (left), $y=0.5$ (middle), and $y=1$ (right).}
\label{plots_u}
\end{figure}

\begin{figure}[H]
\begin{center}
\includegraphics[width=14cm,height=3.0cm]{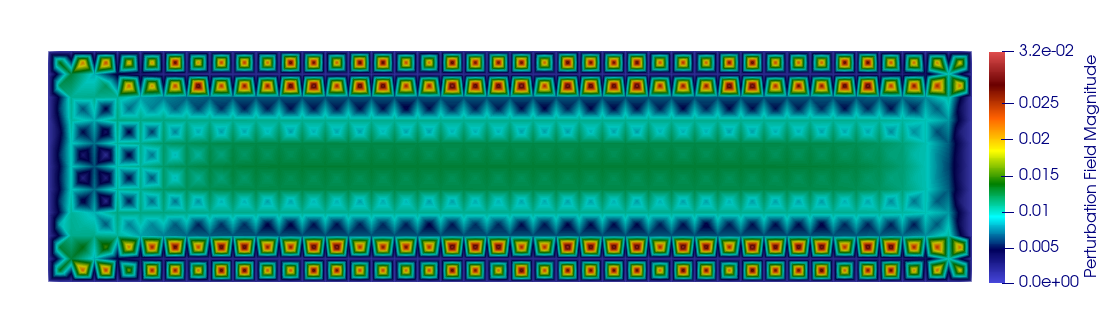}
\includegraphics[width=14.5cm,height=3.0cm]{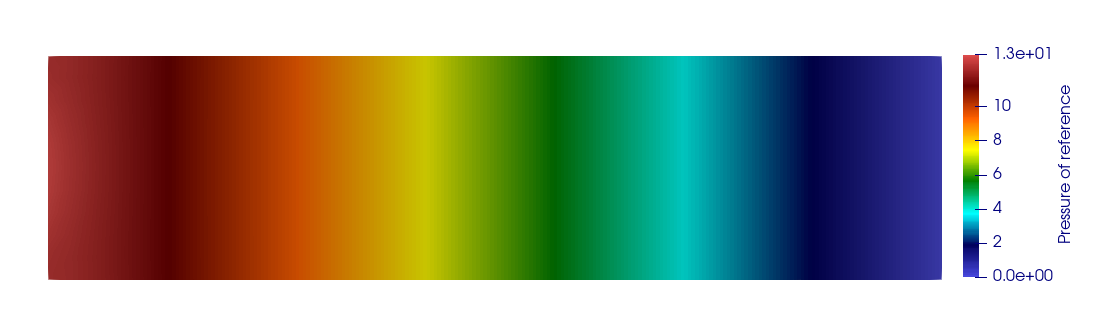}
\caption{Magnitude of the perturbation field $\ww_h$ (top) and pressure field $p_h$ (bottom).} 	\label{results}
\end{center}
\end{figure}

\begin{figure}[!ht]
\centering
\begin{tabular}{ccc}
$x=0$ & $y = 0.5$ & $y=1$\\
\includegraphics[width=5cm,height=3.0cm]{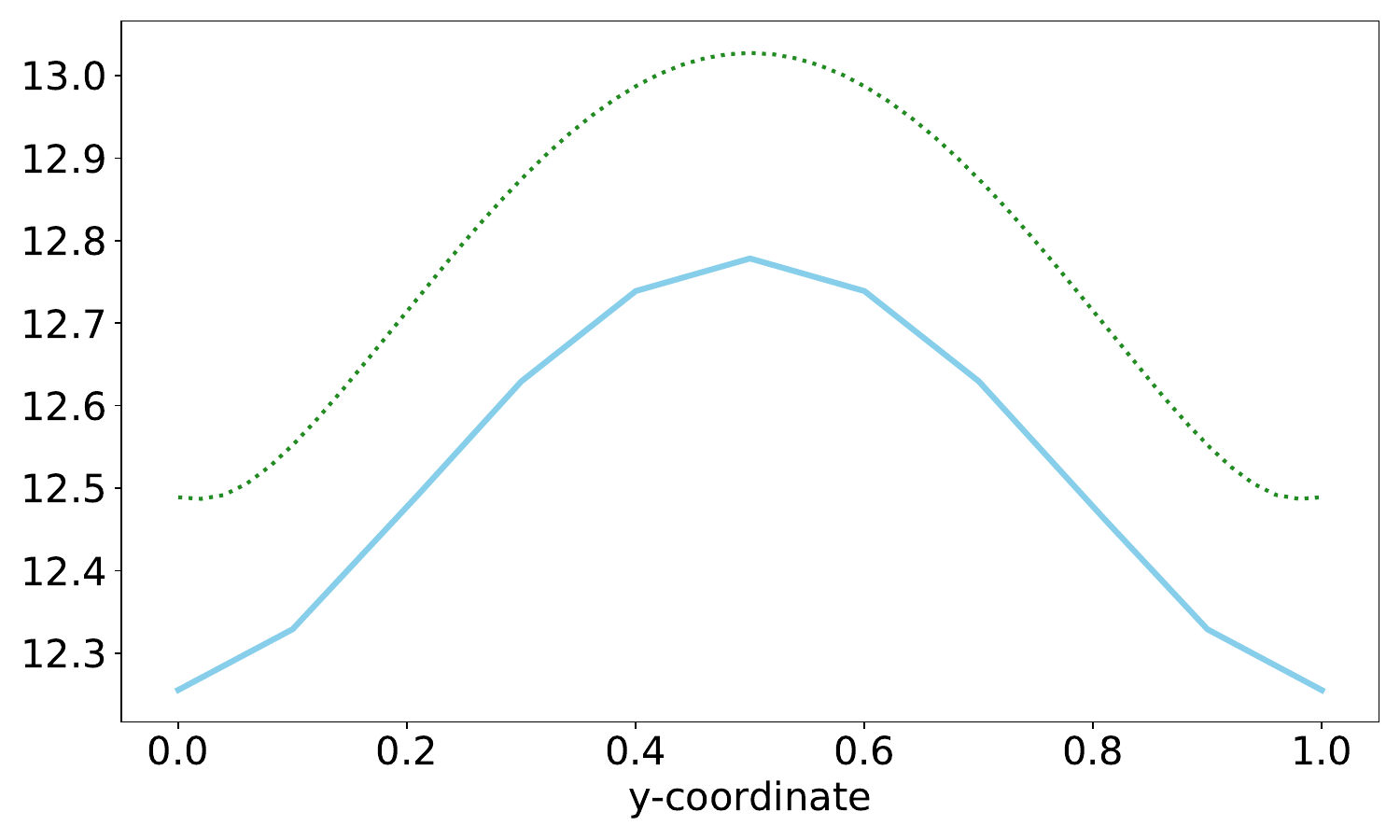}&
\includegraphics[width=5cm,height=3.0cm]{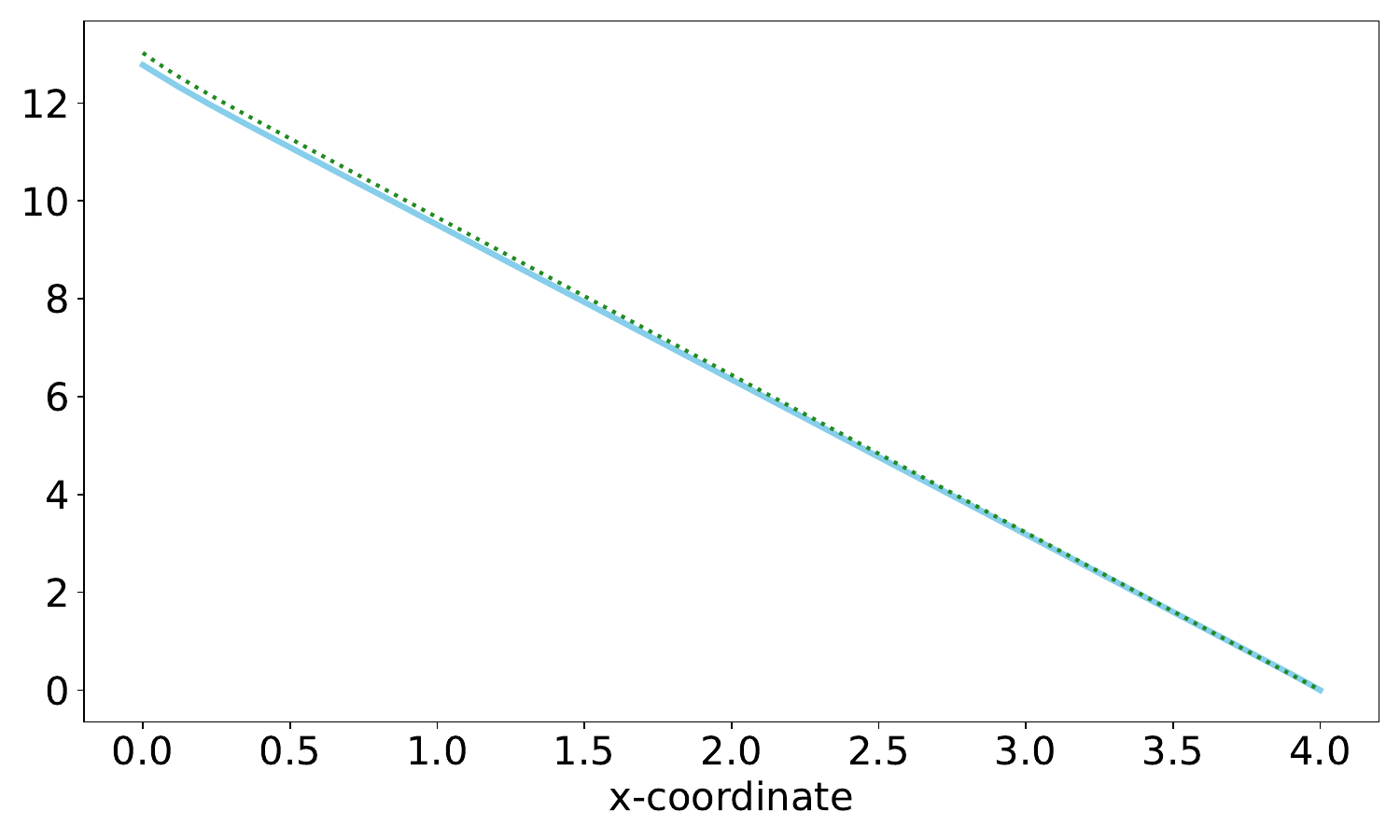}&
\includegraphics[width=5cm,height=3.0cm]{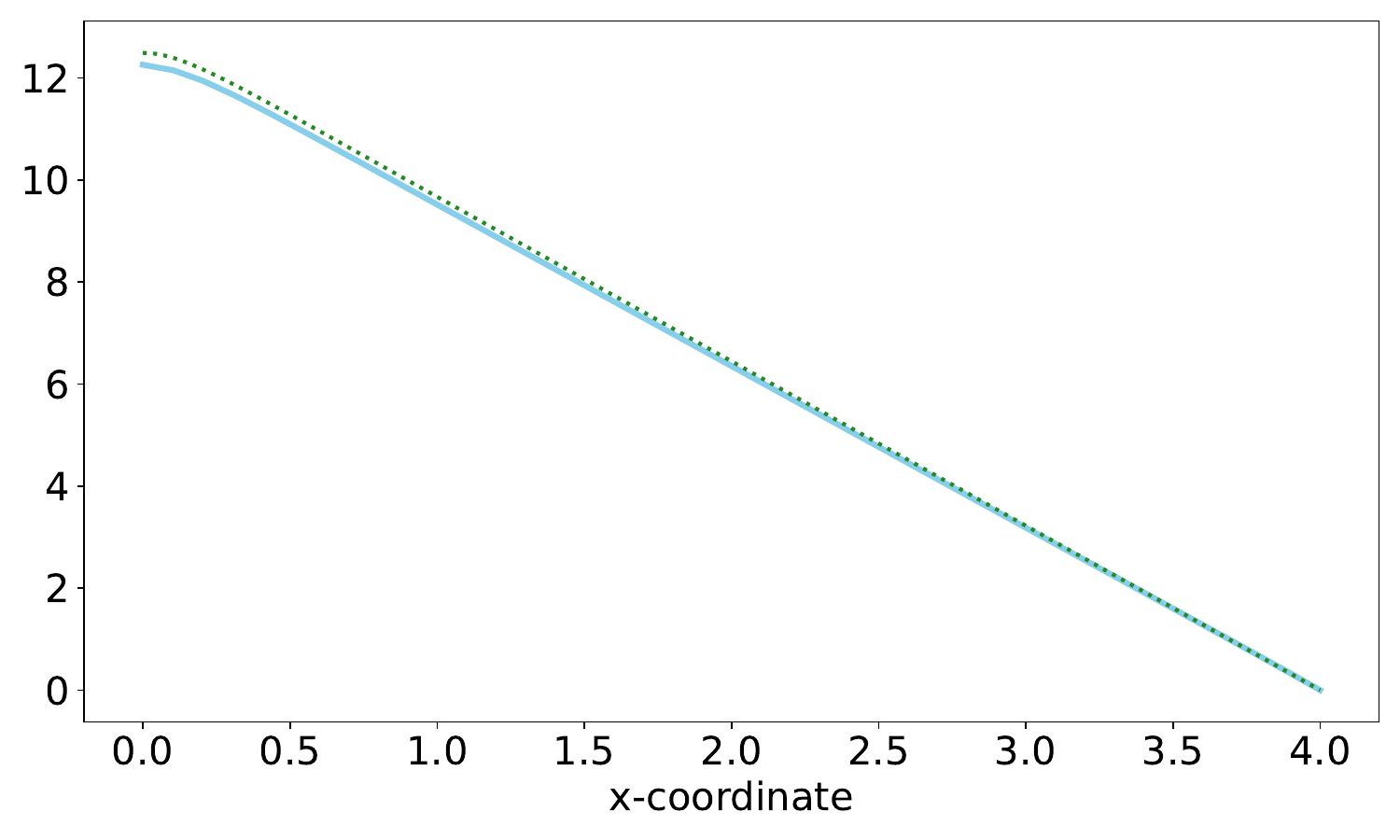}\\
\multicolumn{3}{c}{\includegraphics[width=4.5cm,height=0.5cm]{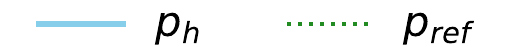}} 
\end{tabular}

\caption{Reference pressure field $p_{ref}$ versus computed pressure field $p_h$ on $x=0$ (left), $y=0.5$ (middle), and $y=1$ (right).} 	\label{plots_p}
\end{figure}

\begin{figure}[!ht]
\centering
\begin{tabular}{ccc}
$x=0$ & $y = 0.5$ & $y=1$\\
\includegraphics[width=5cm,height=3.0cm]{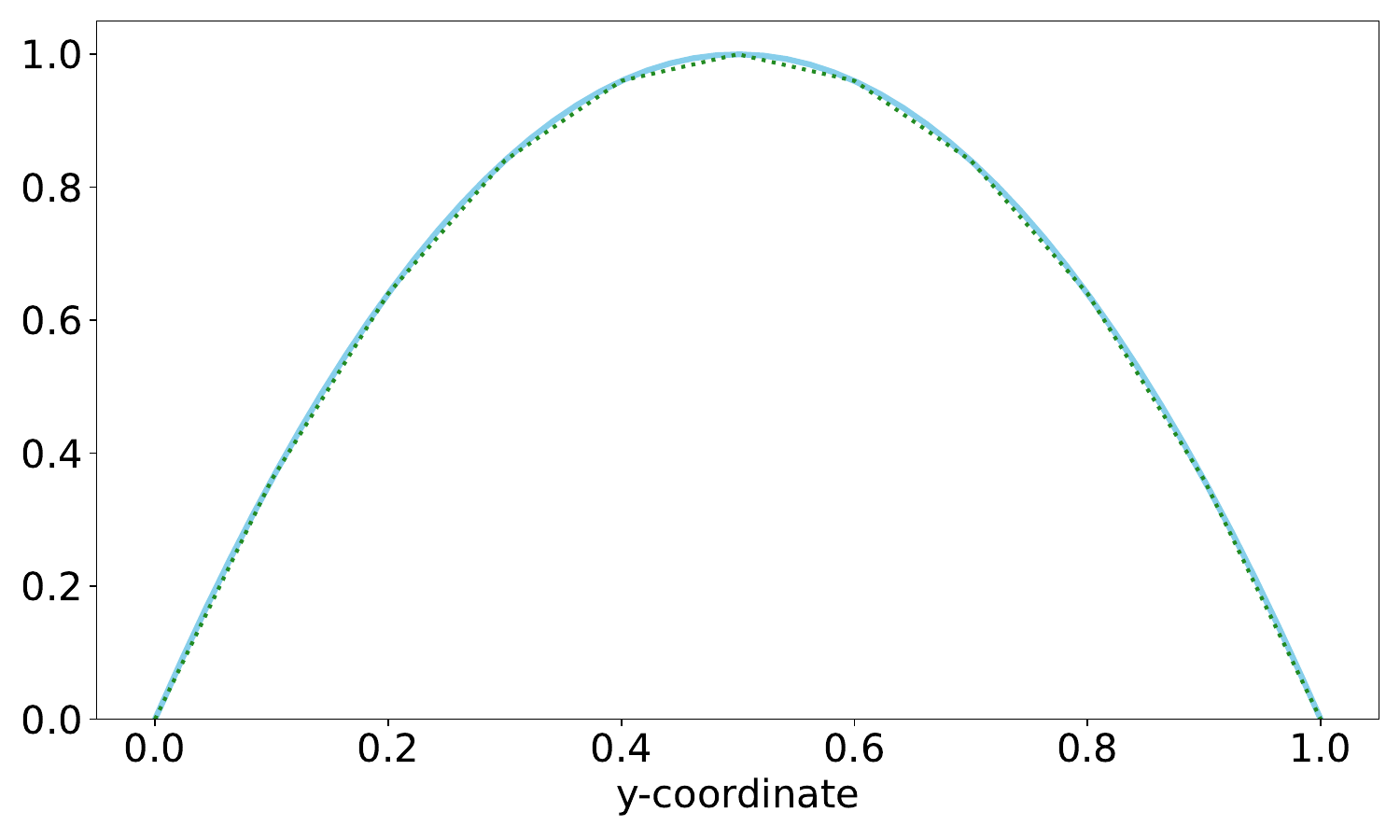}&
\includegraphics[width=5cm,height=3.0cm]{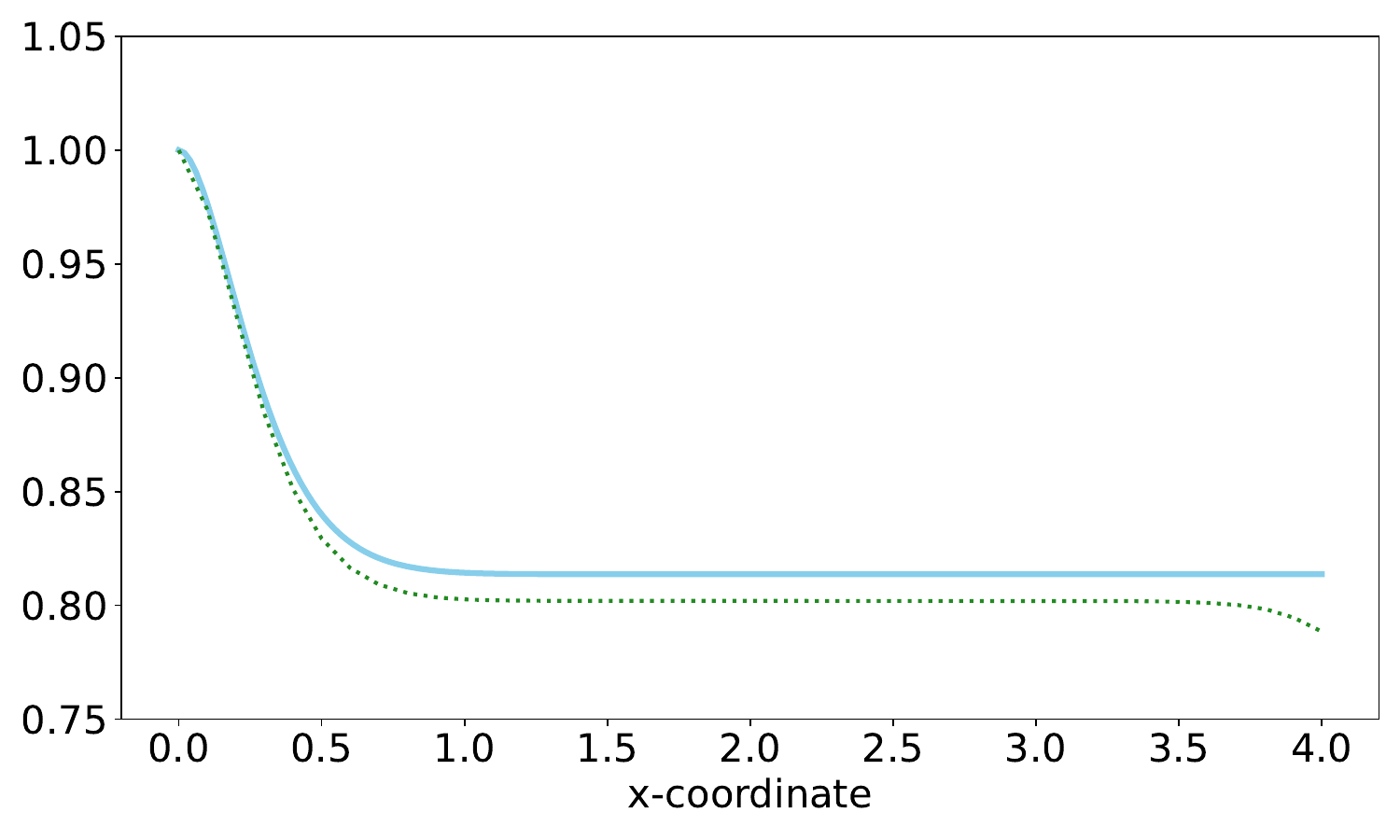}&
\includegraphics[width=5cm,height=3.0cm]{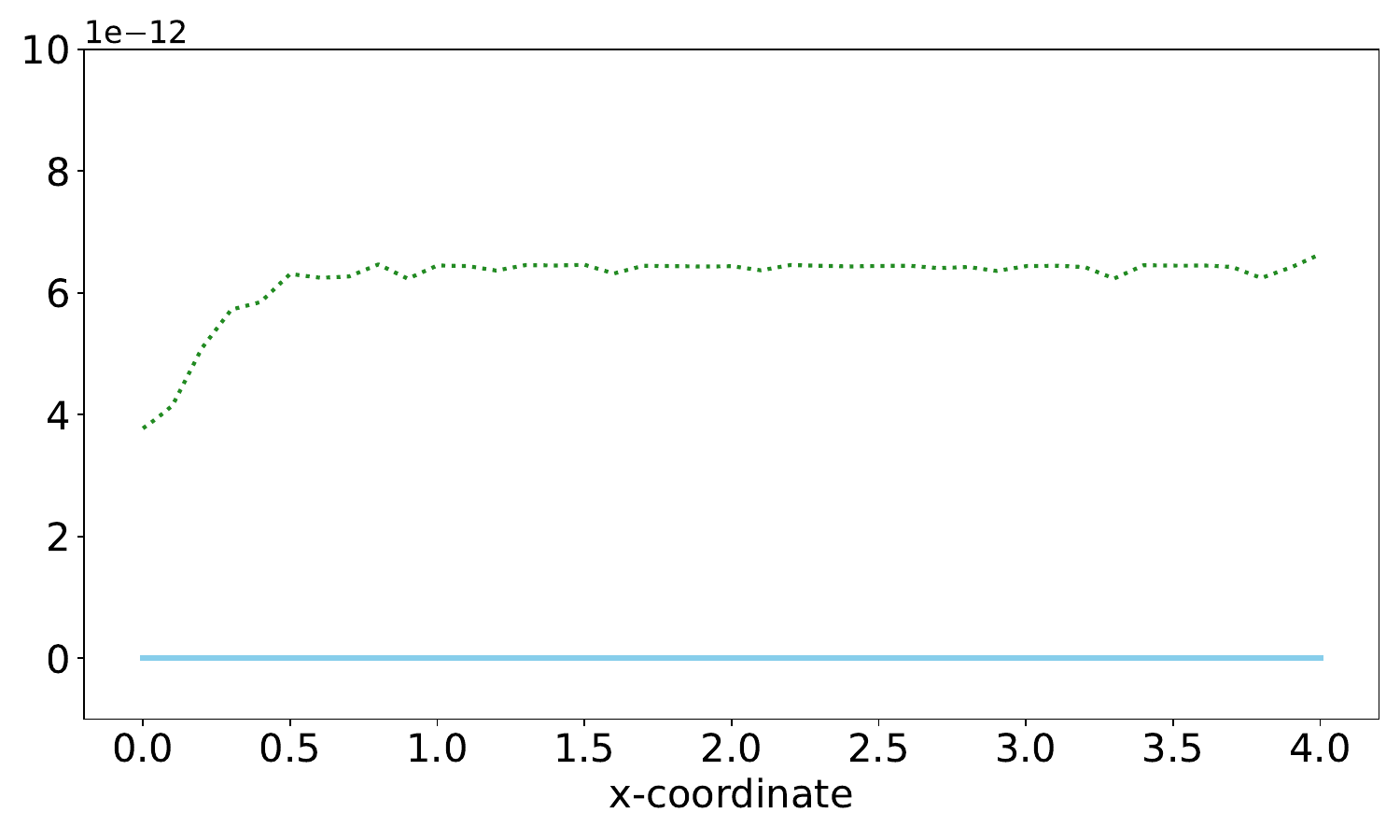}\\
\multicolumn{3}{c}{\includegraphics[width=5.5cm,height=0.5cm]{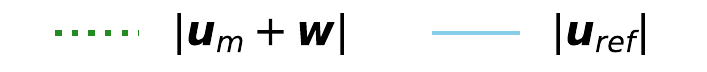}} 
\end{tabular}

\caption{{Magnitude of the} reference velocity field $|\uu_{ref}|$ versus {$|\uu_{m}+\ww_h|$} on $x=0$ (left), $y=0.5$ (middle), and $y=1$ (right).	} 	\label{plots_uref}
\end{figure}


\section{Conclusions}\label{conclusions}
In this paper we have presented and analyzed a stabilized finite element method for the modified (linearized) Navier-Stokes equation
arising from 4D MRI data.  The method was proven stable and optimally convergent, and the numerical experiments show that this
scheme can be used successfully to recover the pressure and correct the velocity data errors. Despite this, there are several open questions
that deserve further investigation.  For example, the extension to time-dependent problems is of practical interest.  For such problems,
the choice of stabilization will be of paramount importance. This, as the search for pressure-robust and/or divergence-free discretizations of (P) will be the topic
of future research. 

\section*{Acknowledgments}
The second authot was partially supported by the WIAS Young Scientist Grant. The  third author was partially 
supported by Direcci\'on de Investigaci\'on of the Universidad Cat\'olica de la Sant\'isima Concepci\'on through project DIREG 01/2025 and Proyecto Ingenier\'ia 2030 (ING222010004).

\paragraph{Note for the reader}
Please note that this manuscript represents a preprint only and has not been (or is in the process of being) peer-reviewed. A DOI link will be made available for this ArXiv preprint as soon as the peer-reviewed version is published online.

\bibliography{references.bib}

\end{document}